\documentclass[reqno,11pt]{article} 
\usepackage[left=2.5cm,right=2.5cm,top=3cm,bottom=3cm,a4paper]{geometry}
\usepackage{amsmath, amssymb}
\usepackage{amsthm, amscd} 
\usepackage[all,cmtip]{xy}
\usepackage{float}
\usepackage{caption}
\usepackage{color}

\DeclareRobustCommand\longtwoheadrightarrow
     {\relbar\joinrel\twoheadrightarrow}

\makeatletter
\DeclareRobustCommand{\rvdots}{%
  \vbox{
    \baselineskip4\p@\lineskiplimit\z@
    \kern-\p@
    \hbox{.}\hbox{.}\hbox{.}
  }}
\makeatother

\newcommand{\Mod}[1]{\ (\textup{mod}\ #1)}
\makeatletter
\def\moverlay{\mathpalette\mov@rlay}
\def\mov@rlay#1#2{\leavevmode\vtop{%
   \baselineskip\z@skip \lineskiplimit-\maxdimen
   \ialign{\hfil$\m@th#1##$\hfil\cr#2\crcr}}}
\newcommand{\charfusion}[3][\mathord]{
    #1{\ifx#1\mathop\vphantom{#2}\fi
        \mathpalette\mov@rlay{#2\cr#3}
      }
    \ifx#1\mathop\expandafter\displaylimits\fi}
\makeatother

\theoremstyle{plain} 
\newtheorem{theorem}{\indent\sc Theorem}[section]
\newtheorem{lemma}[theorem]{\indent\sc Lemma}
\newtheorem{corollary}[theorem]{\indent\sc Corollary}
\newtheorem{proposition}[theorem]{\indent\sc Proposition}

\theoremstyle{definition} 
\newtheorem{definition}[theorem]{\indent\sc Definition}

\newtheorem{remark}[theorem]{\indent\sc Remark}

\newtheorem{thmx}{Theorem}

\makeatletter
\def\address#1#2{\begingroup
\noindent\parbox[t]{7.8cm}{%
\small{\scshape\ignorespaces#1}\par\vskip1ex
\noindent\small{\itshape E-mail address}%
\/: #2\par\vskip4ex}\hfill%
\endgroup}%
\makeatother

\title{Arithmetic properties of orders in imaginary quadratic fields}
\author{
\textsc{Ho Yun Jung, Ja Kyung Koo, Dong Hwa Shin and Dong Sung Yoon} 
}
\date{} 
\begin{document}

\allowdisplaybreaks

\maketitle

 \footnote{
\thanks{
}
}

\begin{abstract}
Let $K$ be an imaginary quadratic field. For an order $\mathcal{O}$ in $K$
and a positive integer $N$, let $K_{\mathcal{O},\,N}$ be the ray class field
of $\mathcal{O}$ modulo $N\mathcal{O}$. We deal with various subjects related to $K_{\mathcal{O},\,N}$, mainly about
Galois representations attached to elliptic curves
with complex multiplication, form class groups and $L$-functions for orders.
\end{abstract}

\maketitle

\tableofcontents

\section {Introduction}

Let $K$ be an imaginary quadratic field.
In 1935, H. S\"{o}hngen (\cite{Sohngen})
first dealt with ray class fields for orders in $K$
which generalizes the usual ray class fields for the maximal order $\mathcal{O}_K$.
Later in \cite{Stevenhagen} Stevenhagen paid attention to the ray class fields for orders
in order to have explicit description of Shimura's reciprocity law (Proposition \ref{reciprocity}).
Furthermore, Cho (\cite{Cho}) constructed 
the fields in view of certain Diophantine equations by extending the ideas in \cite{Cox}.
\par
Let $\mathcal{O}$ be an order in $K$ of discriminant $D_\mathcal{O}$.
We denote by $I(\mathcal{O})$ the group of
proper fractional $\mathcal{O}$-ideals in the sense of \cite[$\S$7.A]{Cox}.
It is well known that every fractional $\mathcal{O}$-ideal is proper
if and only if it is invertible.
We say that a nonzero $\mathcal{O}$-ideal $\mathfrak{a}$ is prime to
a positive integer $\ell$ if $\mathfrak{a}+\ell\mathcal{O}=
\mathcal{O}$.
Let $P(\mathcal{O})$ be the subgroup of $I(\mathcal{O})$ consisting of
principal fractional $\mathcal{O}$-ideals.
For a positive integer $N$, we define the subgroups of $I(\mathcal{O})$ and $P(\mathcal{O})$ as
\begin{eqnarray}
I(\mathcal{O},\,N)&=&\langle\mathfrak{a}~|~
\mathfrak{a}~\textrm{is a nonzero proper $\mathcal{O}$-ideal
prime to $N$}\rangle,\nonumber\\
P_N(\mathcal{O})&=&\langle
\nu\mathcal{O}~|~\nu\in\mathcal{O}\setminus\{0\}~
\textrm{and}~\nu\equiv1\Mod{N\mathcal{O}}\rangle,\label{PNO}
\end{eqnarray}
respectively.
In particular, we have $I(\mathcal{O},\,1)=I(\mathcal{O})$
(cf. \cite[Exercise 7.7]{Cox}) and $P_1(\mathcal{O})=P(\mathcal{O})$
because the field of fractions of $\mathcal{O}$ is $K$.
Then the associated quotient group
\begin{equation*}
\mathcal{C}_N(\mathcal{O})
=I(\mathcal{O},\,N)/P_N(\mathcal{O})
\end{equation*}
is isomorphic to
a generalized ideal class group
modulo $\ell_\mathcal{O}N\mathcal{O}_K$, where
$\ell_\mathcal{O}=[\mathcal{O}_K:\mathcal{O}]$ is the conductor of $\mathcal{O}$
(cf. \cite[Theorem 3.1.8]{Schertz}).
And the existence theorem of class field theory
asserts that there
is a unique abelian extension $K_{\mathcal{O},\,N}$
for which the Artin map
for the modulus $\ell_\mathcal{O}N\mathcal{O}_K$ induces
an isomorphism of the generalized ideal class group
onto $\mathrm{Gal}(K_{\mathcal{O},\,N}/K)$
(cf. \cite[$\S$8]{Cox} or \cite[V.9]{Janusz}).
We call this extension field $K_{\mathcal{O},\,N}$ of $K$
the \textit{ray class field of $\mathcal{O}$ modulo $N\mathcal{O}$}.
In particular,
$K_{\mathcal{O},\,1}$ is just the ring class field $H_\mathcal{O}$
of order $\mathcal{O}$ and $K_{\mathcal{O}_K,\,N}$ is the
ray class field $K_{(N)}$ modulo $(N)=N\mathcal{O}_K$.
\par
When $D_\mathcal{O}\neq-3,\,-4$, we consider the elliptic curve $E_\mathcal{O}$
with $j$-invariant $j(E_\mathcal{O})=j(\mathcal{O})$ given by the Weierstrass equation
\begin{equation*}
E_\mathcal{O}~:~y^2=4x^3-A_\mathcal{O}x
-B_\mathcal{O}
\end{equation*}
as in $\S$\ref{sect3}. Let
$E_\mathcal{O}[N]$ be the group of $N$-torsion points of $E_\mathcal{O}$
and $\mathbb{Q}(E_\mathcal{O}[N])$ be the extension field of $\mathbb{Q}$
generated by the coordinates of points in $E_\mathcal{O}[N]$.
Then it can be shown that if $N\geq2$, then $\mathbb{Q}(E_\mathcal{O}[N])$ contains
$\mathbb{Q}(j(\mathcal{O}))$ (Lemma \ref{QXQj}).
And, in the first main theorem (Theorem \ref{maingenerate}) of this paper, we shall compare
the field $\mathbb{Q}(E_\mathcal{O}[N])$ with $K_{\mathcal{O},\,N}$ .

\begin{thmx}\label{theoremA}
Assume that $D_\mathcal{O}\neq-3,\,-4$.
\begin{enumerate}
\item[\textup{(i)}]
If $D\equiv0\Mod{4}$, then
$\mathbb{Q}(E_\mathcal{O}[2])$ is the maximal real subfield of $K_{\mathcal{O},\,2}$.
\item[\textup{(ii)}] If $D\equiv1\Mod{4}$, then
$\mathbb{Q}(E_\mathcal{O}[2])=K_{\mathcal{O},\,2}$.
\item[\textup{(iii)}] If $N\geq3$, then
$\mathbb{Q}(E_\mathcal{O}[N])$ is an extension field of $K_{\mathcal{O},\,N}$
of degree at most $2$.
\end{enumerate}
\end{thmx}

Here we put an emphasis on the fact that
$\mathbb{Q}(E_\mathcal{O}[N])$ is an extension field of $\mathbb{Q}$, not of $\mathbb{Q}(j(\mathcal{O}))$,
generated by the coordinates of $N$-torsion points.
This fact distinguishes Theorem \ref{theoremA} from prior works
of Bourdon-Clark-Pollack ([3, Lemma 8.4]) and Clark-Pollack ([5, Theorem 4.2]).
\par
Let $E$ be an arbitrary elliptic curve with complex multiplication by $\mathcal{O}$ defined over $\mathbb{Q}(j(\mathcal{O}))$.
Through purely algebraic arguments
Bourdon-Clark (\cite{B-C}) and Lozano-Robledo (\cite{Lozano-Robledo})
independently classified all possible images of ($p$-adic)
Galois representations attached to $E$.
In $\S$7, we shall revisit these results on Galois representations
by making use of modular-analytic approach.
Let
\begin{equation*}
\rho_{\mathcal{O},\,N}:\mathrm{Gal}(\mathbb{Q}(E_\mathcal{O}[N])/
\mathbb{Q}(j(E_\mathcal{O})))\rightarrow\mathrm{GL}_2(\mathbb{Z}/N\mathbb{Z})~(\simeq
\mathrm{Aut}(E_\mathcal{O}[N]))
\end{equation*}
be the Galois representation
attached to a specific model $E_\mathcal{O}$. 
Let
\begin{equation}\label{tauO}
\tau_\mathcal{O}=\left\{
\begin{array}{cl}
\displaystyle\frac{\sqrt{D_\mathcal{O}}}{2} & \textrm{if}~D_\mathcal{O}\equiv0\Mod{4},\vspace{0.1cm}\\
\displaystyle\frac{-1+\sqrt{D_\mathcal{O}}}{2} & \textrm{if}~D_\mathcal{O}\equiv1\Mod{4}
\end{array}\right.
\end{equation}
with $\min(\tau_\mathcal{O},\,\mathbb{Q})=x^2+b_\mathcal{O}x+c_\mathcal{O}$. 
Then we have $\mathcal{O}=[\tau_\mathcal{O},\,1]=\mathbb{Z}\tau_\mathcal{O}+\mathbb{Z}$ 
and $b_\mathcal{O},\,c_\mathcal{O}\in\mathbb{Z}$ (cf. \cite[Lemma 7.2]{Cox}).
Let
$W_{\mathcal{O},\,N}$ be 
the Cartan subgroup of $\mathrm{GL}_2(\mathbb{Z}/N\mathbb{Z})$ associated with
the $(\mathbb{Z}/N\mathbb{Z})$-algebra $\mathcal{O}/N\mathcal{O}$ with 
the ordered basis $\{\tau_\mathcal{O}+N\mathcal{O},\,1+N\mathcal{O}\}$.
In the second main theorem (Theorem \ref{Galoisrepresentation}), we shall describe the image of $\rho_{\mathcal{O},\,N}$
by using Stevenhagen's explicit version of the Shimura reciprocity law.

\begin{thmx}\label{theoremB}
Assume that $D_\mathcal{O}\neq-3,\,-4$ and $N\geq2$. Let
$\widehat{W}_{\mathcal{O},\,N}$ be the subgroup
of $\mathrm{GL}_2(\mathbb{Z}/N\mathbb{Z})$ defined by
\begin{equation*}
\widehat{W}_{\mathcal{O},\,N}=\left\langle
W_{\mathcal{O},\,N},\,\begin{bmatrix}1&~~b_\mathcal{O}\\0&-1\end{bmatrix}\right\rangle.
\end{equation*}
\begin{enumerate}
\item[\textup{(i)}] The image of the
Galois representation
$\rho_{\mathcal{O},\,N}$
is a subgroup of $\widehat{W}_{\mathcal{O},\,N}$
of index at most $2$.
\item[\textup{(ii)}] If $-1$ is a quadratic residue modulo $N$, then
the image of $\rho_{\mathcal{O},\,N}$ is exactly
$\widehat{W}_{\mathcal{O},\,N}$.
\end{enumerate}
\end{thmx}

Let $D$ be a negative integer such that $D\equiv0$ or $1\Mod{4}$.
Let $\mathcal{Q}(D)$ be the set of primitive
positive definite binary quadratic forms over $\mathbb{Z}$ of discriminant $D$
on which the modular group $\mathrm{SL}_2(\mathbb{Z})$
induces the proper equivalence $\sim$.
It was Gauss (\cite{Gauss}) who first introduced a composition law
on $\mathcal{C}(D)=\mathcal{Q}(D)/\sim$, which makes $\mathcal{C}(D)$
a group. When $D=D_\mathcal{O}$ and $N=1$, owing to Dirichlet and Dedekind we have the isomorphism
\begin{equation*}
\mathcal{C}(D_\mathcal{O})\stackrel{\sim}{\rightarrow}\mathcal{C}(\mathcal{O})=I(\mathcal{O})/P(\mathcal{O}),
\quad
[Q\mathrm{]}\mapsto[[\omega_Q,\,1]]=[\mathbb{Z}\omega_Q+\mathbb{Z}]]
\end{equation*}
where $\omega_Q$ is the zero of $Q(x,\,1)$ in
the complex upper half-plane (cf. \cite[Theorem 7.7]{Cox}).
Let
\begin{equation*}
\mathcal{Q}(D_\mathcal{O},\,N)=\{ax^2+bxy+cy^2\in\mathcal{Q}(D_\mathcal{O})~|~\gcd(a,\,N)=1\}
\end{equation*}
on which the congruence subgroup $\Gamma_1(N)$
defines the equivalence relation $\sim_{\Gamma_1(N)}$
($\S$\ref{sect5}).
We shall endow the set of equivalence classes
\begin{equation*}
\mathcal{C}_N(D_\mathcal{O})=\mathcal{Q}(D_\mathcal{O},\,N)/\sim_{\Gamma_1(N)}
\end{equation*}
with a binary operation which reduces to the Dirichlet composition on the classical form class group
$\mathcal{C}(D_\mathcal{O})$,
and prove the next third main theorem (Theorem \ref{formideal}).

\begin{thmx}\label{theoremC}
The form class group $\mathcal{C}_N(D_\mathcal{O})$ is isomorphic to
the ideal class group $\mathcal{C}_N(\mathcal{O})$.
\end{thmx}

Let $\mathcal{F}_N$ be the field of certain modular functions
stated in (\ref{FN}). Then $\mathcal{F}_N$ is a Galois extension of $\mathcal{F}_1$
whose Galois group is isomorphic to $\mathrm{GL}_2(\mathbb{Z}/N\mathbb{Z})/\langle-I_2\rangle$
(cf. \cite[Theorem 6.6]{Shimura}).
By using the theory of
Shimura's canonical models for modular curves (\cite[Chapter 6]{Shimura}),
Cho proved in \cite[Theorem 4]{Cho}  that
\begin{equation}\label{specialization}
K_{\mathcal{O},\,N}=K(f(\tau_\mathcal{O})~|~
f\in\mathcal{F}_N~\textrm{is finite at}~\tau_\mathcal{O}).
\end{equation}
Let $C\in\mathcal{C}_N(\mathcal{O})$ and $f\in\mathcal{F}_N$.
In Definition \ref{invariant}, we shall define the invariant $f(C)$
which is a generalization of the Siegel-Ramachandra invariant (cf. \cite{Ramachandra} and \cite{Siegel}) considered when $\mathcal{O}=\mathcal{O}_K$, $N\geq2$
and $f$ is the $12N$th power of the Siegel function $g_{\left[\begin{smallmatrix}0&\frac{1}{N}\end{smallmatrix}\right]}$ given in (\ref{infiniteproduct}).
In the fourth main theorem (Theorem \ref{transformation}), we shall justify that
$f(C)$ satisfies a natural transformation rule
under the Artin map $\sigma_{\mathcal{O},\,N}:\mathcal{C}_N(\mathcal{O})
\rightarrow\mathrm{Gal}(K_{\mathcal{O},\,N}/K)$.

\begin{thmx}\label{theoremD}
Let $C\in\mathcal{C}_N(\mathcal{O})$ and $f\in\mathcal{F}_N$. If
$f$ is finite at $\tau_\mathcal{O}$, then
$f(C)$ belongs to $K_{\mathcal{O},\,N}$
and satisfies
\begin{equation*}
f(C)^{\sigma_{\mathcal{O},\,N}(C')}=
f(CC')\quad(C'\in\mathcal{C}_N(\mathcal{O})).
\end{equation*}
\end{thmx}

Theorems \ref{theoremC} and \ref{theoremD} lead us to achieve the next fifth main theorem (Theorem \ref{CDGKK})
which describes the isomorphism
of $\mathcal{C}_N(D_\mathcal{O})$
onto $\mathrm{Gal}(K_{\mathcal{O},\,N}/K)$
explicitly.

\begin{thmx}\label{theoremE}
The map
\begin{eqnarray*}
\mathcal{C}_N(D_\mathcal{O}) & \rightarrow & \mathrm{Gal}(K_{\mathcal{O},\,N}/K)\\
\mathrm{[}Q]=[ax^2+bxy+cy^2] & \mapsto & \left(f(\tau_\mathcal{O})\mapsto
f^{\left[\begin{smallmatrix}
1 & -a'(b+b_\mathcal{O})/2\\
0&a'\end{smallmatrix}\right]}(-\overline{\omega}_Q)~|~f\in\mathcal{F}_N~\textrm{is finite at $\tau_\mathcal{O}$}\right)
\end{eqnarray*}
is a well-defined isomorphism, where $a'$ is an integer which holds $aa'\equiv1\Mod{N}$.
\end{thmx}

For a ray class character $\chi$ modulo $N\mathcal{O}_K$,
let $L(s,\,\chi)$ be the Weber $L$-function (cf. \cite[$\S$IV.4]{Janusz}).
It is well known by Kronecker's limit formulae that
$L(1,\,\chi)$
can be expressed by special values of the modular discriminant $\Delta$ when $N=1$, and by
Siegel-Ramachandra invariants when $N\geq2$
(cf. \cite[Chapters 20$\sim$22]{Lang87}, \cite{Schertz} or \cite{Siegel}).
On the other hand, rather than the value at $1$, Stark observed in \cite{Stark} that 
the derivative evaluated at zero $L'(0,\,\chi)$ would seem to be more easily applicable form.
In Definition \ref{defL}, we shall define the $L$-function $L_\mathcal{O}(s,\,\chi)$ for a given character $\chi$ of $\mathcal{C}_N(\mathcal{O})$ by
\begin{equation*}
L_\mathcal{O}(s,\,\chi)=\sum_{\mathfrak{a}}\frac{\chi([\mathfrak{a}])}{\mathrm{N}_\mathcal{O}(\mathfrak{a})^s}
\quad(s\in\mathbb{C},~\mathrm{Re}(s)>1)
\end{equation*}
where $\mathfrak{a}$ runs over all proper $\mathcal{O}$-ideals prime to $N$
and $\mathrm{N}_\mathcal{O}(\mathfrak{a})$ is the norm of $\mathfrak{a}$.
Then we shall show that $L'_\mathcal{O}(0,\,\chi)$
satisfies a similar formula to that of Stark (\cite[(2)]{Stark})
in the following last main theorem (Theorem \ref{derivative}).

\begin{thmx}
If $\chi$ is a character of
$\mathcal{C}_N(\mathcal{O})$, then we have
\begin{equation*}
L_\mathcal{O}'(0,\,\chi)=
\displaystyle-\frac{1}{\gamma_{\mathcal{O},\,N}6N}
\sum_{C\in\mathcal{C}_N(\mathcal{O})}\chi(C)
\ln\left|g_{\mathcal{O},\,N}(C)\right|.
\end{equation*}
\end{thmx}

Here, $\gamma_{\mathcal{O},\,N}=\left|\{\nu\in\mathcal{O}^*~|~\nu\equiv1\Mod{N\mathcal{O}}\}\right|$ and
\begin{equation*}
g_{\mathcal{O},\,N}(C)=\left\{\begin{array}{ll}
(2\pi)^{12}\,\mathrm{N}_\mathcal{O}(\mathfrak{c})^{-6}|\Delta(\mathfrak{c})|^{-1} & \textrm{if}~N=1,\vspace{0.1cm}\\
g_{\left[\begin{smallmatrix}0&\frac{1}{N}\end{smallmatrix}\right]}^{12N}(C) & \textrm{if}~N\geq2,
\end{array}\right.
\end{equation*}
where $\mathfrak{c}$ is a proper $\mathcal{O}$-ideal in the class $C$.
\par
In the last section, when $K=\mathbb{Q}(\sqrt{-2})$,
$\mathcal{O}=[5\sqrt{-2},\,1]$ and $N=3$, we shall present an example
of the group $\mathcal{C}_N(D_\mathcal{O})=
\mathcal{C}_3(-200)$ and its application further in order to find the minimal
polynomial of the invariant $g_{\mathcal{O},\,3}([\mathcal{O}])$ over $K$
which generates the field $K_{\mathcal{O},\,3}$. 
It turns out that this invariant is in fact a unit as algebraic integer. 

\section {The generalized ideal class group
isomorphic to $\mathcal{C}_N(\mathcal{O})$}

Throughout this paper, we let $K$ be an imaginary quadratic field, $\mathcal{O}$ be an order in $K$
and $N$ be a positive integer.
We denote the conductor and the discriminant of $\mathcal{O}$
by $\ell_\mathcal{O}$ and $D_\mathcal{O}$, respectively.
As far as we know, Schertz's book \cite{Schertz} is the only reference for the fact that
the ray class group $\mathcal{C}_N(\mathcal{O})$ modulo $N\mathcal{O}$ is isomorphic
to a generalized ideal class group modulo $\ell_\mathcal{O}N\mathcal{O}_K$.
In this section, we shall explain this fact in modern terms
by adopting the ideas of \cite[$\S$7]{Cox} for $N=1$,
which will be definitely helpful to develop our main theorems.
\par
For a positive integer $\ell$, we denote by
\begin{eqnarray*}
\mathcal{M}(\mathcal{O},\,\ell)&=&
\textrm{the monoid of nonzero proper $\mathcal{O}$-ideals prime to $\ell$},\\
I(\mathcal{O},\,\ell)&=&
\textrm{the subgroup of $I(\mathcal{O})$
generated by $\mathcal{M}(\mathcal{O},\,\ell)$},\\
P_N(\mathcal{O},\,\ell)&=&\langle
\nu\mathcal{O}~|~\nu\in\mathcal{O}\setminus\{0\},~
\textrm{$\nu\mathcal{O}$ is prime to $\ell$}~
\textrm{and}~\nu\equiv1\Mod{N\mathcal{O}}\rangle,\\
\mathcal{C}_N(\mathcal{O},\,\ell)&=&I(\mathcal{O},\,\ell)/P_N(\mathcal{O},\,\ell).
\end{eqnarray*}
Recall that $I(\mathcal{O},\,1)=I(\mathcal{O})$ and $P_1(\mathcal{O},\,1)=P_1(\mathcal{O})=P(\mathcal{O})$, and so
$\mathcal{C}_1(\mathcal{O},\,1)=\mathcal{C}(\mathcal{O})$.
Furthermore, since
$P_N(\mathcal{O},\,N)=P_N(\mathcal{O})$, we have
$\mathcal{C}_N(\mathcal{O},\,N)=\mathcal{C}_N(\mathcal{O})$.
And, for simplicity, we just write
$\mathcal{M}(\mathcal{O})$,
$P(\mathcal{O},\,\ell)$, $\mathcal{C}(\mathcal{O},\,\ell)$ for
$\mathcal{M}(\mathcal{O},\,1)$, $P_1(\mathcal{O},\,\ell)$,  $\mathcal{C}_1(\mathcal{O},\,\ell)$, respectively.
For $\mathfrak{a}\in\mathcal{M}(\mathcal{O})$, we denote its norm by $N_\mathcal{O}(\mathfrak{a})$,
namely, $\mathrm{N}_\mathcal{O}(\mathfrak{a})=|\mathcal{O}/\mathfrak{a}|$.

\begin{lemma}\label{normbasic}
Let $\nu\in
\mathcal{O}\setminus\{0\}$ and $\mathfrak{a},\,\mathfrak{b}\in\mathcal{M}(\mathcal{O})$. We get
\begin{enumerate}
\item[\textup{(i)}] $\mathrm{N}_\mathcal{O}(\nu\mathcal{O})=
\mathrm{N}_{K/\mathbb{Q}}(\nu)$.
\item[\textup{(ii)}] $\mathrm{N}_\mathcal{O}(\mathfrak{a}\mathfrak{b})=
\mathrm{N}_\mathcal{O}(\mathfrak{a})
\mathrm{N}_\mathcal{O}(\mathfrak{b})$.
\item[\textup{(iii)}] $\mathfrak{a}\overline{\mathfrak{a}}=\mathrm{N}_\mathcal{O}(\mathfrak{a})\mathcal{O}$,
where $\overline{\,\cdot\,}$ means the complex conjugation.
\end{enumerate}
\end{lemma}
\begin{proof}
See \cite[Lemma 7.14]{Cox}.
\end{proof}

\begin{lemma}\label{primelemma}
If $\mathfrak{a}$ is a nonzero $\mathcal{O}$-ideal, then
\begin{equation*}
\textrm{$\mathfrak{a}$ is prime to $\ell$}
~\Longleftrightarrow~
\textrm{$\mathrm{N}_\mathcal{O}(\mathfrak{a})$ is relatively prime to $\ell$}.
\end{equation*}
\end{lemma}
\begin{proof}
The proof is the same as that of \cite[Lemma 7.18 (i)]{Cox}
except replacing $f$ by $\ell$.
\end{proof}

\begin{remark}\label{proper}
Observe that every nonzero $\mathcal{O}$-ideal prime to $\ell_\mathcal{O}$
is proper (\cite[Lemma 7.18 (ii)]{Cox}).
\end{remark}

\begin{lemma}\label{IONMOMON}
We have $I(\mathcal{O},\,\ell)\cap\mathcal{M}(\mathcal{O})=\mathcal{M}(\mathcal{O},\,\ell)$ for a positive integer
$\ell$.
\end{lemma}
\begin{proof}
The inclusion $I(\mathcal{O},\,\ell)\cap\mathcal{M}(\mathcal{O})\supseteq\mathcal{M}(\mathcal{O},\,N)$ is obvious.
\par
Now, let $\mathfrak{a}\in I(\mathcal{O},\,\ell)\cap\mathcal{M}(\mathcal{O})$.
Since $\mathfrak{a}\in I(\mathcal{O},\,\ell)$, we attain $\mathfrak{a}=\mathfrak{b}\mathfrak{c}^{-1}$
for some $\mathfrak{b},\,\mathfrak{c}\in\mathcal{M}(\mathcal{O},\,\ell)$.
And we get from the fact $\mathfrak{a}\mathfrak{c}=\mathfrak{b}$ that
\begin{equation}\label{NaNcNb}
\mathrm{N}_\mathcal{O}(\mathfrak{a})\mathrm{N}_\mathcal{O}(\mathfrak{c})=
\mathrm{N}_\mathcal{O}(\mathfrak{b})
\end{equation}
by Lemma \ref{normbasic} (ii). Since $\mathfrak{b}$ is prime to $\ell$,
we deduce by Lemma \ref{primelemma} that
$\gcd(\mathrm{N}_\mathcal{O}(\mathfrak{b}),\,\ell)=1$.
Thus we obtain by (\ref{NaNcNb}) that $\gcd(\mathrm{N}_\mathcal{O}(\mathfrak{a}),\,\ell)=1$,
which implies again by Lemma \ref{primelemma} that $\mathfrak{a}$ is prime to $\ell$.
Therefore we achieve the converse inclusion  $I(\mathcal{O},\,\ell)\cap\mathcal{M}(\mathcal{O})\subseteq\mathcal{M}(\mathcal{O},\,\ell)$.
\end{proof}

\begin{lemma}\label{normpreserving}
Consider the case where $N=1$.
\begin{enumerate}
\item[\textup{(i)}] If $\mathfrak{a}\in\mathcal{M}(\mathcal{O},\,\ell_\mathcal{O})$, then
$\mathfrak{a}\mathcal{O}_K\in\mathcal{M}(\mathcal{O}_K,\,\ell_\mathcal{O})$
and $\mathrm{N}_{\mathcal{O}}(\mathfrak{a})=
\mathrm{N}_{\mathcal{O}_K}(\mathfrak{a}\mathcal{O}_K)$.
\item[\textup{(ii)}] If $\mathfrak{b}\in\mathcal{M}(\mathcal{O}_K,\,\ell_\mathcal{O})$, then
$\mathfrak{b}\cap\mathcal{O}\in\mathcal{M}(\mathcal{O},\,\ell_\mathcal{O})$
and $\mathrm{N}_{\mathcal{O}_K}(\mathfrak{b})=
\mathrm{N}_{\mathcal{O}}(\mathfrak{b}\cap\mathcal{O})$.
\item[\textup{(iii)}]
The map
\begin{equation*}
\mathcal{M}(\mathcal{O},\,\ell_\mathcal{O})\rightarrow\mathcal{M}(\mathcal{O}_K,\,\ell_\mathcal{O}),
\quad\mathfrak{a}\mapsto\mathfrak{a}\mathcal{O}_K
\end{equation*}
induces an isomorphism $I(\mathcal{O},\,\ell_\mathcal{O})\stackrel{\sim}{\rightarrow}I(\mathcal{O}_K,\,\ell_\mathcal{O})$.
\end{enumerate}
\end{lemma}
\begin{proof}
See \cite[Proposition 7.20]{Cox}.
\end{proof}

\begin{lemma}\label{monoid}
The map
\begin{eqnarray*}
\phi~:~\mathcal{M}(\mathcal{O},\,\ell_\mathcal{O}N)&\rightarrow&
\mathcal{M}(\mathcal{O}_K,\,\ell_\mathcal{O}N)\\
\mathfrak{a} & \mapsto & \mathfrak{a}\mathcal{O}_K
\end{eqnarray*}
is well defined, and uniquely gives an isomorphism $I(\mathcal{O},\,\ell_\mathcal{O}N)
\stackrel{\sim}{\rightarrow} I(\mathcal{O}_K,\,\ell_\mathcal{O}N)$.
\end{lemma}
\begin{proof}
Let $\mathfrak{a}\in\mathcal{M}(\mathcal{O},\,\ell_\mathcal{O}N)$ ($\subseteq\mathcal{M}(\mathcal{O},\,\ell_\mathcal{O})$). We get
by Lemmas \ref{primelemma} and \ref{normpreserving} (i) that
$\mathrm{N}_\mathcal{O}
(\mathfrak{a})=\mathrm{N}_{\mathcal{O}_K}(\mathfrak{a}\mathcal{O}_K)$ is relatively prime to $\ell_\mathcal{O}N$.
So $\mathfrak{a}\mathcal{O}_K$ belongs to $\mathcal{M}(\mathcal{O}_K,\,\ell_\mathcal{O}N)$
again by Lemma \ref{primelemma},
which shows that the map $\phi$ is well defined.
Note further that
\begin{equation}\label{preserve}
\phi(\mathfrak{a}\mathfrak{a}')=
(\mathfrak{a}\mathfrak{a}')\mathcal{O}_K=
(\mathfrak{a}\mathcal{O}_K)
(\mathfrak{a}'\mathcal{O}_K)=\phi(\mathfrak{a})\phi(\mathfrak{a}')
\quad(\mathfrak{a},\,\mathfrak{a}'
\in\mathcal{M}(\mathcal{O},\,\ell_\mathcal{O}N)).
\end{equation}
\par
Let $\mathfrak{b}\in\mathcal{M}(\mathcal{O}_K,\,\ell_\mathcal{O}N)$ ($\subseteq\mathcal{M}(\mathcal{O}_K,\,\ell_\mathcal{O})$).
It follows from Lemmas \ref{primelemma} and \ref{normpreserving} (ii) that
$\mathrm{N}_{\mathcal{O}_K}(\mathfrak{b})=\mathrm{N}_\mathcal{O}(\mathfrak{b}
\cap\mathcal{O})$
is relatively prime to $\ell_\mathcal{O}N$, which implies that $\mathfrak{b}\cap\mathcal{O}$
belongs to $\mathcal{M}(\mathcal{O},\,\ell_\mathcal{O}N)$.
Thus we obtain
the well-defined map
\begin{eqnarray*}
\psi~:~\mathcal{M}(\mathcal{O}_K,\,\ell_\mathcal{O}N)&\rightarrow&
\mathcal{M}(\mathcal{O},\,\ell_\mathcal{O}N)\\
\mathfrak{b} & \mapsto & \mathfrak{b}\cap\mathcal{O}.
\end{eqnarray*}
\par
One can then readily show that $\phi$ and $\psi$ are
inverses of each other which preserve multiplication by (\ref{preserve}).
(The proof is the same as that of \cite[(7.21)]{Cox}
except replacing $f$ by $\ell_\mathcal{O}N$.)
Now, we define a map $\widetilde{\phi}:I(\mathcal{O},\,\ell_\mathcal{O}N)\rightarrow I(\mathcal{O}_K,\,\ell_\mathcal{O}N)$
by
\begin{equation*}
\widetilde{\phi}(\mathfrak{a}\mathfrak{c}^{-1})=\phi(\mathfrak{a})\phi(\mathfrak{c})^{-1}
\quad(\mathfrak{a},\,\mathfrak{c}\in\mathcal{M}(\mathcal{O},\,\ell_\mathcal{O}N)).
\end{equation*}
Then $\widetilde{\phi}$ is a well-defined isomorphism
by (\ref{preserve}) and the bijectivity of $\phi$.
\end{proof}

\begin{lemma}\label{equivcondition}
If $\nu\in\mathcal{O}_K\setminus\{0\}$, then
\begin{eqnarray*}
&&\nu\in\mathcal{O},~\nu\mathcal{O}~
\textrm{is prime to $\ell_\mathcal{O}N$}~
\textrm{and}~\nu\equiv1\Mod{N\mathcal{O}}\\
&\Longleftrightarrow&
\nu\mathcal{O}_K~
\textrm{is prime to $\ell_\mathcal{O}N$ and}~\nu\equiv a\Mod{\ell_\mathcal{O}N\mathcal{O}_K}~
\textrm{for some $a\in\mathbb{Z}$ such that}~a\equiv1\Mod{N}.
\end{eqnarray*}
\end{lemma}
\begin{proof}
Assume that $\nu\in\mathcal{O}$, $\nu\mathcal{O}$ is prime to $\ell_\mathcal{O}N$
and $\nu\equiv1\Mod{N\mathcal{O}}$.
By Lemmas \ref{primelemma} and \ref{normpreserving} (i), $(\nu\mathcal{O})\mathcal{O}_K=\nu\mathcal{O}_K$ is also prime to $\ell_\mathcal{O}N$. Since $\nu\equiv1\Mod{N\mathcal{O}}$ and $N\mathcal{O}=
[\ell_\mathcal{O}N\tau_{\mathcal{O}_K},\,N]$, we have
\begin{equation*}
\nu=s\ell_\mathcal{O}N\tau_{\mathcal{O}_K}+tN+1\quad\textrm{for some}~s,\,t\in\mathbb{Z}.
\end{equation*}
It follows from the fact
$\ell_\mathcal{O}N\mathcal{O}_K=[\ell_\mathcal{O}N\tau_{\mathcal{O}_K},\,\ell_\mathcal{O}N]$  that
\begin{equation*}
\nu\equiv a\Mod{\ell_\mathcal{O}N\mathcal{O}_K}\quad
\textrm{with $a=tN+1$ satisfying $a\equiv1\Mod{N}$}.
\end{equation*}
\par
Conversely, assume that $\nu\mathcal{O}_K$ is prime to $\ell_\mathcal{O}N$
and $\nu\equiv a\Mod{\ell_\mathcal{O}N\mathcal{O}_K}$ for some $a\in\mathbb{Z}$
such that $a\equiv1\Mod{N}$. Since
\begin{equation*}
\nu-a\in \ell_\mathcal{O}N\mathcal{O}_K=N(\ell_\mathcal{O}\mathcal{O}_K) \subseteq
N\mathcal{O}\subseteq
\mathcal{O}
\quad\textrm{and}\quad a\in\mathbb{Z}\subset\mathcal{O},
\end{equation*}
we attain that
\begin{equation*}
\nu\in\mathcal{O}\quad
\textrm{and}\quad \nu\equiv a\equiv1\Mod{N\mathcal{O}}.
\end{equation*}
Moreover, $\nu\mathcal{O}$ is prime to $\ell_\mathcal{O}N$ by Lemma \ref{monoid}.
\end{proof}

\begin{proposition}\label{generalized}
Let
$P_{\mathbb{Z},\,N}(\mathcal{O}_K,\,\ell_\mathcal{O}N)$
be the subgroup of $I(\mathcal{O}_K,\,\ell_\mathcal{O}N)$ given by
\begin{equation*}
P_{\mathbb{Z},\,N}(\mathcal{O}_K,\,\ell_\mathcal{O}N)=
\left\langle
\nu\mathcal{O}_K~\Bigg|~
\begin{array}{l}
\nu\in\mathcal{O}_K\setminus\{0\},~
\textrm{$\nu\mathcal{O}_K$ is prime to $\ell_\mathcal{O}N$},\\
\nu\equiv a\Mod{\ell_\mathcal{O}N\mathcal{O}_K}~
\textrm{for some $a\in\mathbb{Z}$ such that $a\equiv1\Mod{N}$}
\end{array}
\right\rangle.
\end{equation*}
Then we get a natural isomorphism
\begin{equation*}
\mathcal{C}_N(\mathcal{O},\,\ell_\mathcal{O}N)
\stackrel{\sim}{\rightarrow}
I(\mathcal{O}_K,\,\ell_\mathcal{O}N)/P_{\mathbb{Z},\,N}(\mathcal{O}_K,\,\ell_\mathcal{O}N).
\end{equation*}
\end{proposition}
\begin{proof}
If $\widetilde{\phi}:I(\mathcal{O},\,\ell_\mathcal{O}N)
\stackrel{\sim}{\rightarrow}I(\mathcal{O}_K,\,\ell_\mathcal{O}N)$
is the isomorphism described in Lemma \ref{monoid},
then we achieve by Lemma \ref{equivcondition} that
\begin{equation*}
\widetilde{\phi}(P_N(\mathcal{O},\,\ell_\mathcal{O}N))=
P_{\mathbb{Z},\,N}(\mathcal{O}_K,\,\ell_\mathcal{O}N).
\end{equation*}
Therefore we establish the isomorphism
\begin{eqnarray*}
\mathcal{C}_N(\mathcal{O},\,\ell_\mathcal{O}N)&\stackrel{\sim}{\rightarrow}&
I(\mathcal{O}_K,\,\ell_\mathcal{O}N)/P_{\mathbb{Z},\,N}(\mathcal{O}_K,\,\ell_\mathcal{O}N)\\
\mathrm{[}\mathfrak{a}\mathfrak{b}^{-1}]&\mapsto&
[(\mathfrak{a}\mathcal{O}_K)
(\mathfrak{b}\mathcal{O}_K)^{-1}]
\end{eqnarray*}
where $\mathfrak{a}$ and $\mathfrak{b}$
are nonzero $\mathcal{O}$-ideals
prime to $\ell_\mathcal{O}N$.
\end{proof}

\begin{remark}
Since $P_{\mathbb{Z},\,N}(\mathcal{O}_K,\,\ell_\mathcal{O}N)$ is
a congruence subgroup for $\ell_\mathcal{O}N\mathcal{O}_K$,
the generalized ideal class group
$I(\mathcal{O}_K,\,\ell_\mathcal{O}N)/P_{\mathbb{Z},\,N}(\mathcal{O}_K,\,\ell_\mathcal{O}N)$
has finite order (\cite[Corollary 1.6 in Chapter IV]{Janusz}).
It then follows from Proposition
\ref{generalized} that $\mathcal{C}_N(\mathcal{O},\,\ell_\mathcal{O}N)$ has finite order as well.
\end{remark}

\begin{lemma}\label{classcontain}
Let $\ell$ be a positive integer.
Every class in $\mathcal{C}(\mathcal{O})$ contains
a proper $\mathcal{O}$-ideal whose norm is relatively prime to $\ell$.
\end{lemma}
\begin{proof}
See \cite[Corollary 7.17]{Cox}.
\end{proof}

\begin{lemma}\label{COlCO1}
If $\ell$ is a positive integer, then the inclusion $I(\mathcal{O},\,\ell)
\hookrightarrow I(\mathcal{O})$
induces an isomorphism $\mathcal{C}(\mathcal{O},\,\ell)
\stackrel{\sim}{\rightarrow}\mathcal{C}(\mathcal{O})$.
\end{lemma}
\begin{proof}
Let $\rho:I(\mathcal{O},\,\ell)\rightarrow \mathcal{C}(\mathcal{O})=I(\mathcal{O})/
P(\mathcal{O})$ be the natural homomorphism. Then the surjectivity of $\rho$
follows from Lemmas \ref{primelemma} and \ref{classcontain}.
And the proof of $\mathrm{ker}(\rho)=P(\mathcal{O},\,\ell)$
is exactly the same as that of \cite[Proposition 7.19]{Cox}
except replacing $f$ by $\ell$. Thus we conclude that
$\mathcal{C}(\mathcal{O},\,\ell)\simeq\mathcal{C}(\mathcal{O})$.
\end{proof}

\begin{lemma}\label{PMNPN}
The inclusion $P(\mathcal{O},\,\ell_\mathcal{O}N)\hookrightarrow P(\mathcal{O},\,N)$
gives an isomorphism
\begin{equation*}
P(\mathcal{O},\,\ell_\mathcal{O}N)/P_N(\mathcal{O},\,\ell_\mathcal{O}N)
\stackrel{\sim}{\rightarrow}P(\mathcal{O},\,N)/P_N(\mathcal{O}).
\end{equation*}
\end{lemma}
\begin{proof}
See \cite[Lemma 15.17 and Exercise 15.10]{Cox}.
\end{proof}

\begin{proposition}\label{MNN}
The inclusion $I(\mathcal{O},\,\ell_\mathcal{O}N)\hookrightarrow
I(\mathcal{O},\,N)$ derives an isomorphism
\begin{equation*}
\mathcal{C}_N(\mathcal{O},\,\ell_\mathcal{O}N)
\stackrel{\sim}{\rightarrow}
\mathcal{C}_N(\mathcal{O}).
\end{equation*}
\end{proposition}
\begin{proof}
Since $P_N(\mathcal{O},\,\ell_\mathcal{O}N)\subseteq
P_N(\mathcal{O})$, the inclusion $I(\mathcal{O},\,\ell_\mathcal{O}N)\hookrightarrow I(\mathcal{O},\,N)$
renders a homomorphism
\begin{equation}\label{C1}
\mathcal{C}_N(\mathcal{O},\,\ell_\mathcal{O}N)\rightarrow\mathcal{C}_N(\mathcal{O}).
\end{equation}
\par
Let $\mathfrak{a}_1,\,\mathfrak{a}_2,\,\ldots,\,\mathfrak{a}_r\in I(\mathcal{O},\,\ell_\mathcal{O}N)$
be representatives of classes in $\mathcal{C}(\mathcal{O},\,\ell_\mathcal{O}N)$.
Since we have the natural isomorphisms
\begin{equation*}
\mathcal{C}(\mathcal{O},\,\ell_\mathcal{O}N)
\stackrel{\sim}{\rightarrow}\mathcal{C}(\mathcal{O})
\quad\textrm{and}\quad
\mathcal{C}(\mathcal{O},\,N)
\stackrel{\sim}{\rightarrow}\mathcal{C}(\mathcal{O})
\end{equation*}
by Lemma \ref{COlCO1}, we obtain an isomorphism
\begin{equation*}
\mathcal{C}(\mathcal{O},\,\ell_\mathcal{O}N)
\stackrel{\sim}{\rightarrow}\mathcal{C}(\mathcal{O},\,N).
\end{equation*}
Thus $\mathfrak{a}_1,\,\mathfrak{a}_2,\,\ldots,\,\mathfrak{a}_r$ are also 
representatives of classes in $\mathcal{C}(\mathcal{O},\,N)$.
Now, let $\mathfrak{b}_1,\,\mathfrak{b}_2,\,\ldots\mathfrak{b}_s\in P(\mathcal{O},\,\ell_\mathcal{O}N)$
be representatives of classes in $P(\mathcal{O},\,\ell_\mathcal{O}N)/P_N(\mathcal{O},\,\ell_\mathcal{O}N)$.
Then we see by Lemma \ref{PMNPN} that they are also representatives of classes in  $P(\mathcal{O},\,N)/P_N(\mathcal{O})$.
\par
Note that there are natural one-to-one correspondences
\begin{eqnarray*}
I(\mathcal{O},\,\ell_\mathcal{O}N)/P(\mathcal{O},\,\ell_\mathcal{O}N)
\times P(\mathcal{O},\,\ell_\mathcal{O}N)/P_N(\mathcal{O},\,\ell_\mathcal{O}N)&\rightarrow&
I(\mathcal{O},\,\ell_\mathcal{O}N)/P_N(\mathcal{O},\,\ell_\mathcal{O}N),\quad\textrm{and}\\
I(\mathcal{O},\,N)/P(\mathcal{O},\,N)\times
P(\mathcal{O},\,N)/P_N(\mathcal{O})&\rightarrow&
I(\mathcal{O},\,N)/P_N(\mathcal{O}).
\end{eqnarray*}
Hence
\begin{equation*}
\mathfrak{a}_i\mathfrak{b}_j\quad(i=1,\,2,\,\ldots,\,r,~
j=1,\,2,\,\ldots,\,s)
\end{equation*}
are representatives of classes in 
$\mathcal{C}_N(\mathcal{O},\,\ell_\mathcal{O}N)$, and also of classes in $\mathcal{C}_N(\mathcal{O})$.
This observation implies that the homomorphism in (\ref{C1})
is in fact an isomorphism.
\end{proof}

\begin{remark}\label{CNOIP}
We achieve by Propositions \ref{generalized} and \ref{MNN} that
\begin{equation*}
\mathcal{C}_N(\mathcal{O})\simeq
I(\mathcal{O}_K,\,\ell_\mathcal{O}N)/P_{\mathbb{Z},\,N}(\mathcal{O}_K,\,\ell_\mathcal{O}N).
\end{equation*}
\end{remark}

\begin{remark}\label{CCOMN}
Let $C\in\mathcal{C}_N(\mathcal{O})$.
By Proposition \ref{MNN}, we have
\begin{equation*}
C=[\mathfrak{a}\mathfrak{b}^{-1}]\quad\textrm{for some}~\mathfrak{a},\,\mathfrak{b}\in\mathcal{M}(\mathcal{O},\,\ell_\mathcal{O}N).
\end{equation*}
If $h$ is the order of the group $\mathcal{C}_N(\mathcal{O})$, then we see that
\begin{equation*}
C=[\mathfrak{b}]^h[\mathfrak{a}\mathfrak{b}^{-1}]=[\mathfrak{a}\mathfrak{b}^{h-1}]
\quad\textrm{and}\quad\mathfrak{a}\mathfrak{b}^{h-1}\in\mathcal{M}(\mathcal{O},\,\ell_\mathcal{O}N),
\end{equation*}
which claims that $C\cap\mathcal{M}(\mathcal{O},\,\ell_\mathcal{O}N)\neq\emptyset$.
\end{remark}

\section {The elliptic curve $E_\mathcal{O}$ with complex multiplication by $\mathcal{O}$}\label{sect3}

We shall introduce specific models
of elliptic curves with complex multiplication.
\par
Let $g_2=g_2(\mathcal{O})$ and $g_3=g_3(\mathcal{O})$ be the usual scaled Eisenstein
series and
\begin{equation*}
j(\mathcal{O})=1728\,\frac{g_2^3}{\Delta}\quad\textrm{with}~\Delta=g_2^3-27g_3^2.
\end{equation*}
When $D_\mathcal{O}\neq-3,\,-4$,
we let $E_\mathcal{O}$ be the elliptic curve with
$j$-invariant $j(E_\mathcal{O})=j(\mathcal{O})$  given by the Weierstrass equation
\begin{equation}\label{EO}
E_\mathcal{O}~:~y^2=4x^3-A_\mathcal{O}x
-B_\mathcal{O}
\end{equation}
with base point $O=[0:1:0]$,
where
\begin{equation}\label{AandB}
A_\mathcal{O}=\frac{j(\mathcal{O})(j(\mathcal{O})-1728)}{2^{12}3^9}\quad\textrm{and}\quad
B_\mathcal{O}=\frac{j(\mathcal{O})(j(\mathcal{O})-1728)^2}{2^{18}3^{15}}.
\end{equation}
If $\wp(\,\cdot\,;\,\mathcal{O}):\mathbb{C}\rightarrow\mathbb{C}\cup\{\infty\}$
is the Weierstrass $\wp$-function for the lattice $\mathcal{O}$ with
derivative $\wp'$, then we get a complex analytic isomorphism
\begin{eqnarray*}
\mathbb{C}/\mathcal{O} & \stackrel{\sim}{\rightarrow} & E(\mathbb{C})\quad(\subset\mathbb{P}^2(\mathbb{C}))\\
z+\mathcal{O} & \mapsto & \left[x(z;\,\mathcal{O}):
y(z;\,\mathcal{O}):1\right]=\left[
\frac{g_2g_3}{\Delta}\wp(z;\,\mathcal{O}):\sqrt{
\left(\frac{g_2g_3}{\Delta}\right)^3}\wp'(z;\,\mathcal{O}):1\right]
\end{eqnarray*}
(cf. \cite[$\S$VI.3]{Silverman09}).
Note that
$g_2g_3\neq0$ because we are assuming that $D_\mathcal{O}\neq-3,\,-4$ (cf. \cite[Exercise 10.19]{Cox}).
For each $\mathbf{v}=\begin{bmatrix}v_1&v_2\end{bmatrix}\in M_{1,\,2}(\mathbb{Q})
\setminus M_{1,\,2}(\mathbb{Z})$, we set
\begin{eqnarray}
X_\mathbf{v}&=&\frac{g_2g_3}{\Delta}\wp(v_1\tau_\mathcal{O}+v_2;\,\mathcal{O}),
\label{Xv}\\
Y_\mathbf{v}&=&\sqrt{\left(\frac{g_2g_3}{\Delta}\right)^3}\wp'(v_1\tau_\mathcal{O}+v_2;\,\mathcal{O}),\label{Yv}
\end{eqnarray}
where $\tau_\mathcal{O}$ is the element of $\mathbb{H}$ described in (\ref{tauO}). 

\begin{lemma}\label{basicXY}
Let $\mathbf{u},\,\mathbf{v}\in M_{1,\,2}(\mathbb{Q})\setminus M_{1,\,2}(\mathbb{Z})$
and $\mathbf{n}\in M_{1,\,2}(\mathbb{Z})$. Then we have
\begin{enumerate}
\item[\textup{(i)}] $X_\mathbf{u}=X_\mathbf{v}$ if and only if $\mathbf{u}\equiv
\mathbf{v}$ or $-\mathbf{v}\Mod{M_{1,\,2}(\mathbb{Z})}$.
\item[\textup{(ii)}] $Y_{\mathbf{v}+\mathbf{n}}=Y_\mathbf{v}$.
\item[\textup{(iii)}] $Y_{-\mathbf{v}}=-Y_\mathbf{v}$.
\item[\textup{(iv)}] $Y_\mathbf{v}=0$ if and only if $2\mathbf{v}\in M_{1,\,2}(\mathbb{Z})$.
\end{enumerate}
\end{lemma}
\begin{proof}
See \cite[$\S$10.A]{Cox}.
\end{proof}

\begin{proposition}\label{CM}
The theory of complex multiplication yields the following results.
\begin{enumerate}
\item[\textup{(i)}] $H_\mathcal{O}=K(j(\mathcal{O}))$.
\item[\textup{(ii)}] If $D_\mathcal{O}\neq-3,\,-4$ and  $N\geq2$, then
\begin{equation*}
K_{\mathcal{O},\,N}=H_\mathcal{O}\left(
X_{\left[\begin{smallmatrix}
0&\frac{1}{N}
\end{smallmatrix}\right]}
\right).
\end{equation*}
\end{enumerate}
\end{proposition}
\begin{proof}
See \cite[Theorem 5 in Chapter 10]{Lang87} and \cite[Theorem 6.2.3]{Schertz}.
\end{proof}

\section {Fields of modular functions}

We shall recall some necessary properties of Fricke functions
and Siegel functions.
\par
The modular group $\mathrm{SL}_2(\mathbb{Z})$ acts on
the complex upper half-plane $\mathbb{H}=\{\tau\in\mathbb{C}~|~
\mathrm{Im}(\tau)>0\}$ by fractional linear transformations.
Let $j$ be the elliptic modular function on $\mathbb{H}$, that is,
\begin{equation*}
j(\tau)=j([\tau,\,1])\quad(\tau\in\mathbb{H}).
\end{equation*}
For each $\mathbf{v}\in M_{1,\,2}(\mathbb{Q})\setminus M_{1,\,2}(\mathbb{Z})$,
the Fricke function $f_\mathbf{v}$ on $\mathbb{H}$ is defined by
\begin{equation}\label{fv}
f_\mathbf{v}(\tau)=-2^73^3\frac{g_2([\tau,\,1])g_3([\tau,\,1])}
{\Delta([\tau,\,1])}\wp(v_1\tau+v_2;\,[\tau,\,1])\quad(\tau\in\mathbb{H}).
\end{equation}
And, for a positive integer $N$, let
\begin{equation}\label{FN}
\mathcal{F}_N=\left\{
\begin{array}{ll}
\mathbb{Q}(j) & \textrm{if}~N=1,\\
\mathbb{Q}(j,\,f_\mathbf{v}~|~\mathbf{v}\in M_{1,\,2}(\mathbb{Q})\setminus M_{1,\,2}(\mathbb{Z})~
\textrm{satisfies}~N\mathbf{v}\in M_{1,\,2}(\mathbb{Z})) & \textrm{if}~N\geq2.
\end{array}\right.
\end{equation}

\begin{proposition}\label{SL2action}
The field $\mathcal{F}_N$ is a Galois extension of $\mathcal{F}_1$ whose Galois group
is isomorphic to $\mathrm{GL}_2(\mathbb{Z}/N\mathbb{Z})/\langle-I_2\rangle$. If $\gamma\in\mathrm{SL}_2(\mathbb{Z})$, then
\begin{equation*}
f^{\widetilde{\gamma}}=f\circ\gamma
\end{equation*}
where $\widetilde{\gamma}$ is the image of $\gamma$ in $\mathrm{GL}_2(\mathbb{Z}/N\mathbb{Z})/\langle-I_2\rangle$.
\end{proposition}
\begin{proof}
See \cite[Theorem 6.6]{Shimura}.
\end{proof}

Furthermore, $\mathcal{F}_N$ coincides with the field of meromorphic modular functions
for the principal congruence subgroup
\begin{equation*}
\Gamma(N)=\{\gamma\in\mathrm{SL}_2(\mathbb{Z})~|~\gamma\equiv I_2\Mod{NM_2(\mathbb{Z})}\}
\end{equation*}
whose Fourier coefficients belong to the $N$th cyclotomic field $\mathbb{Q}(\zeta_N)$
with $\zeta_N=e^{\frac{2\pi\mathrm{i}}{N}}$ (cf. \cite[Proposition 6.9]{Shimura}).
\par
On the other hand, for $\mathbf{v}=\begin{bmatrix}v_1 & v_2\end{bmatrix}\in M_{1,\,2}(\mathbb{Q})\setminus
M_{1,\,2}(\mathbb{Z})$, the Siegel function $g_\mathbf{v}$ on $\mathbb{H}$
is given by the infinite product expansion
\begin{equation}\label{infiniteproduct}
g_\mathbf{v}(\tau)=-q^{\frac{1}{2}\mathbf{B}_2(v_1)}e^{\pi\mathrm{i}v_2(v_1-1)}
(1-q_z)\prod_{n=1}^\infty(1-q^nq_z)(1-q^nq_z^{-1})\quad(\tau\in\mathbb{H})
\end{equation}
where $q_z=e^{2\pi\mathrm{i}z}$ with
$z=v_1\tau+v_2$ and $\mathbf{B}_2(x)=x^2-x+\frac{1}{6}$
is the second Bernoulli polynomial.
Observe that $g_\mathbf{v}$ has neither a zero nor a pole on $\mathbb{H}$.
One can refer to \cite{K-L}
for further details on Siegel functions.
\par
For $N\geq2$, we say that $\mathbf{v}\in M_{1,\,2}(\mathbb{Q})$ is primitive modulo $N$
if $N$ is the smallest positive integer so that $N\mathbf{v}\in M_{1,\,2}(\mathbb{Z})$.
Let $V_N$ be the set of all such primitive vectors $\mathbf{v}$.
We call a collection $\{h_\mathbf{v}\}_{\mathbf{v}\in V_N}$ of functions in $\mathcal{F}_N$
a Fricke family of level $N$ if
\begin{enumerate}
\item[(i)] $h_\mathbf{v}$ is holomorphic on $\mathbb{H}$\quad($\mathbf{v}\in V_N$),
\item[(ii)] $h_\mathbf{u}=h_\mathbf{v}$ if $\mathbf{u}\equiv\mathbf{v}$
or $-\mathbf{v}\Mod{M_{1,\,2}(\mathbb{Z})}$,
\item[(iii)] $h_\mathbf{v}^\gamma=h_{\mathbf{v}\gamma}$\quad
($\mathbf{v}\in V_N$ and $\gamma\in\mathrm{GL}_2(\mathbb{Z}/N\mathbb{Z})/\langle-I_2\rangle
\simeq\mathrm{Gal}(\mathcal{F}_N/\mathcal{F}_1)$).
\end{enumerate}

\begin{proposition}\label{Frickefamily}
The collections $\{f_\mathbf{v}\}_{\mathbf{v}\in V_N}$ and
$\{g_\mathbf{v}^{12N}\}_{\mathbf{v}\in V_N}$ are Fricke families of level $N$.
\end{proposition}
\begin{proof}
See \cite[Proposition 1.3 in Chapter 2]{K-L} and \cite[Theorem 6.6]{Shimura}.
\end{proof}

\begin{lemma}\label{jgrelation}
We have a relationship between $j$ and $g_{\left[\begin{smallmatrix}0&\frac{1}{2}\end{smallmatrix}\right]}$ in such a way that
\begin{equation*}
j=\frac{\left(g_{\left[\begin{smallmatrix}0&\frac{1}{2}\end{smallmatrix}\right]}^{12}+16\right)^3}
{g_{\left[\begin{smallmatrix}0&\frac{1}{2}\end{smallmatrix}\right]}^{12}}.
\end{equation*}
\end{lemma}
\begin{proof}
See \cite[Theorem 12.17]{Cox}.
\end{proof}

Let
\begin{eqnarray*}
V'_N&=&\left\{
\begin{bmatrix}v_1 & v_2\end{bmatrix}\in V_N~|~0\leq v_1,\,v_2<1\right\},\\
T_N&=&\left\{
(\mathbf{u},\,\mathbf{v})\in V'_N\times V'_N~|~\mathbf{u}\not
\equiv\mathbf{v},\,-\mathbf{v}\Mod{M_{1,\,2}(\mathbb{Z})}
\right\}.
\end{eqnarray*}

\begin{lemma}\label{ffjj}
We deduce
\begin{equation*}
\prod_{(\mathbf{u},\,\mathbf{v})\in T_N}\left(
f_\mathbf{u}-f_\mathbf{v}
\right)^6=k\left\{j^2(j-1728)^3\right\}^{|T_N|}\quad
\textrm{for some}~k\in\mathbb{Q}\setminus\{0\}.
\end{equation*}
\end{lemma}
\begin{proof}
See \cite[Lemma 6.2]{J-K-S22}.
\end{proof}

\section {Form class groups}\label{sect5}

In this section, we shall review 
the classical form class groups
and try to find certain generalization of these as well. 

For a negative integer $D$ such that $D\equiv0$ or $1\Mod{4}$, let
\begin{equation*}
\mathcal{Q}(D)=
\left\{Q=Q\left(\begin{bmatrix}x\\y\end{bmatrix}\right)
=ax^2+bx+cy^2\in\mathbb{Z}[x,\,y]~|~
\gcd(a,\,b,\,c)=1,~b^2-4ac=D,~a>0\right\}.
\end{equation*}
The modular group $\mathrm{SL}_2(\mathbb{Z})$ acts on the set $\mathcal{Q}(D)$
from the right as
\begin{equation*}
Q^\gamma=Q\left(\gamma\begin{bmatrix}x\\y\end{bmatrix}\right)
\end{equation*}
and induces the proper equivalence $\sim$ on $\mathcal{Q}(D)$ as follows\,:
\begin{equation*}
Q\sim Q'\quad\Longleftrightarrow\quad
Q'=Q^\gamma~~\textrm{for some}~\gamma\in\mathrm{SL}_2(\mathbb{Z})
\end{equation*}
(cf. \cite[pp. 20--22]{Cox}).
We say that a form $Q=ax^2+bxy+cy^2\in\mathcal{Q}(D)$ is reduced if
\begin{equation*}
\left\{
\begin{array}{l}
|b|\leq a\leq c,\\
b\geq0~\textrm{if either}~|b|=a~\textrm{or}~a=c.
\end{array}\right.
\end{equation*}

\begin{lemma}\label{reduced}
Every form in $\mathcal{Q}(D)$ is properly equivalent to
a unique reduced form.
\end{lemma}
\begin{proof}
See \cite[Theorem 2.8]{Cox}.
\end{proof}

For $Q=ax^2+bxy+cy^2\in\mathcal{Q}(D)$, let $\omega_Q$ be the
zero of the quadratic polynomial $Q(x,\,1)$ lying in $\mathbb{H}$, namely,
\begin{equation*}
\omega_Q=\frac{-b+\sqrt{D}}{2a}.
\end{equation*}

\begin{lemma}\label{QaN}
If $Q=ax^2+bxy+cy^2\in\mathcal{Q}(D)$, then $a[\omega_Q,\,1]$ is a proper $\mathcal{O}$-ideal with
$\mathrm{N}_\mathcal{O}(a[\omega_Q,\,1])=a$.
\end{lemma}
\begin{proof}
See \cite[Lemma 7.5 and (7.16)]{Cox}.
\end{proof}

\begin{lemma}\label{B}
Let $Q=ax^2+bxy+cy^2,\,Q''=a''x^2+b''xy+c''y^2\in\mathcal{Q}(D)$ such that
$\gcd(a,\,a'',\,(b+b'')/2)=1$.
\begin{enumerate}
\item[\textup{(i)}] There is a unique integer $B$ modulo
$2aa''$ such that
\begin{equation}\label{Bcondition}
B\equiv b\Mod{2a},\quad B\equiv b''\Mod{2a''},\quad
B^2\equiv D\Mod{4aa''}.
\end{equation}
\item[\textup{(ii)}]
Let
\begin{equation}\label{DC}
Q'''=aa''x^2+Bxy+\frac{B^2-D}{4aa''}y^2
\end{equation}
where $B$ is an integer satisfying \textup{(\ref{Bcondition})}.
Then we derive
\begin{equation*}
[\omega_Q,\,1][\omega_{Q''},\,1]=[
\omega_{Q'''},\,1].
\end{equation*}
\end{enumerate}
\end{lemma}
\begin{proof}
\begin{enumerate}
\item[(i)] See \cite[Lemma 3.2]{Cox}.
\item[(ii)] See \cite[(7.13)]{Cox}.
\end{enumerate}
\end{proof}

\begin{remark}
We call the form in (\ref{DC}) a \textit{Dirichlet composition} of $Q$ and $Q''$.
\end{remark}

\begin{lemma}\label{Qgprime}
Let $Q\in\mathcal{Q}(D)$ and $M$ be a positive integer.
Then there is a matrix $\gamma$ in $\mathrm{SL}_2(\mathbb{Z})$
so that the coefficient of $x^2$ in $Q^\gamma$ is relatively prime to $M$.
\end{lemma}
\begin{proof}
See \cite[Lemmas 2.3 and 2.25]{Cox}.
\end{proof}

\begin{proposition}\label{Dirichlet}
Let $\mathcal{C}(D)=\mathcal{Q}(D)/\sim$ be the set of equivalence classes.
\begin{enumerate}
\item[\textup{(i)}] The following binary operation on $\mathcal{C}(D)$
is well defined and makes $\mathcal{C}(D)$ into a finite abelian group\,\textup{:}
let $C,\,C'\in\mathcal{C}(D)$ and so
$C=[Q],\,C'=[Q']$ for some
$Q=ax^2+bxy+cy^2,\,Q'\in\mathcal{Q}(D)$, respectively.
By \textup{Lemma \ref{Qgprime}}, one can take a matrix $\gamma$ in $\mathrm{SL}_2(\mathbb{Z})$ 
for which
$Q''=Q'^\gamma=a''x^2+b''xy+c''y^2$ satisfies
$\gcd(a,\,a'',\,(b+b'')/2)=1$.
We define the product $CC'$ by the class of any Dirichlet composition of $Q$ and $Q''$.
\item[\textup{(ii)}] If $D=D_\mathcal{O}$, then the map
\begin{eqnarray*}
\mathcal{C}(D_\mathcal{O})&\rightarrow&\mathcal{C}(\mathcal{O})\\
\mathrm{[}Q] & \mapsto & [[\omega_Q,\,1]]
\end{eqnarray*}
is a well-defined isomorphism of
$\mathcal{C}(D_\mathcal{O})$ onto $\mathcal{C}(\mathcal{O})$.
\end{enumerate}
\end{proposition}
\begin{proof}
See \cite[Theorems 3.9 and 7.7]{Cox}.
\end{proof}

Now, let
\begin{equation*}
\mathcal{Q}(D_\mathcal{O},\,N)=\{ax^2+bxy+cy^2\in\mathcal{Q}(D_\mathcal{O})~|~\gcd(a,\,N)=1\}
\end{equation*}
and
\begin{equation*}
\Gamma_1(N)=\left\{\gamma\in\mathrm{SL}_2(\mathbb{Z})~|~\gamma\equiv\begin{bmatrix}1&\mathrm{*}\\
0&1\end{bmatrix}\Mod{NM_2(\mathbb{Z})}\right\}
\end{equation*}
which is a congruence subgroup of $\mathrm{SL}_2(\mathbb{Z})$ of level $N$.
Observe that for $Q=ax^2+bxy+cy^2\in\mathcal{Q}(D_\mathcal{O},\,N)$ and $\gamma=\begin{bmatrix}p&q\\r&s\end{bmatrix}
\in\Gamma_1(N)$,
\begin{equation*}
Q^\gamma\equiv Q\left(\begin{bmatrix}1&q\\0&1\end{bmatrix}\begin{bmatrix}x\\y\end{bmatrix}\right)
\equiv ax^2+(2aq+b)xy+(aq^2+bq+c)y^2\Mod{N\mathbb{Z}[x,\,y]}.
\end{equation*}
This implies that $\Gamma_1(N)$ acts on $\mathcal{Q}(D_\mathcal{O},\,N)$.
Let $\sim_{\Gamma_1(N)}$ be the equivalence relation on $\mathcal{Q}(D_\mathcal{O},\,N)$ induced from the action
of $\Gamma_1(N)$. We mean by $\mathcal{C}_N(D_\mathcal{O})$ the set of equivalence classes, that is,
\begin{equation*}
\mathcal{C}_N(D_\mathcal{O})=\mathcal{Q}(D_\mathcal{O},\,N)/\sim_{\Gamma_1(N)}.
\end{equation*}
For $\alpha=\begin{bmatrix}s&t\\u&v\end{bmatrix}\in\mathrm{SL}_2(\mathbb{Z})$ and $\tau\in\mathbb{H}$,
we write
\begin{equation*}
j(\alpha,\,\tau)=u\tau+v.
\end{equation*}

\begin{definition}\label{binaryoperation}
We define a binary operation on $\mathcal{C}_N(D_\mathcal{O})$ as follows\,:
let $C,\,C'\in\mathcal{C}_N(D_\mathcal{O})$, and so
$C=[Q],\,C'=[Q']$ for some $Q=ax^2+bxy+cy^2,\,Q'\in\mathcal{Q}(D_\mathcal{O},\,N)$, respectively.
By \textup{Lemma \ref{Qgprime}}, there is a matrix $\gamma$ in $\mathrm{SL}_2(\mathbb{Z})$ so that
$Q''=Q'^\gamma=a''x^2+b''xy+c''y^2$ satisfies
$\gcd(a,\,a'',\,(b+b'')/2)=1$.
Let $Q'''$ be a Dirichlet composition of $Q$ and $Q''$.
And, set
\begin{equation*}
\nu_1=\frac{\omega_{Q'''}}{j(\gamma,\,\omega_{Q'''})}
\quad\textrm{and}\quad
\nu_2=\frac{1}{j(\gamma,\,\omega_{Q'''})}.
\end{equation*}
One can show that there is a pair $(u,\,v)$ of integers satisfying
\begin{equation*}
u\nu_1+v\nu_2=1
\end{equation*}
and in $\mathrm{SL}_2(\mathbb{Z})$ there exists a matrix $\sigma$ such that
\begin{equation*}
\sigma\equiv\begin{bmatrix}\mathrm{*} & \mathrm{*}\\
u & v\end{bmatrix}\Mod{NM_2(\mathbb{Z})}.
\end{equation*}
We then define
\begin{equation*}
CC'=\left[
\left(Q'''\right)^{\sigma^{-1}}\right].
\end{equation*}
\end{definition}

In $\S$\ref{sect9}, we shall prove that the binary operation given in Definition \ref{binaryoperation}
is well defined
and makes $\mathcal{C}_N(D_\mathcal{O})$ a finite abelian group isomorphic to $\mathcal{C}_N(\mathcal{O})$
(Theorem \ref{formideal}).
And, through
the binary operation on $\mathcal{C}_N(D_\mathcal{O})$ we regard
the natural map $\mathcal{C}_N(D_\mathcal{O})\rightarrow
\mathcal{C}_1(D_\mathcal{O})=\mathcal{C}(D_\mathcal{O})$ as a homomorphism.

\section {Inequalities on special values of modular functions}

We shall develop certain inequalities on the
special values of $j$ and Siegel functions.
By using these inequalities we shall present another new generators of $H_\mathcal{O}$
and $K_{\mathcal{O},\,N}$,
for later use in $\S$\ref{sect7}, 
which are different from the classical ones stated in Proposition \ref{CM}.

\begin{lemma}\label{ginequality}
Let $\tau\in\mathbb{H}$ and $t=|e^{2\pi\mathrm{i}\tau}|$.
\begin{enumerate}
\item[\textup{(i)}] $\left|g_{\left[\begin{smallmatrix}0&\frac{1}{2}\end{smallmatrix}\right]}(\tau)\right|
\leq 2t^\frac{1}{12}e^{-2+\frac{2}{1-t}}$.
\item[\textup{(ii)}] $\left|g_{\left[\begin{smallmatrix}\frac{1}{2}&s\end{smallmatrix}\right]}(\tau)\right|
\leq t^{-\frac{1}{24}}e^{t^\frac{1}{2}(1+\frac{2}{1-t})}$
for any $s\in\mathbb{Q}$.
\end{enumerate}
\end{lemma}
\begin{proof}
See \cite[Lemma 5.2]{J-K-S18}.
\end{proof}

Let $Q_0$ be the principal form in $\mathcal{Q}(D_\mathcal{O})$ defined by
\begin{equation*}
Q_0=\left\{
\begin{array}{ll}
\displaystyle x^2-\frac{D_\mathcal{O}}{4}y^2 & \textrm{if}~D_\mathcal{O}\equiv0\Mod{4},\\
\displaystyle x^2+xy+\frac{1-D_\mathcal{O}}{4}y^2 & \textrm{if}~D_\mathcal{O}\equiv1\Mod{4},
\end{array}\right.
\end{equation*}
which represents the identity class in $\mathcal{C}(D_\mathcal{O})$
(cf. \cite[Theorem 3.9]{Cox}). 
Then we see that
\begin{equation*}
\omega_{Q_0}=\tau_\mathcal{O}\quad\textrm{and}\quad [\omega_{Q_0},\,1]=\mathcal{O}.
\end{equation*}
Let $h_\mathcal{O}$ denote the order of the group $\mathcal{C}(\mathcal{O})$ and so
$h_\mathcal{O}=[H_\mathcal{O}:K]$.

\begin{lemma}\label{jinequality}
Assume that $h_\mathcal{O}\geq2$.
Let $Q=ax^2+bxy+cy^2$ be a reduced form in $\mathcal{Q}(D_\mathcal{O})$ such that $Q\neq Q_0$.
\begin{enumerate}
\item[\textup{(i)}] $\left|g_{\left[\begin{smallmatrix}
0&\frac{1}{2}\end{smallmatrix}\right]}(\omega_Q)\right|
>1.98e^{-\frac{\pi\sqrt{|D_\mathcal{O}|}}{24}}$.
\item[\textup{(ii)}] $\displaystyle\left|\frac{j(\omega_Q)^2(j(\omega_Q)-1728)^3}{j(\omega_{Q_0})^2(j(\omega_{Q_0})-1728)^3}
\right|<877383 e^{-\frac{5\pi\sqrt{|D_\mathcal{O}|}}{2}}$ \textup{(}$<1$\textup{)}.
\end{enumerate}
\end{lemma}
\begin{proof}
\begin{enumerate}
\item[(i)] Since $Q$ is a reduced form in $\mathcal{Q}(D_\mathcal{O})$ such that $Q\neq Q_0$, we have
\begin{equation}\label{ainequality}
2\leq a\leq \sqrt{\frac{|D_\mathcal{O}|}{3}}
\end{equation}
(\cite[p. 24]{Cox}). Then we derive that
\begin{eqnarray*}
\left|g_{\left[\begin{smallmatrix}
0&\frac{1}{2}\end{smallmatrix}\right]}(\omega_Q)\right|
&\geq&2e^{-\frac{\pi\sqrt{|D_\mathcal{O}|}}{24}}
\prod_{n=1}^\infty(1-e^{-\pi\sqrt{3}n})^2\quad\textrm{by
the definition (\ref{infiniteproduct}) and (\ref{ainequality})}\\
&\geq&2e^{-\frac{\pi\sqrt{|D_\mathcal{O}|}}{24}}
\prod_{n=1}^\infty e^{-2k^n}
\quad\textrm{with}~k=e^{-\frac{\pi99\sqrt{3}}{100}}\\
&&\hspace{3.7cm}\textrm{because}~1-X>e^{-X^\frac{99}{100}}~\textrm{for}~0<X\leq e^{-\pi\sqrt{3}}\\
&=&2e^{-\frac{\pi\sqrt{|D_\mathcal{O}|}}{24}}
e^{-\frac{2k}{1-k}}\\
&>&1.98e^{-\frac{\pi\sqrt{|D_\mathcal{O}|}}{24}}.
\end{eqnarray*}
\item[(ii)] First, consider the case where $D_\mathcal{O}\leq-20$.
We then observe by Lemma \ref{ginequality} (i) and (\ref{ainequality}) that
\begin{equation}\label{0.0033}
\left|g_{\left[\begin{smallmatrix}
0&\frac{1}{2}\end{smallmatrix}\right]}(\omega_{Q_0})^{12}\right|
\leq2^{12}e^{-\pi\sqrt{|D_\mathcal{O}|}}e^{-24+\frac{24}{1-e^{-\pi\sqrt{|D_\mathcal{O}|}}}}
\leq4097e^{-\pi\sqrt{|D_\mathcal{O}|}}
<0.0033
\end{equation}
and
\begin{equation}\label{19.71}
\left|g_{\left[\begin{smallmatrix}
0&\frac{1}{2}\end{smallmatrix}\right]}(\omega_Q)^{12}\right|
\leq2^{12}e^{-\frac{\pi\sqrt{|D_\mathcal{O}|}}{a}}e^{-24+\frac{24}
{1-e^{-\frac{\pi\sqrt{|D_\mathcal{O}|}}{a}}}}
\leq
2^{12}e^{-\pi\sqrt{3}}
e^{-24+\frac{24}
{1-e^{-\pi\sqrt{3}}}}
<19.71.
\end{equation}
Hence we find that
\begin{eqnarray*}
&&\left|\frac{j(\omega_Q)^2(j(\omega_Q)-1728)^3}{j(\omega_{Q_0})^2(j(\omega_{Q_0})-1728)^3}\right|\\
&=&
\left|
\frac{g_{\left[\begin{smallmatrix}0&\frac{1}{2}\end{smallmatrix}\right]}(\omega_{Q_0})^{12}}{g_{\left[\begin{smallmatrix}0&\frac{1}{2}\end{smallmatrix}\right]}(\omega_Q)^{12}}
\right|^5
\left|
\frac
{\left(g_{\left[\begin{smallmatrix}0&\frac{1}{2}\end{smallmatrix}\right]}(\omega_Q)^{12}+16\right)^2
\left(g_{\left[\begin{smallmatrix}0&\frac{1}{2}\end{smallmatrix}\right]}(\omega_Q)^{12}+64\right)
\left(g_{\left[\begin{smallmatrix}0&\frac{1}{2}\end{smallmatrix}\right]}(\omega_Q)^{12}-8\right)^2}
{\left(g_{\left[\begin{smallmatrix}0&\frac{1}{2}\end{smallmatrix}\right]}(\omega_{Q_0})^{12}+16\right)^2
\left(g_{\left[\begin{smallmatrix}0&\frac{1}{2}\end{smallmatrix}\right]}(\omega_{Q_0})^{12}+64\right)
\left(g_{\left[\begin{smallmatrix}0&\frac{1}{2}\end{smallmatrix}\right]}(\omega_{Q_0})^{12}-8\right)^2}
\right|^3\\
&&\hspace{10.8cm}\textrm{by Lemma \ref{jgrelation}}\\
&\leq&\left(
\frac{4097e^{-\pi\sqrt{|D_\mathcal{O}|}}}{1.98e^{-\frac{\pi\sqrt{|D_\mathcal{O}|}}{2}}}
\right)^5
\left|\frac{(16+19.71)^2(64+19.71)(8+19.71)^2}{(16-0.0033)^2(64-0.0033)(8-0.0033)^2}
\right|^3
\quad\textrm{by (i), (\ref{0.0033}) and (\ref{19.71})}\\
&<&877383e^{-\frac{5\pi\sqrt{|D_\mathcal{O}|}}{2}}.
\end{eqnarray*}
\par
The only remaining case is
$K=\mathbb{Q}(\sqrt{-15})$ and $\mathcal{O}=\mathcal{O}_K$, and so $h_\mathcal{O}=2$ and
\begin{equation*}
Q_0=x^2+xy+4y^2,\quad
Q=2x^2+xy+2y^2.
\end{equation*}
One can then numerically verify that (ii) also holds in this case (cf. \cite[Remark 4.2]{J-K-S22}).
\end{enumerate}
\end{proof}

\begin{proposition}\label{j(omega)}
The special value $j(\tau_\mathcal{O})=j(\omega_{Q_0})$ generates $H_\mathcal{O}$ over $K$.
If $Q_1,\,Q_2,\,\ldots,\,Q_{h_\mathcal{O}}$ are reduced forms in $\mathcal{Q}(D_\mathcal{O})$,
then the special values $j(\omega_{Q_1}),\, j(\omega_{Q_2}),\,
\ldots,\,j(\omega_{Q_{h_\mathcal{O}}})$ are distinct Galois conjugates of $j(\tau_\mathcal{O})$ over $K$.
\end{proposition}
\begin{proof}
See Lemma \ref{reduced}, Proposition \ref{Dirichlet} (ii) and \cite[Theorem 5 in Chapter 10]{Lang87}.
\end{proof}

\begin{proposition}\label{jgenerate}
The special value
$\left\{j(\tau_\mathcal{O})^2(j(\tau_\mathcal{O})-1728)^3\right\}^n$
generates $H_\mathcal{O}$ over $K$ for any positive integer $n$.
\end{proposition}
\begin{proof}
If $h_\mathcal{O}=1$ and so $H_\mathcal{O}=K$, then the assertion is trivial.
\par
Now, consider the case where $h_\mathcal{O}\geq2$.
Since $D_\mathcal{O}\neq-3,\,-4$, we have $j(\tau_\mathcal{O})^2(j(\mathcal{O})-1728)^3\neq0$
(cf. \cite[Theorem 7.30 (ii), (10.8) and Exercise 10.19]{Cox}).
Let $\sigma$ be an element of $\mathrm{Gal}(H_\mathcal{O}/K)$ leaving the value $\left\{j(\tau_\mathcal{O})^2(j(\tau_\mathcal{O})-1728)^3\right\}^n$ fixed. We then achieve
\begin{eqnarray*}
1&=&\left|\frac{\{j(\omega_{Q_0})^2(j(\omega_{Q_0})-1728)^3\}^\sigma}{j(\omega_{Q_0})^2(j(\omega_{Q_0})-1728)^3}\right|^n\\
&=&\left|\frac{j(\omega_Q)^2(j(\omega_Q)-1728)^3}{j(\omega_{Q_0})^2(j(\omega_{Q_0})-1728)^3}\right|^n
\quad\textrm{for some reduced form $Q$ in $\mathcal{Q}(D_\mathcal{O})$ by Proposition \ref{j(omega)}}.
\end{eqnarray*}
By Lemma \ref{jinequality} (ii), we must get $Q=Q_0$
and hence $\sigma$ is the identity element. This implies by Galois theory that
$\left\{j(\tau_\mathcal{O})^2(j(\tau_\mathcal{O})-1728)^3\right\}^n$ generates $H_\mathcal{O}$ over $K$.
\end{proof}

When $D_\mathcal{O}\neq-3,\,-4$, let $E_\mathcal{O}$ be the elliptic curve given by the special model in (\ref{EO}).
For $\mathbf{v}\in M_{1,\,2}(\mathbb{Q})\setminus
M_{1,\,2}(\mathbb{Z})$, recall the definitions of $X_\mathbf{v}$ and $Y_\mathbf{v}$
in (\ref{Xv}) and (\ref{Yv}), respectively.

\begin{proposition}
Assume that $D_\mathcal{O}\neq-3,\,-4$.
If $N\geq2$ and $K_{\mathcal{O},\,N}$ properly contains $H_\mathcal{O}$, then
\begin{equation*}
K_{\mathcal{O},\,N}=K\left(
X_{\left[\begin{smallmatrix}0&\frac{1}{N}\end{smallmatrix}\right]},\,
Y_{\left[\begin{smallmatrix}0&\frac{1}{N}\end{smallmatrix}\right]}^2\right).
\end{equation*}
\end{proposition}
\begin{proof}
Let $L=K(X,\,Y^2)$ with $X=X_{\left[\begin{smallmatrix}0&\frac{1}{N}\end{smallmatrix}\right]}$
and $Y=Y_{\left[\begin{smallmatrix}0&\frac{1}{N}\end{smallmatrix}\right]}$.
Since $X$ and $Y^2$ belong to $K_{\mathcal{O},\,N}$ by Proposition \ref{CM} and the relation
\begin{equation*}
Y_{\left[\begin{smallmatrix}0&\frac{1}{N}\end{smallmatrix}\right]}^2
=4X_{\left[\begin{smallmatrix}0&\frac{1}{N}\end{smallmatrix}\right]}^3-
A_\mathcal{O}X_{\left[\begin{smallmatrix}0&\frac{1}{N}\end{smallmatrix}\right]}-B_\mathcal{O},
\end{equation*}
$L$ is a subfield of $K_{\mathcal{O},\,N}$. And, $A_\mathcal{O},\,B_\mathcal{O}\in H_\mathcal{O}$.
\par
Suppose on the contrary that $L\neq K_{\mathcal{O},\,N}$. Then there
exists a nonidentity element $\sigma$ of $\mathrm{Gal}(K_{\mathcal{O},\,N}/K)$
which leaves the values $X$ and $Y^2$ fixed.
Here we further note that
\begin{equation}\label{sigmanotin}
\sigma\not\in\mathrm{Gal}(K_{\mathcal{O},\,N}/H_\mathcal{O})
\end{equation}
because $K_{\mathcal{O},\,N}=H_\mathcal{O}(X)$ by Proposition \ref{CM}.
Since
\begin{equation*}
4X^3-A_\mathcal{O}X-B_\mathcal{O}=Y^2=
4X^3-A_\mathcal{O}^\sigma X-B_\mathcal{O}^\sigma,
\end{equation*}
we obtain that
\begin{equation}\label{AABB}
(A_\mathcal{O}^\sigma-A_\mathcal{O})X=B_\mathcal{O}-B_\mathcal{O}^\sigma.
\end{equation}
On the other hand, we see that
\begin{equation*}
A_\mathcal{O}B_\mathcal{O}=\frac{j(\tau_\mathcal{O})^2(j(\tau_\mathcal{O})-1728)^3}{2^{30}3^{24}},
\end{equation*}
which generates $H_\mathcal{O}$ over $K$ by Proposition \ref{jgenerate}.
Thus $A_\mathcal{O}B_\mathcal{O}$ is not fixed by $\sigma$ by (\ref{sigmanotin}),
and so $A_\mathcal{O}^\sigma\neq A_\mathcal{O}$ or $B_\mathcal{O}^\sigma\neq B_\mathcal{O}$.
It follows
from  (\ref{AABB}) that $A_\mathcal{O}^\sigma\neq A_\mathcal{O}$ and
\begin{equation*}
X=\frac{B_\mathcal{O}-B_\mathcal{O}^\sigma}{A_\mathcal{O}^\sigma-A_\mathcal{O}}\in H_\mathcal{O}.
\end{equation*}
Then we derive
$K_{\mathcal{O},\,N}=H_\mathcal{O}(X)=H_\mathcal{O}$,
which contradicts the hypothesis that $K_{\mathcal{O},\,N}$ properly contains $H_\mathcal{O}$.
\par
Therefore we conclude that
\begin{equation*}
L=K\left(X_{\left[\begin{smallmatrix}0&\frac{1}{N}\end{smallmatrix}\right]},\,
Y_{\left[\begin{smallmatrix}0&\frac{1}{N}\end{smallmatrix}\right]}^2\right)=K_{\mathcal{O},\,N}.
\end{equation*}
\end{proof}

\section {Extension fields of $\mathbb{Q}$ generated by torsion points of $E_\mathcal{O}$}\label{sect7}

In $\S$\ref{sect7} and $\S$\ref{sect8}, we assume that $D_\mathcal{O}\neq-3,\,-4$.
Let
$E_\mathcal{O}[N]$ be the group of $N$-torsion points of $E_\mathcal{O}$
and $\mathbb{Q}(E_\mathcal{O}[N])$ be the extension field of $\mathbb{Q}$
generated by the coordinates of points in $E_\mathcal{O}[N]$.
Then we have
\begin{equation*}
\mathbb{Q}(E_\mathcal{O}[N])=
\left\{\begin{array}{ll}
\mathbb{Q} & \textrm{if}~N=1,\\
\mathbb{Q}(X_\mathbf{v},\,Y_\mathbf{v}~|~
\mathbb{Q}(X_\mathbf{v},\,Y_\mathbf{v}~|~
\mathbf{v}\in W_N) & \textrm{if}~N\geq2,
\end{array}\right.
\end{equation*}
where
\begin{equation*}
W_N=\{\mathbf{v}\in M_{1,\,2}(\mathbb{Q})\setminus M_{1,\,2}(\mathbb{Z})~|~
N\mathbf{v}\in M_{1,\,2}(\mathbb{Z})\}.
\end{equation*}
In this section, we shall examine the field $\mathbb{Q}(E_\mathcal{O}[N])$ by comparing with
$K_{\mathcal{O},\,N}$.

\begin{lemma}\label{complex1}
Let $N\geq2$.
If
$\{h_\mathbf{v}\}_{\mathbf{v}\in V_N}$ is a Fricke family of level $N$, then
\begin{equation*}
\overline{h_\mathbf{v}(\tau_\mathcal{O})}=h_{\mathbf{v}\left[\begin{smallmatrix}
1&~b_\mathcal{O}\\0&-1
\end{smallmatrix}\right]}(\tau_\mathcal{O})
\quad(\mathbf{v}\in V_N).
\end{equation*}
\end{lemma}
\begin{proof}
See \cite[Proposition 1.4 in Chapter 2]{K-L}.
\end{proof}

By the definitions (\ref{Xv}) and (\ref{fv}), we attain
\begin{equation}\label{Xf}
X_\mathbf{v}=-\frac{1}{2^73^3}f_\mathbf{v}(\tau_\mathcal{O})
\quad(\mathbf{v}\in W_N).
\end{equation}

\begin{remark}\label{Xcomplex}
We find that for $\mathbf{v}\in W_N$
\begin{eqnarray*}
\overline{X_\mathbf{v}}
&=&-\frac{1}{2^73^3}\overline{f_\mathbf{v}(\tau_\mathcal{O})}\\
\\&=&
-\frac{1}{2^73^3}f_{\mathbf{v}\left[\begin{smallmatrix}
1&~b_\mathcal{O}\\0&-1
\end{smallmatrix}\right]}(\tau_\mathcal{O})\quad\textrm{by Proposition \ref{Frickefamily} and Lemma \ref{complex1}}\\
&=&X_{\mathbf{v}\left[\begin{smallmatrix}
1&~b_\mathcal{O}\\0&-1
\end{smallmatrix}\right]}.
\end{eqnarray*}
\end{remark}

Let $R_{\mathcal{O},\,N}$ be the maximal real subfield of $K_{\mathcal{O},\,N}$.

\begin{lemma}\label{maximalreal}
We have
\begin{equation*}
R_{\mathcal{O},\,N}=\mathbb{Q}\left(j(\tau_\mathcal{O}),\,
X_{\left[\begin{smallmatrix}0&\frac{1}{N}\end{smallmatrix}\right]}\right)
\quad\textrm{and}\quad K_{\mathcal{O},\,N}=KR_{\mathcal{O},\,N}.
\end{equation*}
\end{lemma}
\begin{proof}
Observe that
\begin{eqnarray*}
\overline{X_{\left[\begin{smallmatrix}
0&\frac{1}{N}
\end{smallmatrix}\right]}}
&=&
X_{\left[\begin{smallmatrix}
0&\frac{1}{N}
\end{smallmatrix}\right]\left[\begin{smallmatrix}
1&~b_\mathcal{O}\\0&-1
\end{smallmatrix}\right]}(\tau_\mathcal{O})\quad\textrm{by Remark \ref{Xcomplex}}\\
&=&X_{\left[\begin{smallmatrix}
0&-\frac{1}{N}\end{smallmatrix}\right]}\\
&=&X_{\left[\begin{smallmatrix}
0&\frac{1}{N}
\end{smallmatrix}\right]}\quad\textrm{by Lemma \ref{basicXY} (i)},
\end{eqnarray*}
and hence $X_{\left[\begin{smallmatrix}0&\frac{1}{N}\end{smallmatrix}\right]}\in\mathbb{R}$.
Furthermore, since $j(\tau_\mathcal{O})\in\mathbb{R}$ (cf. \cite[p. 179]{Silverman94}) and $K$ is imaginary quadratic,
we derive by Proposition \ref{CM} that \begin{equation*}
R_{\mathcal{O},\,N}=\mathbb{Q}\left(j(\tau_\mathcal{O}),\,
X_{\left[\begin{smallmatrix}0&\frac{1}{N}\end{smallmatrix}\right]}\right)
\quad\textrm{and}\quad K_{\mathcal{O},\,N}=KR_{\mathcal{O},\,N}.
\end{equation*}
\end{proof}

\begin{lemma}\label{QXQj}
If $N\geq2$, then $\mathbb{Q}(X_\mathbf{v}~|~\mathbf{v}\in W_N)$ contains $\mathbb{Q}(j(\tau_\mathcal{O}))$.
\end{lemma}
\begin{proof}
Let $M$ be the maximal real subfield of $H_\mathcal{O}$.
Since $j(\tau_\mathcal{O})\in\mathbb{R}$ and
\begin{equation}\label{Mj}
H_\mathcal{O}=
K(j(\tau_\mathcal{O}))=
K\left(
\left\{j(\tau_\mathcal{O})^2(j(\tau_\mathcal{O})-1728)^3\right\}^{|T_N|}\right)
\end{equation}
by Propositions \ref{j(omega)} and \ref{jgenerate}, we deduce that $[H_\mathcal{O}:M]=2$ and
\begin{equation*}
M=\mathbb{Q}(j(\tau_\mathcal{O}))=
\mathbb{Q}\left(
\left\{j(\tau_\mathcal{O})^2(j(\tau_\mathcal{O})-1728)^3\right\}^{|T_N|}\right).
\end{equation*}
Then we find that
\begin{eqnarray*}
\mathbb{Q}(X_\mathbf{v}~|~\mathbf{v}\in W_N)&\supseteq&\mathbb{Q}
\left(
\prod_{(\mathbf{u},\,\mathbf{v})\in T_N}
(X_\mathbf{u}-X_\mathbf{v})
\right)\\
&=&\mathbb{Q}\left(
\left\{j(\tau_\mathcal{O})^2(j(\tau_\mathcal{O})-1728)^3\right\}^{|T_N|}\right)
\quad\textrm{by (\ref{Xf}) and Lemma \ref{ffjj}}\\
&=&\mathbb{Q}(j(\tau_\mathcal{O}))\quad\textrm{by (\ref{Mj})}.
\end{eqnarray*}
\end{proof}

\begin{proposition}\label{Xgenerate}
Assume that $D_\mathcal{O}\neq-3,\,-4$ and $N\geq2$.
\begin{enumerate}
\item[\textup{(i)}] The field $\mathbb{Q}(X_\mathbf{v}~|~\mathbf{v}\in W_N)$ contains
$R_{\mathcal{O},\,N}$. Furthermore, $K(X_\mathbf{v}~|~\mathbf{v}\in W_N)=K_{\mathcal{O},\,N}$.
\item[\textup{(ii)}] If $D_\mathcal{O}\equiv0\Mod{4}$, then
$\mathbb{Q}(X_\mathbf{v}~|~\mathbf{v}\in W_2)=R_{\mathcal{O},\,2}$.
\item[\textup{(iii)}] If $D_\mathcal{O}\equiv1\Mod{4}$ or $N\geq3$, then $\mathbb{Q}(X_\mathbf{v}~|~\mathbf{v}\in W_N)
=K_{\mathcal{O},\,N}$.
\end{enumerate}
\end{proposition}
\begin{proof}
\begin{enumerate}
\item[(i)]
We see that
\begin{eqnarray*}
R_{\mathcal{O},\,N}&=&
\mathbb{Q}\left(j(\tau_\mathcal{O}),\,
X_{\left[\begin{smallmatrix}0&\frac{1}{N}\end{smallmatrix}\right]}\right)
\quad\textrm{by Lemma \ref{maximalreal}}\\
&\subseteq&\mathbb{Q}(X_\mathbf{v}~|~\mathbf{v}\in W_N)\quad\textrm{by
Lemma \ref{QXQj}}\\
&\subseteq& K_{\mathcal{O},\,N}\quad\textrm{by (\ref{specialization}) and (\ref{Xf})}\\
&=&KR_{\mathcal{O},\,N}\quad\textrm{by Lemma \ref{maximalreal}}.
\end{eqnarray*}
Thus it yields that
\begin{equation*}
K(X_\mathbf{v}~|~\mathbf{v}\in W_N)=K_{\mathcal{O},\,N}.
\end{equation*}
\item[(ii)] If $D_\mathcal{O}\equiv0\Mod{4}$, then we claim that for $\mathbf{v}\in W_2$
\begin{eqnarray*}
\overline{X_\mathbf{v}}&=&X_{\mathbf{v}\left[\begin{smallmatrix}
1&\phantom{-}0\\0&-1
\end{smallmatrix}\right]}\quad\textrm{by Remark \ref{Xcomplex}}\\
&=&X_\mathbf{v}\quad\textrm{by Lemma \ref{basicXY} (i)}.
\end{eqnarray*}
This implies by (i) that $\mathbb{Q}(X_\mathbf{v}~|~\mathbf{v}\in W_2)=R_{\mathcal{O},\,2}$.
\item[(iii)] If $D_\mathcal{O}\equiv1\Mod{4}$, then we establish that
\begin{eqnarray*}
\overline{X_{\left[\begin{smallmatrix}
\frac{1}{N}&0
\end{smallmatrix}\right]}}
&=&
X_{\left[\begin{smallmatrix}
\frac{1}{N}&0
\end{smallmatrix}\right]\left[\begin{smallmatrix}
1&\phantom{-}1\\0&-1
\end{smallmatrix}\right]}(\tau_\mathcal{O})\quad\textrm{by Remark \ref{Xcomplex}}\\
&=&X_{\left[\begin{smallmatrix}
\frac{1}{N}&\frac{1}{N}\end{smallmatrix}\right]}\\
&\neq&X_{\left[\begin{smallmatrix}
\frac{1}{N}&0
\end{smallmatrix}\right]}\quad\textrm{by Lemma \ref{basicXY} (i)}.
\end{eqnarray*}
Similarly, if $D_\mathcal{O}\equiv0\Mod{4}$ and $N\geq3$, then
\begin{equation*}
\overline{X_{\left[\begin{smallmatrix}
\frac{1}{N}&\frac{1}{N}
\end{smallmatrix}\right]}}
=X_{\left[\begin{smallmatrix}
\frac{1}{N}&\frac{1}{N}
\end{smallmatrix}\right]\left[\begin{smallmatrix}
1&\phantom{-}0\\0&-1
\end{smallmatrix}\right]}(\tau_\mathcal{O})
=X_{\left[\begin{smallmatrix}
\frac{1}{N}&-\frac{1}{N}\end{smallmatrix}\right]}
\neq X_{\left[\begin{smallmatrix}
\frac{1}{N}&\frac{1}{N}
\end{smallmatrix}\right]}.
\end{equation*}
These observations hold that if $D_\mathcal{O}\equiv1\Mod{4}$ or $N\geq3$, then
\begin{equation*}
\mathbb{Q}(X_\mathbf{v}~|~\mathbf{v}\in W_N)\not\subseteq\mathbb{R}.
\end{equation*}
Therefore we conclude by (i) and Lemma \ref{maximalreal} that
$\mathbb{Q}(X_\mathbf{v}~|~\mathbf{v}\in W_N)
=K_{\mathcal{O},\,N}$.
\end{enumerate}
\end{proof}

\begin{lemma}\label{Yratio}
If $\mathbf{u},\,\mathbf{v}\in M_{1,\,2}(\mathbb{Q})\setminus M_{1,\,2}(\mathbb{Z})$
satisfy
\begin{equation*}
2\mathbf{u}\not\in M_{1,\,2}(\mathbb{Z})\quad\textrm{and}\quad
N\mathbf{u},\,N\mathbf{v}\in M_{1,\,2}(\mathbb{Z}),
\end{equation*}
then the ratio $\displaystyle\frac{Y_\mathbf{v}}{Y_\mathbf{u}}$ lies in $K_{\mathcal{O},\,N}$.
\end{lemma}
\begin{proof}
See \cite[Proof of Lemma 5.3]{J-K-S-Y} and (\ref{specialization}).
\end{proof}

\begin{remark}
\begin{enumerate}
\item[(i)] If $2\mathbf{u}\in M_{1,\,2}(\mathbb{Z})$, then
$Y_\mathbf{u}=0$ by Lemma \ref{basicXY} (iv).
\item[(ii)] By \cite[Lemma 3.3]{J-K-S-Y}, we have
\begin{equation*}
\frac{Y_\mathbf{v}}{Y_\mathbf{u}}=\frac{g_{2\mathbf{v}}(\tau_\mathcal{O})
g_{\mathbf{u}}(\tau_\mathcal{O})^4}{g_{\mathbf{v}}(\tau_\mathcal{O})^4
g_{2\mathbf{u}}(\tau_\mathcal{O})}.
\end{equation*}
\end{enumerate}
\end{remark}

\begin{theorem}\label{maingenerate}
Assume that $D_\mathcal{O}\neq-3,\,-4$.
\begin{enumerate}
\item[\textup{(i)}]
If $D_\mathcal{O}\equiv0\Mod{4}$, then
$\mathbb{Q}(E_\mathcal{O}[2])$ is the maximal real subfield of $K_{\mathcal{O},\,2}$.
\item[\textup{(ii)}] If $D_\mathcal{O}\equiv1\Mod{4}$, then
$\mathbb{Q}(E_\mathcal{O}[2])=K_{\mathcal{O},\,2}$.
\item[\textup{(iii)}] If $N\geq3$, then
$\mathbb{Q}(E_\mathcal{O}[N])$ is an extension field of $K_{\mathcal{O},\,N}$
of degree at most $2$.
\end{enumerate}
\end{theorem}
\begin{proof}
We get by Lemma \ref{basicXY} (iv) that
\begin{equation}\label{QE2}
\mathbb{Q}(E_\mathcal{O}[2])=\mathbb{Q}(X_\mathbf{v}~|~\mathbf{v}\in W_2).
\end{equation}
\begin{enumerate}
\item[(i)] If $D_\mathcal{O}\equiv0\Mod{4}$, then
we obtain by (\ref{QE2}) and Proposition \ref{Xgenerate} that
$\mathbb{Q}(E_\mathcal{O}[2])=R_{\mathcal{O},\,2}$.
\item[(ii)] If $D_\mathcal{O}\equiv1\Mod{4}$, then we derive by (\ref{QE2})
and Proposition \ref{Xgenerate} (iii) that
$\mathbb{Q}(E_\mathcal{O}[2])=K_{\mathcal{O},\,2}$.
\item[(iii)] If $N\geq3$, then we find that
\begin{eqnarray*}
\mathbb{Q}(E_\mathcal{O}[N])&=&K_{\mathcal{O},\,N}(Y_\mathbf{v}~|~\mathbf{v}\in
W_N)\quad\textrm{by Proposition \ref{Xgenerate} (iii)}\\
&=&K_{\mathcal{O},\,N}\left(Y_{\left[\begin{smallmatrix}0&\frac{1}{N}\end{smallmatrix}\right]}
\cdot\displaystyle\frac{Y_\mathbf{v}}{Y_{\left[\begin{smallmatrix}0&\frac{1}{N}\end{smallmatrix}\right]}}~|~\mathbf{v}\in W_N\right)\\
&=&K_{\mathcal{O},\,N}\left(
Y_{\left[\begin{smallmatrix}0&\frac{1}{N}\end{smallmatrix}\right]}
\right)\quad\textrm{by Lemma \ref{Yratio}}.
\end{eqnarray*}
Here we observe by Proposition \ref{CM} that
\begin{equation}\label{YXin}
Y_{\left[\begin{smallmatrix}0&\frac{1}{N}\end{smallmatrix}\right]}^2
=4X_{\left[\begin{smallmatrix}0&\frac{1}{N}\end{smallmatrix}\right]}^3-
A_\mathcal{O}X_{\left[\begin{smallmatrix}0&\frac{1}{N}\end{smallmatrix}\right]}-B_\mathcal{O}
\in K_{\mathcal{O},\,N}.
\end{equation}
Therefore $\mathbb{Q}(E_\mathcal{O}[N])$ is an extension field of $K_{\mathcal{O},\,N}$
of degree at most $2$.
\end{enumerate}
\end{proof}

\section {Galois representations attached to $E_\mathcal{O}$}\label{sect8}

Since the elliptic curve $E_\mathcal{O}$ is defined over $\mathbb{Q}(j(E_\mathcal{O}))$, the field
$\mathbb{Q}(E_\mathcal{O}[N])$ is a finite Galois
extension of $\mathbb{Q}(j(E_\mathcal{O}))$ by Lemma \ref{QXQj} and \cite[pp. 53--54]{Silverman09}.
So we get the right action of
the Galois group $\mathrm{Gal}(\mathbb{Q}(E_\mathcal{O}[N])/
\mathbb{Q}(j(E_\mathcal{O})))$ on
the $\mathbb{Z}/N\mathbb{Z}$-module $E_\mathcal{O}[N]$.
This action gives us the faithful representation
\begin{equation*}
\rho_{\mathcal{O},\,N}:\mathrm{Gal}(\mathbb{Q}(E_\mathcal{O}[N])/
\mathbb{Q}(j(E_\mathcal{O})))\rightarrow\mathrm{GL}_2(\mathbb{Z}/N\mathbb{Z})~(\simeq
\mathrm{Aut}(E_\mathcal{O}[N]))
\end{equation*}
with respect to the ordered basis
\begin{equation*}
\left\{
\left[X_{\left[\begin{smallmatrix}
\frac{1}{N} & 0\end{smallmatrix}\right]}:
Y_{\left[\begin{smallmatrix}
\frac{1}{N} & 0\end{smallmatrix}\right]}:1
\right],~
\left[X_{\left[\begin{smallmatrix}
0& \frac{1}{N}\end{smallmatrix}\right]}:
Y_{\left[\begin{smallmatrix}
0 & \frac{1}{N}\end{smallmatrix}\right]}:1
\right]
\right\}
\end{equation*}
for $E_\mathcal{O}[N]$
so that
\begin{equation}\label{XsX}
[X_\mathbf{v}:Y_\mathbf{v}:1]^\sigma
=[X_{\mathbf{v}\rho_{\mathcal{O},\,N}(\sigma)}:
Y_{\mathbf{v}\rho_{\mathcal{O},\,N}(\sigma)}:1]
\quad\left(\mathbf{v}\in W_N\right).
\end{equation}
In this section, we shall
determine the image of $\rho_{\mathcal{O},\,N}$
by utilizing the Shimura reciprocity law.
\par
If we let
\begin{equation*}
\min(\tau_\mathcal{O},\,\mathbb{Q})=x^2+b_\mathcal{O}x+c_\mathcal{O}
\quad
(\in\mathbb{Z}[x]),
\end{equation*}
then we have a well-defined homomorphism of groups
\begin{eqnarray*}
\mu_{\mathcal{O},\,N}~:~(\mathcal{O}/N\mathcal{O})^* & \rightarrow &
\mathrm{GL}_2(\mathbb{Z}/N\mathbb{Z})\\
\mathrm{[}s\tau_\mathcal{O}+t] & \mapsto & \begin{bmatrix}
t-b_\mathcal{O}s & -c_\mathcal{O}s\\
s & t
\end{bmatrix}.
\end{eqnarray*}
Let
\begin{equation}\label{WU}
W_{\mathcal{O},\,N}=\mu_{\mathcal{O},\,N}((\mathcal{O}/N\mathcal{O})^*)\quad\textrm{and}\quad
U_{\mathcal{O},\,N}=\mu_{\mathcal{O},\,N}(\pi_{\mathcal{O},\,N}(\mathcal{O}^*))
\end{equation}
where $\pi_{\mathcal{O},\,N}:\mathcal{O}\rightarrow\mathcal{O}/N\mathcal{O}$
is the canonical homomorphism.
Then we achieve that
\begin{equation}\label{WUOG}
W_{\mathcal{O},\,N}/U_{\mathcal{O},\,N}
\simeq(\mathcal{O}/N\mathcal{O})^*/\pi_{\mathcal{O},\,N}(\mathcal{O}^*)
\simeq\mathrm{Gal}(K_{\mathcal{O},\,N}/H_\mathcal{O})
\end{equation}
(cf. \cite[Lemma 15.17]{Cox}).

\begin{proposition}[The Shimura reciprocity law]\label{reciprocity}
The map
\begin{eqnarray*}
W_{\mathcal{O},\,N}/U_{\mathcal{O},\,N} & \rightarrow & \mathrm{Gal}(K_{\mathcal{O},\,N}/H_\mathcal{O})\\
\mathrm{[}\gamma] & \mapsto &
\left(f(\tau_\mathcal{O})\mapsto f^{\widetilde{\gamma}}(\tau_\mathcal{O})~
|~f\in\mathcal{F}_N~\textrm{is finite at $\tau_\mathcal{O}$}
\right)
\end{eqnarray*}
is a well-defined isomorphism.
Here, $\widetilde{\gamma}$ means the image of $\gamma$ in
$\mathrm{GL}_2(\mathbb{Z}/N\mathbb{Z})/\langle-I_2\rangle$
\textup{($\simeq\mathrm{Gal}(\mathcal{F}_N/\mathcal{F}_1)$)}.
\end{proposition}
\begin{proof}
See \cite[p. 859]{Cho} or \cite[Theorem 15.22]{Cox}.
\end{proof}

\begin{lemma}\label{betaI}
Let $\beta\in M_2(\mathbb{Z})$ such that $\gcd(\det(\beta),\,N)=1$.
If
\begin{equation}\label{beta}
\mathbf{v}\beta\equiv\mathbf{v}~\textrm{or}~-\mathbf{v}\Mod{M_{1,\,2}(\mathbb{Z})}\quad
\textrm{for each}~\mathbf{v}\in W_N,
\end{equation}
then $\beta\equiv I_2$ or $-I_2\Mod{NM_2(\mathbb{Z})}$.
\end{lemma}
\begin{proof}
See \cite[Lemma 6.1]{J-K-S-Y}.
\end{proof}

\begin{theorem}\label{Galoisrepresentation}
Assume that $D_\mathcal{O}\neq-3,\,-4$ and $N\geq2$.
Let $\widehat{W}_{\mathcal{O},\,N}$ be the subgroup
of $\mathrm{GL}_2(\mathbb{Z}/N\mathbb{Z})$ defined by
\begin{equation*}
\widehat{W}_{\mathcal{O},\,N}=\left\langle
W_{\mathcal{O},\,N},\,\begin{bmatrix}1&~~b_\mathcal{O}\\0&-1\end{bmatrix}\right\rangle.
\end{equation*}
\begin{enumerate}
\item[\textup{(i)}] The image of the
Galois representation
$\rho_{\mathcal{O},\,N}$
is a subgroup of $\widehat{W}_{\mathcal{O},\,N}$
of index at most $2$.
\item[\textup{(ii)}] If $-1$ is a quadratic residue modulo $N$, then
the image of $\rho_{\mathcal{O},\,N}$ is exactly
$\widehat{W}_{\mathcal{O},\,N}$.
\end{enumerate}
\end{theorem}
\begin{proof}
\begin{enumerate}
\item[(i)] Let $\sigma\in\mathrm{Gal}(\mathbb{Q}(E_\mathcal{O}[N])/\mathbb{Q}(j(E_\mathcal{O})))$
and $\gamma=\rho_{\mathcal{O},\,N}(\sigma)$. We get by (\ref{XsX}) that
\begin{equation}\label{XsXvg}
X_\mathbf{v}^\sigma=X_{\mathbf{v}\gamma}
\quad(\mathbf{v}\in W_N).
\end{equation}
Recall that
$j(E_\mathcal{O})=j(\tau_\mathcal{O})\in\mathbb{R}$ and $K(j(E_\mathcal{O}))=H_\mathcal{O}$
by Proposition \ref{CM} (i).
Moreover, since $K$ is imaginary quadratic, we have
\begin{equation}\label{smmc}
\sigma=
\mu|_{\mathbb{Q}(E_\mathcal{O}[N])}
~\textrm{or}~
\mathfrak{c}|_{\mathbb{Q}(E_\mathcal{O}[N])}
\mu|_{\mathbb{Q}(E_\mathcal{O}[N])}
\quad\textrm{for some}~\mu\in
\mathrm{Gal}(K(E_\mathcal{O}[N])/H_\mathcal{O})
\end{equation}
where $\mathfrak{c}$ is the complex conjugation on $\mathbb{C}$.
Then we see that
\begin{eqnarray*}
X_\mathbf{v}^\sigma&=&-\frac{1}{2^73^3}f_\mathbf{v}(\tau_\mathcal{O})^\sigma\quad\textrm{by
(\ref{Xf})}\\
&=&
-\frac{1}{2^73^3}f_{\mathbf{v}\alpha}(\tau_\mathcal{O})
\quad\textrm{for some}~\alpha\in\widehat{W}_{\mathcal{O},\,N}\\
&&\hspace{2.8cm}\textrm{by (\ref{smmc}), Propositions \ref{Frickefamily}, \ref{reciprocity} and Lemma \ref{complex1}}\\
&=&X_{\mathbf{v}\alpha}\quad\textrm{again by (\ref{Xf})}.
\end{eqnarray*}
And, it follows from (\ref{XsXvg}) that
\begin{equation*}
X_{\mathbf{v}\gamma}=X_{\mathbf{v}\alpha}\quad(\mathbf{v}\in W_N)
\end{equation*}
and so
\begin{equation*}
\mathbf{v}\gamma\equiv\mathbf{v}\alpha~
\textrm{or}~-\mathbf{v}\alpha\Mod{M_{1,\,2}(\mathbb{Z})}\quad(\mathbf{v}\in W_N)
\end{equation*}
by Lemma \ref{basicXY} (i).
Thus we obtain by Lemma \ref{betaI} that
\begin{equation*}
\gamma=\alpha~\textrm{or}~-\alpha~\textrm{in}~\mathrm{GL}_2(\mathbb{Z}/N\mathbb{Z}).
\end{equation*}
This observation asserts that the image of
$\rho_{\mathcal{O},\,N}$ is a subgroup of $\widehat{W}_{K,\,N}$.
Furthermore, we find that
\begin{eqnarray*}
&&|\rho_{\mathcal{O},\,N}(\mathrm{Gal}(\mathbb{Q}(E_\mathcal{O}[N])/\mathbb{Q}(j(E_\mathcal{O}))))|\\&=&
[\mathbb{Q}(E_\mathcal{O}[N]):\mathbb{Q}(j(E_\mathcal{O}))]
\quad\textrm{since $\rho_{\mathcal{O},\,N}$ is injective}\\
&=&\left\{\begin{array}{ll}
\phantom{2}[K_{\mathcal{O},\,N}:H_\mathcal{O}] & \textrm{if}~N=2~\textrm{and}~D_\mathcal{O}\equiv0\Mod{4},\\
2[K_{\mathcal{O},\,N}:H_\mathcal{O}] & \textrm{if}~N=2~\textrm{and}~D_\mathcal{O}\equiv1\Mod{4},\\
2[K_{\mathcal{O},\,N}:H_\mathcal{O}] ~\textrm{or}~
4[K_{\mathcal{O},\,N}:H_\mathcal{O}] & \textrm{if}~N\geq3
\end{array}\right.\\
&&\hspace{3cm}\textrm{by Proposition \ref{CM} (i), the fact $j(E_\mathcal{O})\in\mathbb{R}$ and Theorem \ref{maingenerate}}\\
&=&
|W_{\mathcal{O},\,N}/\{I_2,\,-I_2\}|\times
\left\{\begin{array}{ll}
1& \textrm{if}~N=2~\textrm{and}~D_\mathcal{O}\equiv0\Mod{4},\\
2 & \textrm{if}~N=2~\textrm{and}~D_\mathcal{O}\equiv1\Mod{4},\\
2~\textrm{or}~4& \textrm{if}~N\geq3,
\end{array}\right.\\
&&\hspace{3cm}\textrm{by Proposition \ref{reciprocity} because the assumption $D_\mathcal{O}\neq-3,\,-4$ yields}\\
&&\hspace{3cm}\textrm{that $\mathcal{O}^*=\{1,\,-1\}$ (cf. \cite[Exercise 5.9]{Cox})}\\
&=&
|W_{\mathcal{O},\,N}|\times
\left\{\begin{array}{ll}
1& \textrm{if}~N=2~\textrm{and}~D_\mathcal{O}\equiv0\Mod{4},\\
2 & \textrm{if}~N=2~\textrm{and}~D_\mathcal{O}\equiv1\Mod{4},\\
1~\textrm{or}~2& \textrm{if}~N\geq3\\
\end{array}\right.\\
&=&|\widehat{W}_{\mathcal{O},\,N}|\times
\left\{\begin{array}{ll}
1& \textrm{if}~N=2,\\
\frac{1}{2}~\textrm{or}~1& \textrm{if}~N\geq3.
\end{array}\right.
\end{eqnarray*}
Hence the image of $\rho_{\mathcal{O},\,N}$ is a subgroup of $\widehat{W}_{\mathcal{O},\,N}$
of index at most $2$.
\item[(ii)] Let $R$ be the image of $\rho_{\mathcal{O},\,N}$.
Suppose on the contrary that $R\neq\widehat{W}_{\mathcal{O},\,N}$.
Then we derive by (i) and its proof that
\begin{equation}\label{index2}
[\widehat{W}_{\mathcal{O},\,N}:R]=2
\end{equation}
and
\begin{equation}\label{N3QK}
N\geq3\quad\textrm{and}\quad
\mathbb{Q}(E_\mathcal{O}[N])=K_{\mathcal{O},\,N}.
\end{equation}
If $R$ contains $-I_2$ and so
\begin{equation}\label{-I}
-I_2=\rho_{\mathcal{O},\,N}(\sigma)
\quad\textrm{for some}~\sigma\in\mathrm{Gal}(\mathbb{Q}(E_\mathcal{O}[N])/
\mathbb{Q}(j(E_\mathcal{O}))),
\end{equation}
then we see that for $\mathbf{v}\in W_N$
\begin{eqnarray*}
X_\mathbf{v}^\sigma&=&X_{\mathbf{v}(-I_2)}\quad\textrm{by (\ref{XsX})}\\
&=&X_\mathbf{v}\quad\textrm{by Lemma \ref{basicXY} (i)}.
\end{eqnarray*}
This claims by Proposition \ref{Xgenerate} (iii) that $\sigma=\mathrm{id}_{K_{\mathcal{O},\,N}}$, which contradicts (\ref{N3QK}) and (\ref{-I}).
Thus we should have
\begin{equation}\label{Rnot-I}
R\not\ni-I_2,
\end{equation}
and hence
\begin{equation}\label{R-I}
\langle R,\,-I_2\rangle=
\widehat{W}_{\mathcal{O},\,N}
\end{equation}
by (\ref{index2}).
Let $t$ be an integer such that $t^2\equiv-1\Mod{N}$. Since
$t+N\mathcal{O}\in(\mathcal{O}/N\mathcal{O})^*$,
\begin{equation}\label{Wcontain}
W_{\mathcal{O},\,N}\ni\mu_{\mathcal{O},\,N}([t])=tI_2\quad\textrm{and}\quad
W_{\mathcal{O},\,N}\ni\mu_{\mathcal{O},\,N}([-t])=-tI_2,
\end{equation}
$R$ contains at least one of $tI_2$ or $-tI_2$ by (\ref{R-I}). And it follows that
\begin{equation*}
R\ni(\pm tI_2)^2=t^2I_2=-I_2\quad\textrm{in}~\mathrm{GL}_2(\mathbb{Z}/N\mathbb{Z}),
\end{equation*}
which contradicts (\ref{Rnot-I}).
\par
Therefore we conclude that if $-1$ is a quadratic residue modulo $N$, then the image of
$\rho_{\mathcal{O},\,N}$ is the whole of $\widehat{W}_{\mathcal{O},\,N}$.
\end{enumerate}
\end{proof}

\section {An isomorphism of $\mathcal{C}_N(D_\mathcal{O})$ onto
$\mathcal{C}_N(\mathcal{O})$}\label{sect9}

By constructing a bijection between $\mathcal{C}_N(D_\mathcal{O})$ and
$\mathcal{C}_N(\mathcal{O})$,
we shall simultaneously prove that
the binary operation on $\mathcal{C}_N(D_\mathcal{O})$ given in Definition \ref{binaryoperation}
is well defined
and $\mathcal{C}_N(D_\mathcal{O})$ is isomorphic to $\mathcal{C}_N(\mathcal{O})$.

\begin{lemma}\label{QQwI}
If $Q=ax^2+bxy+cy^2\in\mathcal{Q}(D_\mathcal{O})$, then we have
\begin{equation*}
Q\in\mathcal{Q}(D_\mathcal{O},\,N)\quad\Longleftrightarrow\quad
[\omega_Q,\,1]\in I(\mathcal{O},\,N).
\end{equation*}
\end{lemma}
\begin{proof}
Let $\mathfrak{a}=[\omega_Q,\,1]$.
Observe by Lemma \ref{QaN} that $a\mathfrak{a}\in\mathcal{M}(\mathcal{O})$ and
$\mathrm{N}_\mathcal{O}(a\mathfrak{a})=a$.
\par
Assume that $Q\in\mathcal{Q}(D_\mathcal{O},\,N)$. Since
$\gcd(\mathrm{N}_\mathcal{O}(a\mathfrak{a}),\,N)=\gcd(a,\,N)=1$,
$a\mathfrak{a}$ lies in $\mathcal{M}(\mathcal{O},\,N)$
by Lemma \ref{primelemma}.
Therefore $\mathfrak{a}=(a\mathfrak{a})(a\mathcal{O})^{-1}$ belongs to $I(\mathcal{O},\,N)$.
\par
Conversely, assume that $\mathfrak{a}\in I(\mathcal{O},\,N)$. Then we attain
$\mathfrak{a}=\mathfrak{b}\mathfrak{c}^{-1}$ for some $\mathfrak{b},\,\mathfrak{c}\in\mathcal{M}(\mathcal{O},\,N)$
and so
\begin{equation*}
(a\mathfrak{a})\mathfrak{c}=(a\mathcal{O})\mathfrak{b}.
\end{equation*}
Taking norm on both sides, we find that
$a\mathrm{N}_\mathcal{O}(\mathfrak{c})=a^2\mathrm{N}_\mathcal{O}(\mathfrak{b})$
by the fact $\mathrm{N}_\mathcal{O}(a\mathfrak{a})=a$ and Lemma \ref{normbasic}, and hence
\begin{equation}\label{NcaNb}
\mathrm{N}_\mathcal{O}(\mathfrak{c})=a\mathrm{N}_\mathcal{O}(\mathfrak{b}).
\end{equation}
Since
$\gcd(\mathrm{N}_\mathcal{O}(\mathfrak{c}),\,N)=1$ by
the fact $\mathfrak{c}\in\mathcal{M}(\mathcal{O},\,N)$ and
Lemma \ref{primelemma}, we achieve from (\ref{NcaNb}) that
$\gcd(a,\,N)=1$. Thus $Q$ belongs to $\mathcal{Q}(D_\mathcal{O},\,N)$.
\end{proof}

\begin{lemma}\label{anotherP}
The set $P_N(\mathcal{O})$ coincides with
\begin{equation*}
P=\left\{
\frac{\nu_1}{\nu_2}\mathcal{O}~|~\nu_1,\,
\nu_2\in\mathcal{O}\setminus\{0\}~\textrm{satisfy}~
\nu_1\mathcal{O},\,\nu_2\mathcal{O}\in P(\mathcal{O},\,N)~\textrm{and}~
\nu_1\equiv\nu_2\Mod{N\mathcal{O}}\right\}.
\end{equation*}
\end{lemma}
\begin{proof}
By the definition (\ref{PNO})
of $P_N(\mathcal{O})$, we deduce the inclusion $P_N(\mathcal{O})\subseteq P$.
\par
Now, let $\mathfrak{a}\in P$ and so
\begin{equation*}
\mathfrak{a}=
\frac{\nu_1}{\nu_2}\mathcal{O}
\quad\textrm{for some}~\nu_1,\,
\nu_2\in\mathcal{O}\setminus\{0\}~\textrm{such that}~
\left\{\begin{array}{l}
\nu_1\mathcal{O},\,\nu_2\mathcal{O}\in P(\mathcal{O},\,N),\\
\nu_1\equiv\nu_2\Mod{N\mathcal{O}}.
\end{array}\right.
\end{equation*}
Since $\nu_i\mathcal{O}$ ($i=1,\,2$) is prime to $N$, that is,
$\nu_i\mathcal{O}+N\mathcal{O}=\mathcal{O}$,
the coset $\nu_i+N\mathcal{O}$ in the quotient ring $\mathcal{O}/N\mathcal{O}$
is a unit.
If we let $m$ be the order of the unit group $(\mathcal{O}/N\mathcal{O})^*$, then
we get from the fact $\nu_1\equiv\nu_2\Mod{N\mathcal{O}}$ that
\begin{equation*}
\nu_1\nu_2^{m-1}\equiv\nu_2^m\equiv1\Mod{N\mathcal{O}}.
\end{equation*}
And we obtain that
\begin{equation*}
\mathfrak{a}=\frac{\nu_1}{\nu_2}\mathcal{O}=(\nu_1\nu_2^{m-1}\mathcal{O})
(\nu_2^m\mathcal{O})^{-1}\in P_N(\mathcal{O}),
\end{equation*}
which proves the converse inclusion $P\subseteq P_N(\mathcal{O})$.
Therefore we conclude that $P_N(\mathcal{O})=P$.
\end{proof}

\begin{proposition}\label{phibijective}
The map
\begin{eqnarray*}
\phi_{\mathcal{O},\,N}~:~\mathcal{C}_N(D_\mathcal{O})&\rightarrow&\mathcal{C}_N(\mathcal{O})\\
\mathrm{[}Q]~~ & \mapsto & [[\omega_Q,\,1]]
\end{eqnarray*}
is a well-defined bijection.
\end{proposition}
\begin{proof}
First, we shall show that $\phi_{\mathcal{O},\,N}$ is well defined.
Let $Q=ax^2+bxy+cy^2\in\mathcal{Q}(D_\mathcal{O},\,N)$ and $\gamma\in\Gamma_1(N)$
with $\gamma^{-1}=\begin{bmatrix}p&q\\r&s\end{bmatrix}$ ($\in\Gamma_1(N)$).
Then we see that
\begin{equation}\label{wQg}
[\omega_{Q^\gamma},\,1]=[\gamma^{-1}(\omega_Q),\,1]=
\frac{1}{j(\gamma^{-1},\,\omega_Q)}[\omega_Q,\,1].
\end{equation}
Since $\gcd(a,\,N)=1$, we have
\begin{equation}\label{ap1}
a^{\varphi(N)}\equiv1\Mod{N}
\end{equation}
where $\varphi$ is the Euler totient function.
Furthermore, since
\begin{equation*}
b^2-4ac=D_\mathcal{O}=b_\mathcal{O}^2-4c_\mathcal{O},
\end{equation*}
$b$ and $b_\mathcal{O}$ have the same parity.
So we find that
\begin{equation}\label{aj}
a^{\varphi(N)}\cdot j(\gamma^{-1},\,\omega_Q)
=ra^{\varphi(N)-1}\left(\tau_\mathcal{O}+\frac{b_\mathcal{O}-b}{2}\right)+sa^{\varphi(N)}
\equiv1\Mod{N\mathcal{O}}
\end{equation}
because $r\equiv0,\,s\equiv1\Mod{N}$.
Hence we obtain by (\ref{wQg}), (\ref{ap1}) and (\ref{aj}) that
\begin{equation*}
[[\omega_{Q^\gamma},\,1]]=
[[\omega_Q,\,1]]\quad\textrm{in}~\mathcal{C}_N(\mathcal{O}),
\end{equation*}
which shows that $\phi_{\mathcal{O},\,N}$ is well defined.
\par
Second, we shall prove that $\phi_{\mathcal{O},\,N}$ is surjective.
Let $h$ be the order of the group $\mathcal{C}(\mathcal{O},\,N)$.
By Proposition \ref{Dirichlet} (ii) and Lemmas \ref{COlCO1}, \ref{Qgprime}, \ref{QQwI}, we derive that
\begin{equation*}
\mathcal{C}(\mathcal{O},\,N)=I(\mathcal{O},\,N)/P(\mathcal{O},\,N)=
\{[\omega_{Q_1},\,1]P(\mathcal{O},\,N),\,
[\omega_{Q_2},\,1]P(\mathcal{O},\,N),\,
\ldots,\,
[\omega_{Q_h},\,1]P(\mathcal{O},\,N)\}
\end{equation*}
for some $Q_i=a_ix^2+b_ixy+c_iy^2\in\mathcal{Q}(D_\mathcal{O},\,N)$ ($i=1,\,2,\,\ldots,\,h$).
Thus we deduce the decomposition
\begin{equation*}
\mathcal{C}_N(\mathcal{O})=
I(\mathcal{O},\,N)/P_N(\mathcal{O})=(P(\mathcal{O},\,N)/P_N(\mathcal{O}))\cdot
\{[\omega_{Q_i},\,1]P_N(\mathcal{O})~|~i=1,\,2,\,\ldots,\,h\}.
\end{equation*}
Now, let $C\in\mathcal{C}_N(\mathcal{O})$. By the above decomposition, we have
\begin{equation*}
C=C'\cdot[\omega_{Q_i},\,1]P_N(\mathcal{O})\quad\textrm{for some $C'\in P(\mathcal{O},\,N)/P_N(\mathcal{O})$ and $1\leq i\leq h$}.
\end{equation*}
And one can take an $\mathcal{O}$-ideal $\mathfrak{d}$ in $C'^{-1}\cap\mathcal{M}(\mathcal{O},\,N)$
by Remark \ref{CCOMN}. Since $\mathcal{O}=[a_i\omega_{Q_i},\,1]$, we achieve
\begin{equation*}
\mathfrak{d}=(ka_i\omega_{Q_i}+v)\mathcal{O}\quad\textrm{for some}~k,\,v\in\mathbb{Z}
\end{equation*}
and so
\begin{equation}\label{C1w}
C=\left[\frac{1}{u\omega_{Q_i}+v}\,[\omega_{Q_i},\,1]\right]\quad\textrm{with}~u=ka_i.
\end{equation}
Since $\gcd(u,\,v,\,N)=1$
by the facts $\mathfrak{d}\in\mathcal{M}(\mathcal{O},\,N)$
and $\gcd(a_i,\,N)=1$, we can take a matrix
$\sigma=\begin{bmatrix}
\mathrm{*} & \mathrm{*}\\
\widetilde{u}&\widetilde{v}
\end{bmatrix}$ in $\mathrm{SL}_2(\mathbb{Z})$ such that
$\widetilde{u}\equiv u,\,\widetilde{v}\equiv v\Mod{N}$.
Then we see that
\begin{equation}\label{belongP}
\frac{u\omega_{Q_i}+v}
{\widetilde{u}\omega_{Q_i}+\widetilde{v}}\,\mathcal{O}=\frac{u(a_i\omega_{Q_i})+a_iv}
{\widetilde{u}(a_i\omega_{Q_i})+a_i\widetilde{v}}\,\mathcal{O}\in P_N(\mathcal{O})
\end{equation}
by the fact $a_i\omega_{Q_i}\in\mathcal{O}$ and Lemma \ref{anotherP}.
And we establish that
\begin{eqnarray*}
C&=&\left[\frac{u\omega_{Q_i}+v}
{\widetilde{u}\omega_{Q_i}+\widetilde{v}}\,\mathcal{O}\right]C\quad\textrm{by (\ref{belongP})}\\
&=&\left[
\frac{1}{j(\sigma,\,\omega_{Q_i})}[\omega_{Q_i},\,1]\right]\quad\textrm{by (\ref{C1w})}\\
&=&\left[
[\sigma(\omega_{Q_i}),\,1]\right]\\
&=&[[\omega_{Q_i^{\sigma^{-1}}},\,1]].
\end{eqnarray*}
Thus, if we let $Q=Q_i^{\sigma^{-1}}$, then we get by Lemma \ref{QQwI} that
\begin{equation*}
Q\in\mathcal{Q}(D_\mathcal{O},\,N)\quad\textrm{and}\quad\phi_{\mathcal{O},\,N}([Q])=C.
\end{equation*}
This proves that $\phi_{\mathcal{O},\,N}$ is surjective.
\par
Third, we shall show that $\phi_{\mathcal{O},\,N}$ is injective.
Suppose that
\begin{equation*}
\phi_{\mathcal{O},\,N}([Q_1])
=\phi_{\mathcal{O},\,N}([Q_2])
\quad\textrm{for some}~Q_i=a_ix^2+b_ixy+c_iy^2\in\mathcal{Q}(D_\mathcal{O},\,N)~(i=1,\,2)
\end{equation*}
Then we have
\begin{equation}\label{wnw}
[\omega_{Q_1},\,1]=\frac{\nu_1}{\nu_2}[\omega_{Q_2},\,1]
\quad\textrm{for some}~\nu_1,\,\nu_2\in\mathcal{O}\setminus\{0\}~
\textrm{such that}~
\nu_1\equiv\nu_2\equiv1\Mod{N\mathcal{O}},
\end{equation}
which implies by Proposition \ref{Dirichlet} (ii) that
\begin{equation*}
Q_1=Q_2^\gamma\quad\textrm{for some}~\gamma\in\mathrm{SL}_2(\mathbb{Z}).
\end{equation*}
And it follows from (\ref{wnw}) that
\begin{equation*}
[\omega_{Q_1},\,1]=\frac{\nu_1}{\nu_2}[\gamma(\omega_{Q_1}),\,1]=\frac{\nu_1}{\nu_2}\cdot\frac{1}{j}[\omega_{Q_1},\,1]\quad\textrm{where}~j=j(\gamma,\,\omega_{Q_1}),
\end{equation*}
and hence
\begin{equation}\label{znn1j}
\zeta:=
\frac{\nu_1}{\nu_2}\cdot\frac{1}{j}\in\mathcal{O}^*.
\end{equation}
Now, since
$[\omega_{Q_1},\,1]=\zeta j[\omega_{Q_2},\,1]$
by (\ref{wnw}) and $\omega_{Q_1},\,\omega_{Q_2}\in\mathbb{H}$, there is a matrix $\alpha=\begin{bmatrix}p&q\\r&s\end{bmatrix}$ in $\mathrm{SL}_2(\mathbb{Z})$ so that
\begin{equation}\label{pqrs}
\begin{bmatrix}\zeta j\omega_{Q_2}\\\zeta j\end{bmatrix}
=\begin{bmatrix}p&q\\r&s\end{bmatrix}\begin{bmatrix}\omega_{Q_1}\\1\end{bmatrix}.
\end{equation}
We then find by (\ref{pqrs}) that
\begin{equation*}
\omega_{Q_2}=\frac{\zeta j\omega_{Q_2}}{\zeta j}
=\frac{p\omega_{Q_1}+q}{r\omega_{Q_1}+s}=\alpha(\omega_{Q_1}),
\end{equation*}
which gives
\begin{equation}\label{QQa}
Q_1=Q_2^\alpha.
\end{equation}
We get again by (\ref{pqrs}) and the fact $\nu_2\equiv1\Mod{N\mathcal{O}}$ that
\begin{equation}\label{anzj}
a_1\nu_2(\zeta j)\equiv a_1(r\omega_{Q_1}+s)\equiv r(a_1\omega_{Q_1})+a_1s
\Mod{N\mathcal{O}}.
\end{equation}
On the other hand, we see by (\ref{znn1j}) and the fact $\nu_1\equiv1\Mod{N\mathcal{O}}$ that
\begin{equation}\label{anzj'}
a_1\nu_2(\zeta j)=
a_1\nu_1\equiv a_1\Mod{N\mathcal{O}}.
\end{equation}
Thus we obtain by (\ref{anzj}) and (\ref{anzj'}) that
\begin{equation*}
r(a_1\omega_{Q_1})+a_1s\equiv a_1\Mod{N\mathcal{O}},
\end{equation*}
from which it follows that
\begin{equation}\label{rawas10}
r(a_1\omega_{Q_1})+a_1(s-1)\equiv0\Mod{N\mathcal{O}}.
\end{equation}
Since $N\mathcal{O}=[N(a_1\omega_{Q_1}),\,N]$ and $\gcd(a_1,\,N)=1$, (\ref{rawas10}) implies that $r\equiv0,\,s\equiv1\Mod{N}$ and hence $\alpha\in\Gamma_1(N)$. Therefore we achieve by (\ref{QQa}) that
\begin{equation*}
[Q_1]=[Q_2]\quad\textrm{in}~\mathcal{C}_N(D_\mathcal{O}),
\end{equation*}
which proves the injectivity of $\phi_{\mathcal{O},\,N}$.
\end{proof}

\begin{theorem}\label{formideal}
The form class group $\mathcal{C}_N(D_\mathcal{O})$ is isomorphic to
the ideal class group $\mathcal{C}_N(\mathcal{O})$.
\end{theorem}
\begin{proof}
Let $\phi_{\mathcal{O},\,N}:\mathcal{C}_N(D_\mathcal{O})\rightarrow\mathcal{C}_N(\mathcal{O})$ be the bijection
given in Proposition \ref{phibijective}.
Let $C,\,C'\in\mathcal{C}_N(D_\mathcal{O})$ and so
$C=[Q],\,C'=[Q']$ for some $Q=ax^2+bxy+cy^2,\,Q'=a'x^2+b'xy+c'y^2\in\mathcal{Q}(D_\mathcal{O},\,N)$, respectively.
By Lemma \ref{Qgprime}, we see that there is a matrix $\gamma$ in $\mathrm{SL}_2(\mathbb{Z})$ such that
$Q''=Q'^\gamma=a''x^2+b''xy+c''y^2$ satisfies
$\gcd(a,\,a'',\,(b+b'')/2)=1$.
We then find that
\begin{eqnarray*}
\phi_{\mathcal{O},\,N}(C)
\phi_{\mathcal{O},\,N}(C')&=&\left[
[\omega_{Q},\,1][\omega_{Q'},\,1]
\right]\\
&=&\left[
[\omega_{Q},\,1][\gamma(\omega_{Q''}),\,1]
\right]\\
&=&
\left[\frac{1}{j}[\omega_{Q},\,1][\omega_{Q''},\,1]
\right]\quad\textrm{with}~j=j(\gamma,\,\omega_{Q''})\\
&=&\left[\frac{1}{j}[\omega_{Q'''},\,1]\right]
\quad\textrm{where $Q'''$ is a Dirichlet composition of $Q$ and $Q''$}\\
&&\hspace{2.6cm}\textrm{by Lemma \ref{Dirichlet}}\\
&=&\left[[\nu_1,\,\nu_2]\right]\quad\textrm{where}~
\nu_1=\frac{\omega_{Q'''}}{j}~\textrm{and}~\nu_2=\frac{1}{j}.
\end{eqnarray*}
Since $[\omega_Q,\,1][\omega_{Q'},\,1]=[\nu_1,\,\nu_2]$ contains $1$, we derive
\begin{equation}\label{uv1}
u\nu_1+v\nu_2=1\quad
\textrm{for some unique}~(u,\,v)\in\mathbb{Z}^2.
\end{equation}
Furthermore, since
\begin{equation}\label{aa'M}
aa'[\nu_1,\,\nu_2]=(a[\omega_Q,\,1])(a'[\omega_{Q'},\,1])\in \mathcal{M}(\mathcal{O},\,N)
\end{equation}
by Lemmas \ref{primelemma}, \ref{QaN} and the fact $\gcd(N,\,a)=\gcd(N,\,a')=1$,
we must have $\gcd(N,\,u,\,v)=1$.
Since the reduction $\mathrm{SL}_2(\mathbb{Z})\rightarrow\mathrm{SL}_2(\mathbb{Z}/N\mathbb{Z})$ is surjective
(\cite[Lemma 1.38]{Shimura}), there exists a matrix
$\sigma=\begin{bmatrix}\mathrm{*}&\mathrm{*}\\
u'&v'\end{bmatrix}$ in $\mathrm{SL}_2(\mathbb{Z})$ such that
\begin{equation}\label{equivuv}
\sigma\equiv\begin{bmatrix}\mathrm{*}&\mathrm{*}\\u&v\end{bmatrix}\Mod{NM_2(\mathbb{Z})}.
\end{equation}
If we let $Q''''=Q'''^{\sigma^{-1}}$, then we find by the fact $\displaystyle\frac{\nu_1}{\nu_2}=\omega_{Q'''}$ that
\begin{equation}\label{Q''''}
[\omega_{Q''''},\,1]=[\sigma(\omega_{Q'''}),\,1]=\frac{1}{j(\sigma,\,\omega_{Q'''})}[\omega_{Q'''},\,1]
=\frac{aa'}{u'(aa'\nu_1)+v'(aa'\nu_2)}[\nu_1,\,\nu_2].
\end{equation}
Thus we establish that
\begin{eqnarray*}
\{u'(aa'\nu_1)+v'(aa'\nu_2)\}-aa'&=&\{u'(aa'\nu_1)+v'(aa'\nu_2)\}-aa'(u\nu_1+v\nu_2)
\quad\textrm{by (\ref{uv1})}\\
&=&(u'-u)(aa'\nu_1)+(v'-v)(aa'\nu_2)\\
&\in& N(aa'[\nu_1,\,\nu_2])\quad\textrm{by (\ref{equivuv})}\\
&\subseteq&N\mathcal{O}\quad\textrm{by (\ref{aa'M})},
\end{eqnarray*}
and hence
\begin{equation*}
nu'(aa'\nu_1)+v'(aa'\nu_2)\equiv aa'\Mod{N\mathcal{O}}.
\end{equation*}
And we attain by (\ref{Q''''}) and Lemma \ref{anotherP} that
\begin{equation*}
\phi_{\mathcal{O},\,N}([Q''''])=\left[[\omega_{Q''''},\,1]\right]
=\left[[\nu_1,\,\nu_2]\right]
=\phi_{\mathcal{O},\,N}(C)\phi_{\mathcal{O},\,N}(C').
\end{equation*}
Therefore the binary operation on $\mathcal{C}_N(D_\mathcal{O})$ given in Definition \ref{binaryoperation} is
well defined, which makes $\mathcal{C}_N(D_\mathcal{O})$ a group isomorphic to $\mathcal{C}_N(\mathcal{O})$
via the isomorphism $\phi_{\mathcal{O},\,N}$.
\end{proof}

\section {The Shimura reciprocity law}

We shall briefly review the original version of Shimura's reciprocity law
in order to prepare for $\S$\ref{sect11}.
\par
Let $\widehat{\mathbb{Z}}=\displaystyle\prod_{p\,:\,\textrm{primes}}\mathbb{Z}_p$ and
$\widehat{\mathbb{Q}}=\mathbb{Q}\otimes_\mathbb{Z}\widehat{\mathbb{Z}}$.
Note that
\begin{equation*}
\mathrm{GL}_2(\widehat{\mathbb{Q}})=
\left\{\gamma=(\gamma_p)_p\in\prod_p\mathrm{GL}_2(\mathbb{Q}_p)~|~
\gamma_p\in\mathrm{GL}_2(\mathbb{Z}_p)~\textrm{for all but finitely many $p$}\right\}
\end{equation*}
(\cite[Exercise 15.6]{Cox}).
Let $\mathcal{F}$ be the field of all meromorphic modular functions, namely,
\begin{equation*}
\mathcal{F}=\displaystyle\bigcup_{N=1}^\infty\mathcal{F}_N.
\end{equation*}

\begin{proposition}\label{idelicF}
There is a
surjective homomorphism
\begin{equation*}
\sigma_\mathcal{F}~:~\mathrm{GL}_2(\widehat{\mathbb{Q}})\rightarrow\mathrm{Aut}(\mathcal{F})
\end{equation*}
described as follows\,\textup{:} Let $\gamma\in\mathrm{GL}_2(\widehat{\mathbb{Q}})$ and
$f\in\mathcal{F}_N$ for some positive integer $N$.
One can decompose $\gamma$ as
\begin{equation*}
\gamma=\alpha\beta\quad\textrm{for some}~\alpha=(\alpha_p)_p\in\mathrm{GL}_2(\widehat{\mathbb{Z}})
~\textrm{and}~\beta\in\mathrm{GL}_2^+(\mathbb{Q})
\end{equation*}
\textup{(\cite[Theorem 1 in Chapter 7]{Lang87})}.
By using the Chinese remainder theorem, take a matrix $A$ in $M_2(\mathbb{Z})$ such that $A\equiv\alpha_p\Mod{N M_2(\mathbb{Z}_p)}$ for all
primes $p$ dividing $N$. Then we have
\begin{equation*}
f^{\sigma_\mathcal{F}(\gamma)}=f^{\widetilde{A}}\circ\beta
\end{equation*}
where $\widetilde{A}$ is the image of $A$ in $\mathrm{GL}_2(\mathbb{Z}/N\mathbb{Z})/\langle-I_2\rangle$
\textup{(}$\simeq\mathrm{Gal}(\mathcal{F}_N/\mathcal{F}_1)$\textup{)}
and $\beta$ is regarded as a fractional linear transformation on $\mathbb{H}$.
\end{proposition}
\begin{proof}
See \cite[Theorem 6 in Chapter 7]{Lang87} or \cite[Theorem 6.23]{Shimura}.
\end{proof}

\begin{remark}
The kernel of $\sigma_\mathcal{F}$ is the image of $\mathbb{Q}^*$ through the diagonal embedding
into $\mathrm{GL}_2(\widehat{\mathbb{Q}})$.
\end{remark}

Let $\omega\in K\cap\mathbb{H}$. We define the normalized embedding
\begin{equation*}
q_\omega:K^*\rightarrow\mathrm{GL}_2^+(\mathbb{Q})
\end{equation*}
by the relation
\begin{equation*}
\nu\begin{bmatrix}\omega\\1\end{bmatrix}=q_\omega(\nu)\begin{bmatrix}\omega\\1\end{bmatrix}
\quad(\nu\in K^*).
\end{equation*}
For each prime $p$, we can continuously extend $q_\omega$ to the embedding
\begin{equation*}
q_{\omega,\,p}:(K\otimes_\mathbb{Z}\mathbb{Z}_p)^*\rightarrow\mathrm{GL}_2(\mathbb{Q}_p),
\end{equation*}
and hence to the embedding
\begin{equation*}
q_\omega:\widehat{K}^*\rightarrow\mathrm{GL}_2(\widehat{\mathbb{Q}})
\end{equation*}
where $\widehat{K}=K\otimes_\mathbb{Z}\widehat{\mathbb{Z}}$.
Let $K^{\mathrm{ab}}$ be the maximal abelian extension of $K$, and denote by
\begin{equation*}
[\,\cdot\,,\,K]:\widehat{K}^*\rightarrow\mathrm{Gal}(K^{\mathrm{ab}}/K)
\end{equation*}
the Artin map for $K$ defined on the group $\widehat{K}^*$ of finite $K$-ideles.
By using his theory of canonical models for modular curves,
Shimura established the following reciprocity law.

\begin{proposition}\label{idelic}
Let $\omega\in K\cap\mathbb{H}$ and $f\in\mathcal{F}$. If
$f$ is finite at $\omega$, then $f(\omega)$ belongs to $K^{\mathrm{ab}}$
and satisfies
\begin{equation*}
f(\omega)^{[s^{-1},\,K]}=f^{\sigma_\mathcal{F}(q_\omega(s))}(\omega)\quad(s\in\widehat{K}^*).
\end{equation*}
\end{proposition}
\begin{proof}
See \cite[6.31]{Shimura}.
\end{proof}

\section {Ray class invariants for orders}\label{sect11}

We shall define invariants for each class in $\mathcal{C}_N(\mathcal{O})$
in terms of special values of modular functions, and examine their
Galois conjugates via the Artin map.

\begin{definition}\label{invariant}
Let $C\in\mathcal{C}_N(\mathcal{O})$.
For each $f\in\mathcal{F}_N$, we define the invariant $f(C)$ as follows\,:
One can take an $\mathcal{O}$-ideal $\mathfrak{c}$ in $C\cap\mathcal{M}(\mathcal{O},\,N)$
by Remark \ref{CCOMN}. Observe that $\gcd(\mathrm{N}_\mathcal{O}(\mathfrak{c}),\,N)=1$ by Lemma \ref{primelemma}.
Choose
a $\mathbb{Z}$-basis $\{\xi_1,\,\xi_2\}$ of $\mathfrak{c}^{-1}$
so that
\begin{equation*}
\xi:=\frac{\xi_1}{\xi_2}\in\mathbb{H}.
\end{equation*}
Since $[\tau_\mathcal{O},\,1]=\mathcal{O}\subseteq\mathfrak{c}^{-1}=[\xi_1,\,\xi_2]$
and $\tau_Q,\,\xi\in\mathbb{H}$, we have
\begin{equation}\label{tAx}
\begin{bmatrix}
\tau_\mathcal{O}\\1
\end{bmatrix}=A\begin{bmatrix}\xi_1\\\xi_2\end{bmatrix}
\quad\textrm{for some}~A\in M_2(\mathbb{Z})\cap\mathrm{GL}_2^+(\mathbb{Q}).
\end{equation}
On the other hand, we get by Lemma \ref{normbasic} (iii) that
\begin{equation*}
[\tau_\mathcal{O},\,1]=\mathcal{O}\supseteq\overline{\mathfrak{c}}=\mathrm{N}_\mathcal{O}(\mathfrak{c})\mathfrak{c}^{-1}
=[\mathrm{N}_\mathcal{O}(\mathfrak{c})\xi_1,\,\mathrm{N}_\mathcal{O}(\mathfrak{c})\xi_2]
\end{equation*}
and so
\begin{equation}\label{NBt}
\begin{bmatrix}\mathrm{N}_\mathcal{O}(\mathfrak{c})\xi_1\\
\mathrm{N}_\mathcal{O}(\mathfrak{c})\xi_2\end{bmatrix}=
B\begin{bmatrix}\tau_\mathcal{O}\\1\end{bmatrix}\quad
\textrm{for some}~B\in M_2(\mathbb{Z})\cap\mathrm{GL}_2^+(\mathbb{Q}).
\end{equation}
Thus we obtain by (\ref{tAx}) and (\ref{NBt}) that
\begin{equation*}
AB\begin{bmatrix}\tau_\mathcal{O} & \overline{\tau}_\mathcal{O}\\1&1\end{bmatrix}=\mathrm{N}_\mathcal{O}(\mathfrak{c})
\begin{bmatrix}\tau_\mathcal{O} & \overline{\tau}_\mathcal{O}\\1&1\end{bmatrix}.
\end{equation*}
By taking determinant and squaring, we attain 
\begin{equation*}
\det(A)^2\det(B)^2D_\mathcal{O}=
\mathrm{N}_\mathcal{O}(\mathfrak{c})^4D_\mathcal{O},
\end{equation*}
which shows that $\gcd(\det(A),\,N)=1$.
We then define
\begin{equation*}
f(C)=f^{\widetilde{A}}(\xi)
\end{equation*}
where $\widetilde{A}$ is the image of $A$ in
$\mathrm{GL}_2(\mathbb{Z}/N\mathbb{Z})/
\langle-I_2\rangle$, if
$f^{\widetilde{A}}$ is finite at $\xi$.
\end{definition}

\begin{lemma}\label{a+b=O}
Let $\mathfrak{a},\,\mathfrak{b}\in\mathcal{M}(\mathcal{O})$. If $\mathfrak{a}+\mathfrak{b}=\mathcal{O}$,
then $\mathfrak{a}\cap\mathfrak{b}\subseteq\mathfrak{a}\mathfrak{b}$.
\end{lemma}
\begin{proof}
Since $\mathfrak{a}+\mathfrak{b}=\mathcal{O}$, we deduce
\begin{equation}\label{a+b=1}
1=a+b\quad\textrm{for some $a\in\mathfrak{a}$ and $b\in\mathfrak{b}$}.
\end{equation}
Let $c\in\mathfrak{a}\cap\mathfrak{b}$.
Then we see that
\begin{eqnarray*}
c&=&c(a+b)\quad\textrm{by (\ref{a+b=1})}\\
&=&ac+cb\\
&\in&\mathfrak{a}\mathfrak{b}\quad\textrm{since $a\in\mathfrak{a}$, $c\in\mathfrak{b}$
and $c\in\mathfrak{a}$, $b\in\mathfrak{b}$}.
\end{eqnarray*}
This proves that $\mathfrak{a}\cap\mathfrak{b}\subseteq\mathfrak{a}\mathfrak{b}$.
\end{proof}

\begin{lemma}\label{welldefined}
Let $C\in\mathcal{C}_N(\mathcal{O})$
and $f\in\mathcal{F}_N$. The value $f(C)$
in \textup{Definition \ref{invariant}}, if
it is finite, does not depend on the choice
of $\mathfrak{c}$ and $\{\xi_1,\,\xi_2\}$.
\end{lemma}
\begin{proof}
We follow the notation of Definition \ref{invariant}.
Let $\mathfrak{c}'\in C\cap\mathcal{M}(\mathcal{O},\,N)$, and let $\{\xi_1',\,\xi_2'\}$
be a $\mathbb{Z}$-basis of $\mathfrak{c}'^{-1}$
so that
\begin{equation*}
\xi':=\frac{\xi_1'}{\xi_2'}\in\mathbb{H}.
\end{equation*}
Let $A'$ be the matrix in $M_2(\mathbb{Z})\cap\mathrm{GL}_2^+(\mathbb{Q})$
such that
\begin{equation}\label{tA'x'}
\begin{bmatrix}\tau_\mathcal{O}\\1\end{bmatrix}
=A'\begin{bmatrix}\xi_1'\\\xi_2'\end{bmatrix}.
\end{equation}
Since $C=[\mathfrak{c}]=[\mathfrak{c}']$,
we have
\begin{equation}\label{c'c}
\mathfrak{c}'=\frac{\nu_1}{\nu_2}\mathfrak{c}
\end{equation}
for some  $\nu_1,\,\nu_2\in\mathcal{O}\setminus\{0\}$ satisfying
\begin{equation}\label{nn1N}
\nu_1\equiv\nu_2\equiv1\Mod{N\mathcal{O}}.
\end{equation}
If we let $\displaystyle\nu=\frac{\nu_1}{\nu_2}$, then
we get by (\ref{c'c}) that
\begin{equation*}
[\xi_1',\,\xi_2']=\mathfrak{c}'^{-1}=
\nu^{-1}\mathfrak{c}^{-1}=
[\nu^{-1}\xi_1,\,\nu^{-1}\xi_2],
\end{equation*}
and hence
\begin{equation}\label{xBx}
\begin{bmatrix}\xi_1'\\\xi_2'
\end{bmatrix}=B\begin{bmatrix}\nu^{-1}\xi_1\\
\nu^{-1}\xi_2\end{bmatrix}=\nu^{-1}B\begin{bmatrix}\xi_1\\\xi_2\end{bmatrix}
\quad\textrm{for some}~B\in\mathrm{SL}_2(\mathbb{Z}).
\end{equation}
Note that
\begin{eqnarray*}
(\nu_1-\nu_2)\mathfrak{c}&\subseteq&N\mathcal{O}\cap
(\nu_1-\nu_2)\mathfrak{c}\quad\textrm{by (\ref{nn1N})}\\
&\subseteq&N\mathcal{O}\cap(\nu_1\mathfrak{c}+\nu_2\mathfrak{c})\\
&=&N\mathcal{O}\cap(\nu_2\mathfrak{c}'+\nu_2\mathfrak{c})\quad\textrm{by (\ref{c'c})}\\
&\subseteq&N\mathcal{O}\cap\nu_2\mathcal{O}\\
&\subseteq&N\nu_2\mathcal{O}\quad\textrm{by Lemma \ref{a+b=O}
because $\nu_2\mathcal{O}$ is prime to $N$ owing to (\ref{nn1N})}
\end{eqnarray*}
and so
\begin{equation*}
[(\nu-1)\tau_\mathcal{O},\,(\nu-1)]=(\nu-1)\mathcal{O}\subseteq N\mathfrak{c}^{-1}=
[N\xi_1,\,N\xi_2].
\end{equation*}
Thus we obtain that
\begin{equation}\label{nAN}
\begin{bmatrix}(\nu-1)\tau_\mathcal{O}\\\nu-1\end{bmatrix}
=A''\begin{bmatrix}N\xi_1\\N\xi_2\end{bmatrix}
\quad\textrm{for some}~A''\in M_2(\mathbb{Z})\cap\mathrm{GL}_2^+(\mathbb{Q}).
\end{equation}
And we find that
\begin{eqnarray*}
NA''
\begin{bmatrix}\xi_1\\\xi_2\end{bmatrix}
&=&\nu\begin{bmatrix}\tau_\mathcal{O}\\
1\end{bmatrix}-\begin{bmatrix}\tau_\mathcal{O}\\1\end{bmatrix}
\quad\textrm{by (\ref{nAN})}\\
&=&\nu A'\begin{bmatrix}\xi_1'\\\xi_2'\end{bmatrix}-A\begin{bmatrix}\xi_1\\\xi_2\end{bmatrix}\quad\textrm{by
(\ref{tAx}) and (\ref{tA'x'})}\\
&=&(A'B-A)\begin{bmatrix}\xi_1\\\xi_2\end{bmatrix}
\quad\textrm{by (\ref{xBx})}.
\end{eqnarray*}
It then follows that
\begin{equation*}
NA''\begin{bmatrix}\xi_1&\overline{\xi}_1\\\xi_2&\overline{\xi}_2\end{bmatrix}
=(A'B-A)\begin{bmatrix}\xi_1&\overline{\xi}_1\\\xi_2&\overline{\xi}_2\end{bmatrix},
\end{equation*}
and hence $NA''=A'B-A$ and
\begin{equation}\label{AAB}
A'\equiv AB^{-1}\Mod{NM_2(\mathbb{Z})}.
\end{equation}
Finally we derive that
\begin{eqnarray*}
f^{\widetilde{A}}(\xi)&=&f^{\widetilde{A}}(B^{-1}B(\xi))\\
&=&f^{\widetilde{AB^{-1}}}(B(\xi))\quad\textrm{by Proposition \ref{SL2action}}\\
&=&f^{\widetilde{AB^{-1}}}(\xi')\quad\textrm{by (\ref{xBx})}\\
&=&f^{\widetilde{A'}}(\xi')\quad\textrm{by (\ref{AAB})},
\end{eqnarray*}
which proves the well-definedness of the invariant $f(C)$.
\end{proof}

By composing
three isomorphisms
\begin{enumerate}
\item[(i)] $\mathcal{C}_N(\mathcal{O})\stackrel{\sim}{\rightarrow}
\mathcal{C}_N(\mathcal{O},\,\ell_\mathcal{O}N)$ achieved from Proposition \ref{MNN},
\item[(ii)] $\mathcal{C}_N(\mathcal{O},\,\ell_\mathcal{O}N)\stackrel{\sim}{\rightarrow}
I(\mathcal{O}_K,\,\ell_\mathcal{O}N)/P_{\mathbb{Z},\,N}(\mathcal{O}_K,\,\ell_\mathcal{O}N)$ established
in Proposition \ref{generalized},
\item[(iii)] the Artin map
$I(\mathcal{O}_K,\,\ell_\mathcal{O}N)/P_{\mathbb{Z},\,N}(\mathcal{O}_K,\,\ell_\mathcal{O}N)
\stackrel{\sim}{\rightarrow}\mathrm{Gal}(K_{\mathcal{O},\,N}/K)$,
\end{enumerate}
we get the isomorphism
\begin{equation*}
\sigma_{\mathcal{O},\,N}:\mathcal{C}_N(\mathcal{O})
\stackrel{\sim}{\rightarrow}
\mathrm{Gal}(K_{\mathcal{O},\,N}/K).
\end{equation*}
Let $C_0$ denote the identity class in $\mathcal{C}_N(\mathcal{O})$.

\begin{theorem}\label{transformation}
Let $C\in\mathcal{C}_N(\mathcal{O})$ and $f\in\mathcal{F}_N$. If
$f$ is finite at $\tau_\mathcal{O}$, then
$f(C)$ belongs to $K_{\mathcal{O},\,N}$
and satisfies
\begin{equation*}
f(C)^{\sigma_{\mathcal{O},\,N}(C')}=
f(CC')\quad(C'\in\mathcal{C}_N(\mathcal{O})).
\end{equation*}
\end{theorem}
\begin{proof}
By Definition \ref{invariant}, Lemma \ref{welldefined} and (\ref{specialization}), we attain
\begin{equation}\label{fC0}
f(C_0)=f(\tau_\mathcal{O})\in K_{\mathcal{O},\,N}.
\end{equation}
And one can take $\mathfrak{c}\in C\cap\mathcal{M}(\mathcal{O},\,\ell_\mathcal{O}N)$
by Remark \ref{CCOMN}.
Let $\{\xi_1,\,\xi_2\}$ be a $\mathbb{Z}$-basis
of $\mathfrak{c}^{-1}$ such that
\begin{equation*}
\xi:=\frac{\xi_1}{\xi_2}\in\mathbb{H}.
\end{equation*}
Furthermore, let $A$ be the matrix in $M_2(\mathbb{Z})\cap\mathrm{GL}_2^+(\mathbb{Q})$ satisfying
\begin{equation}\label{t1axx}
\begin{bmatrix}\tau_\mathcal{O}\\1\end{bmatrix}=A\begin{bmatrix}\xi_1\\
\xi_2\end{bmatrix}.
\end{equation}
Take an idele $s=(s_p)_{p\,:\,\mathrm{primes}}$ in $\widehat{K}^*=(K\otimes_\mathbb{Z}\widehat{\mathbb{Z}})^*$ such that
\begin{equation}\label{s1}
\left\{\begin{array}{ll}
s_p=1 & \textrm{if}~p\,|\,\ell_\mathcal{O}N,\\
s_p(\mathcal{O}_K\otimes_\mathbb{Z}\mathbb{Z}_p)=
\mathfrak{c}\mathcal{O}_K\otimes_\mathbb{Z}\mathbb{Z}_p
 & \textrm{if}~p\nmid \ell_\mathcal{O}N.
\end{array}\right.
\end{equation}
If $p\,|\,\ell_\mathcal{O}N$, then we see that
\begin{eqnarray*}
s_p(\mathcal{O}\otimes_\mathbb{Z}\mathbb{Z}_p)&=&
\mathcal{O}\otimes_\mathbb{Z}\mathbb{Z}_p\quad\textrm{since}~s_p=1\\
&\supseteq&\mathfrak{c}\otimes_\mathbb{Z}\mathbb{Z}_p\\
&\supseteq&
\mathrm{N}_\mathcal{O}(\mathfrak{c})\mathcal{O}\otimes_\mathbb{Z}\mathbb{Z}_p\quad\textrm{by Lemma \ref{normbasic} (iii)}\\
&=&\mathcal{O}\otimes_\mathbb{Z}\mathbb{Z}_p\quad\textrm{by the fact $\mathfrak{c}\in\mathcal{M}(\mathcal{O},\,\ell_\mathcal{O}N)$
and Lemma \ref{primelemma}},
\end{eqnarray*}
and hence
\begin{equation}\label{pMN}
s_p(\mathcal{O}\otimes_\mathbb{Z}\mathbb{Z}_p)=\mathfrak{c}\otimes_\mathbb{Z}
\mathbb{Z}_p\quad(\textrm{if}~p\,|\,\ell_\mathcal{O}N).
\end{equation}
If $p\nmid \ell_\mathcal{O}N$, then $\ell_\mathcal{O}$ is a unit in $\mathbb{Z}_p$ and so
\begin{equation}\label{pnotMN}
s_p(\mathcal{O}\otimes_\mathbb{Z}\mathbb{Z}_p)
=s_p(\mathcal{O}_K\otimes_\mathbb{Z}\mathbb{Z}_p)
=\mathfrak{c}\mathcal{O}_K\otimes_\mathbb{Z}\mathbb{Z}_p
=\mathfrak{c}\mathcal{O}\otimes_\mathbb{Z}\mathbb{Z}_p
=\mathfrak{c}\otimes_\mathbb{Z}\mathbb{Z}_p
\quad(\textrm{if}~p\nmid \ell_\mathcal{O}N).
\end{equation}
Then it follows from (\ref{pMN}) and (\ref{pnotMN}) that
\begin{equation}\label{s-1c-1}
s_p^{-1}(\mathcal{O}\otimes_\mathbb{Z}\mathbb{Z}_p)=\mathfrak{c}^{-1}\otimes_\mathbb{Z}\mathbb{Z}_p
\quad\textrm{for every prime $p$}.
\end{equation}
Since
\begin{equation*}
s_p^{-1}\begin{bmatrix}\tau_\mathcal{O}\\1\end{bmatrix}=q_{\tau_\mathcal{O},\,p}(s_p^{-1})\begin{bmatrix}\tau_\mathcal{O}\\1\end{bmatrix}=
q_{\tau_\mathcal{O},\,p}(s_p^{-1})A\begin{bmatrix}\xi_1\\\xi_2\end{bmatrix}
\end{equation*}
by (\ref{t1axx}),
we deduce by (\ref{s-1c-1}) and the fact $\mathfrak{c}^{-1}=[\xi_1,\,\xi_2]$ that
\begin{equation*}
q_{\tau_\mathcal{O},\,p}(s_p^{-1})A\in\mathrm{GL}_2(\mathbb{Z}_p),
\end{equation*}
and so
\begin{equation}\label{qsAG}
q_{\tau_\mathcal{O}}(s^{-1})A\in\mathrm{GL}_2(\widehat{\mathbb{Z}}).
\end{equation}
We then find that
\begin{eqnarray*}
f(C_0)^{\sigma_{\mathcal{O},\,N}(C)}&=&f(C_0)^{\sigma_{\mathcal{O},\,N}([\mathfrak{c}])}\\
&=&f(\tau_\mathcal{O})^{[s,\,K]}\quad\textrm{by (\ref{fC0}) and (\ref{s1})}\\
&=&f^{\sigma_\mathcal{F}(q_{\tau_\mathcal{O}}(s^{-1}))}(\tau_\mathcal{O})
\quad\textrm{by Proposition \ref{idelic}}\\
&=&f^{\sigma_\mathcal{F}((q_{\tau_\mathcal{O}}(s^{-1})A)A^{-1})}(\tau_\mathcal{O})\\
&=&f^{\sigma_\mathcal{F}(q_{\tau_\mathcal{O}}(s^{-1})A)}(A^{-1}(\tau_\mathcal{O}))\quad
\textrm{by Proposition \ref{idelicF} because}~A^{-1}\in\mathrm{GL}_2^+(\mathbb{Q})\\
&=&f^{\sigma_\mathcal{F}(q_{\tau_\mathcal{O}}(s^{-1})A)}(\xi)\quad
\textrm{by (\ref{t1axx})}\\
&=&f^{\widetilde{G}}(\xi)\quad\textrm{where $G$ is a matrix in $M_2(\mathbb{Z})$ such that}\\
&&\hspace{1.4cm}G\equiv q_{\tau_\mathcal{O},\,p}(s_p^{-1})A\Mod{NM_2(\mathbb{Z}_p)}~\textrm{for all primes $p$ dividing $N$,}\\
&&\hspace{1.4cm}\textrm{by (\ref{qsAG})
and Proposition \ref{idelicF}}\\
&=&f^{\widetilde{A}}(\xi)\quad\textrm{since the fact $s_p=1$ for all primes $p$ dividing $N$ implies}\\
&&\hspace{1.4cm}G\equiv A\Mod{NM_2(\mathbb{Z})}\\
&=&f(C).
\end{eqnarray*}
This proves that $f(C)$ is finite and belongs to $K_{\mathcal{O},\,N}$.
And we further derive that for $C'\in\mathcal{C}_N(\mathcal{O})$
\begin{equation*}
f(C)^{\sigma_{\mathcal{O},\,N}(C')}
=(f(C_0)^{\sigma_{\mathcal{O},\,N}(C)})^{\sigma_{\mathcal{O},\,N}(C')}
=f(C_0)^{\sigma_{\mathcal{O},\,N}(CC')}=f(CC').
\end{equation*}
\end{proof}

\section {An isomorphism of $\mathcal{C}_{N}(D_\mathcal{O})$ onto
$\mathrm{Gal}(K_{\mathcal{O},\,N}/K)$}

In this section, we shall explicitly describe the isomorphism
\begin{equation*}
\sigma_{\mathcal{O},\,N}\circ\phi_{\mathcal{O},\,N}:\mathcal{C}_N(D_\mathcal{O})
\stackrel{\sim}{\rightarrow}
\mathrm{Gal}(K_{\mathcal{O},\,N}/K).
\end{equation*}
Furthermore, we shall show that there exists a form class group associated with
the principal congruence subgroup $\Gamma(N)$.

\begin{definition}\label{DeffQ}
Let $Q\in\mathcal{Q}(D_\mathcal{O},\,N)$ and $f\in\mathcal{F}_N$. If $f$ is finite at $\tau_\mathcal{O}$, then
we define
\begin{equation*}
f([Q])=f(\phi_{\mathcal{O},\,N}([Q])).
\end{equation*}
\end{definition}

\begin{lemma}\label{fmatrix}
Let $Q=ax^2+bxy+cy^2\in\mathcal{Q}(D_\mathcal{O},\,N)$ and $f\in\mathcal{F}_N$ which is finite at $\tau_\mathcal{O}$.
Then we have
\begin{equation*}
f([Q])=f^{\left[\begin{smallmatrix}
1 & -a'(b+b_\mathcal{O})/2\\
0&a'\end{smallmatrix}\right]}(-\overline{\omega}_Q)
\end{equation*}
where $a'$ is an integer satisfying $aa'\equiv1\Mod{N}$.
\end{lemma}
\begin{proof}
Let $C=\phi_{\mathcal{O},\,N}([Q])=[[\omega_Q,\,1]]$.
We see by the facts $\gcd(a,\,N)=1$, $a^{\varphi(N)}\equiv1\Mod{N}$ and Lemma \ref{QaN} that
\begin{equation*}
\mathfrak{c}:=a^{\varphi(N)}[\omega_Q,\,1]
=a^{\varphi(N)-1}(a[\omega_Q,\,1])\in C\cap\mathcal{M}(\mathcal{O},\,N).
\end{equation*}
Now, we find that
\begin{eqnarray*}
\mathfrak{c}^{-1}&=&a^{-\varphi(N)+1}\mathfrak{b}^{-1}\quad\textrm{where}~\mathfrak{b}=a[\omega_Q,\,1]~(\in\mathcal{M}
(\mathcal{O},\,N))\\
&=&a^{-\varphi(N)+1}\mathrm{N}_\mathcal{O}(\mathfrak{b})^{-1}\overline{\mathfrak{b}}\quad\textrm{by Lemma \ref{normbasic} (iii)}\\
&=&a^{-\varphi(N)+1}a^{-1}(a[\overline{\omega}_Q,\,1])\quad\textrm{by Lemma \ref{QaN}}\\
&=&a^{-\varphi(N)+1}[-\overline{\omega}_Q,\,1].
\end{eqnarray*}
Thus, if we take
\begin{equation*}
\xi_1=-a^{-\varphi(N)+1}\overline{\omega}_Q\quad\textrm{and}\quad
\xi_2=a^{-\varphi(N)+1},
\end{equation*}
then we achieve
\begin{equation*}
\mathfrak{c}^{-1}=[\xi_1,\,\xi_2]\quad\textrm{and}\quad
\xi:=\frac{\xi_1}{\xi_2}=-\overline{\omega}_Q\in\mathbb{H}.
\end{equation*}
Furthermore, since
\begin{equation*}
-\overline{\omega}_Q=\frac{1}{a}\left(\tau_\mathcal{O}+\frac{b+b_\mathcal{O}}{2}\right),
\end{equation*}
we obtain that
\begin{equation*}
\begin{bmatrix}\tau_\mathcal{O}\\1\end{bmatrix}=
A\begin{bmatrix}\xi_1\\\xi_2\end{bmatrix}\quad\textrm{with}~
A=\begin{bmatrix}a^{\varphi(N)} & -a^{\varphi(N)-1}(b+b_\mathcal{O})/2\\
0&a^{\varphi(N)-1}\end{bmatrix}.
\end{equation*}
Since
\begin{equation*}
A\equiv
\begin{bmatrix}1 & -a'(b+b_\mathcal{O})/2\\
0&a'\end{bmatrix}\Mod{NM_2(\mathbb{Z})},
\end{equation*}
we conclude by Definition \ref{invariant} and Lemma \ref{welldefined} that
\begin{equation*}
f([Q])=
f(C)=
f^{\left[\begin{smallmatrix}
1 & -a'(b+b_\mathcal{O})/2\\
0&a'\end{smallmatrix}\right]}(-\overline{\omega}_Q).
\end{equation*}
\end{proof}

\begin{theorem}\label{CDGKK}
The map
\begin{eqnarray*}
\mathcal{C}_N(D_\mathcal{O}) & \rightarrow & \mathrm{Gal}(K_{\mathcal{O},\,N}/K)\\
\mathrm{[}Q]=[ax^2+bxy+cy^2] & \mapsto & \left(f(\tau_\mathcal{O})\mapsto
f^{\left[\begin{smallmatrix}
1 & -a'(b+b_\mathcal{O})/2\\
0&a'\end{smallmatrix}\right]}(-\overline{\omega}_Q)~|~f\in\mathcal{F}_N~\textrm{is finite at $\tau_\mathcal{O}$}\right)
\end{eqnarray*}
is a well-defined isomorphism, where $a'$ is an integer which holds $aa'\equiv1\Mod{N}$.
\end{theorem}
\begin{proof}
Let $\psi_{\mathcal{O},\,N}:\mathcal{C}_N(D_\mathcal{O})\stackrel{\sim}{\rightarrow}\mathrm{Gal}(K_{\mathcal{O},\,N}/K)$ be the isomorphism arising from composition of two isomorphisms
\begin{equation*}
\phi_{\mathcal{O},\,N}:\mathcal{C}_N(D_\mathcal{O})\stackrel{\sim}{\rightarrow}\mathcal{C}_N(\mathcal{O})
\quad\textrm{and}\quad
\sigma_{\mathcal{O},\,N}:\mathcal{C}_N(\mathcal{O})\stackrel{\sim}{\rightarrow}\mathrm{Gal}(K_{\mathcal{O},\,N}/K).
\end{equation*}
Let $Q=ax^2+bxy+cy^2\in\mathcal{Q}(D_\mathcal{O},\,N)$ and $f\in\mathcal{F}_N$ which is finite at $\tau_\mathcal{O}$. If we let $C=\phi_{\mathcal{O},\,N}([Q])$, then
we find that
\begin{eqnarray*}
f(\tau_\mathcal{O})^{\psi_{\mathcal{O},\,N}([Q])}&=&
f(C_0)^{\sigma_{\mathcal{O},\,N}(C)}\quad\textrm{by (\ref{fC0})}\\
&=&f(C)\quad\textrm{by Theorem \ref{transformation}}\\
&=&f([Q])\quad\textrm{by Definition \ref{DeffQ}}\\
&=&f^{\left[\begin{smallmatrix}
1 & -a'(b+b_\mathcal{O})/2\\
0&a'\end{smallmatrix}\right]}(-\overline{\omega}_Q)\quad\textrm{where $a'$ is an integer satisfying $aa'\equiv1\Mod{N}$},\\
&&\hspace{3.9cm}\textrm{by Lemma \ref{fmatrix}}.
\end{eqnarray*}
And, the result follows from (\ref{specialization}).
\end{proof}

Fix a positive transcendental number $t$.
Then the extension
$K_{\mathcal{O},\,N}(\sqrt[N]{t})/K(t)$ is Galois
because $K_{\mathcal{O},\,N}$ contains
$\zeta_N$ by (\ref{specialization}).
Since the polynomial $x^N-t$ in $x$ is irreducible over $K_{\mathcal{O},\,N}(t)$, we establish the isomorphism
\begin{equation}\label{ZNZ}
\begin{array}{rcl}
\mathbb{Z}/N\mathbb{Z}&\stackrel{\sim}{\rightarrow}&\mathrm{Gal}(K_{\mathcal{O},\,N}(\sqrt[N]{t})/K_{\mathcal{O},\,N}(t))\vspace{0.1cm}\\
\mathrm{[}m]&\mapsto&\left(\sqrt[N]{t}\mapsto\zeta_N^m\sqrt[N]{t}\right).
\end{array}
\end{equation}
Furthermore, since  $\sqrt[N]{t}$ is also a transcendental number, we get the following isomorphism
\begin{equation}\label{ssK}
\begin{array}{rcl}
\mathrm{Gal}(K_{\mathcal{O},\,N}(\sqrt[N]{t})/K(\sqrt[N]{t}))
&\stackrel{\sim}{\rightarrow}&\mathrm{Gal}(K_{\mathcal{O},\,N}/K)\\
\sigma&\mapsto&\sigma|_{K_{\mathcal{O},\,N}}.
\end{array}
\end{equation}

\begin{lemma}\label{semiproduct}
We have
\begin{equation*}
\mathrm{Gal}(K_{\mathcal{O},\,N}(\sqrt[N]{t})/K(t))=
\mathrm{Gal}(K_{\mathcal{O},\,N}(\sqrt[N]{t})/K_{\mathcal{O},\,N}(t))\rtimes
\mathrm{Gal}(K_{\mathcal{O},\,N}(\sqrt[N]{t})/K(\sqrt[N]{t})),
\end{equation*}
where $\mathrm{Gal}(K_{\mathcal{O},\,N}(\sqrt[N]{t})/K(\sqrt[N]{t}))$
acts on $\mathrm{Gal}(K_{\mathcal{O},\,N}(\sqrt[N]{t})/K_{\mathcal{O},\,N}(t))$
by conjugation.
\end{lemma}
\begin{proof}
Let
\begin{equation*}
G=\mathrm{Gal}(K_{\mathcal{O},\,N}(\sqrt[N]{t})/K(t)),~
N=\mathrm{Gal}(K_{\mathcal{O},\,N}(\sqrt[N]{t})/K_{\mathcal{O},\,N}(t))
~\textrm{and}~
H=\mathrm{Gal}(K_{\mathcal{O},\,N}(\sqrt[N]{t})/K(\sqrt[N]{t})).
\end{equation*}
Note that the map
\begin{eqnarray*}
N\times H&\rightarrow&
NH\\
(\sigma_1,\,\sigma_2)&\mapsto&\sigma_1\sigma_2
\end{eqnarray*}
is bijective because $N\cap H
=\{\mathrm{id}_{K_{\mathcal{O},\,N}(\sqrt[N]{t})}\}$.
Since the extension $K_{\mathcal{O},\,N}(t)/K(t)$ is abelian,
$N$ is normal in  $G$
by Galois theory.
Moreover, we see that
\begin{equation*}
|N|\cdot|H|=|N|\cdot
|\mathrm{Gal}(K_{\mathcal{O},\,N}(t)/K(t))|
=|G|.
\end{equation*}
Therefore we conclude that $G$ is the
semidirect product of the normal subgroup $N$
and $H$
in the sense of \cite[p. 76]{Lang02}.
\end{proof}

Since the group $\Gamma_1(N)$ acts on the set $\mathcal{Q}(D_\mathcal{O},\,N)$, so does its subgroup
$\Gamma(N)$.
Let $\sim_{\Gamma(N)}$ be the equivalence relation on $\mathcal{Q}(D_\mathcal{O},\,N)$ induced from the action of $\Gamma(N)$.

\begin{lemma}\label{Qgexpand}
If $Q=ax^2+bxy+cy^2\in\mathcal{Q}(D_\mathcal{O},\,N)$ and $\gamma=\begin{bmatrix}p&q\\r&s\end{bmatrix}\in\Gamma_1(N)$, then we have
\begin{equation*}
Q^\gamma=(ap^2+bpr+cr^2)x^2+(2apq+2bqr+2crs+b)xy+(aq^2+bqs+cs^2)y^2.
\end{equation*}
\end{lemma}
\begin{proof}
It is straightforward from the fact $\det(\gamma)=ps-qr=1$.
\end{proof}

\begin{corollary}
One can regard
$\mathcal{Q}(D_\mathcal{O},\,N)/\sim_{\Gamma(N)}$
as a group isomorphic to the Galois group $\mathrm{Gal}(K_{\mathcal{O},\,N}(\sqrt[N]{t})/K(t))$.
\end{corollary}
\begin{proof}
For $Q\in\mathcal{Q}(D_\mathcal{O},\,N)$, we denote by $[Q]_{\Gamma(N)}$ and $[Q]_{\Gamma_1(N)}$
its classes in $\mathcal{Q}(D_\mathcal{O},\,N)/\sim_{\Gamma(N)}$ and $\mathcal{C}_N(D_\mathcal{O})$, respectively.
Define a map
\begin{eqnarray*}
\psi~:~\mathcal{Q}(D_\mathcal{O},\,N)/\sim_{\Gamma(N)}&\rightarrow&\mathrm{Gal}(K_{\mathcal{O},\,N}(\sqrt[N]{t})/K(t))\\
\mathrm{[}Q]_{\Gamma(N)}=[ax^2+bxy+cy^2]_{\Gamma(N)} & \mapsto &
\left(\begin{array}{l}
f(\tau_\mathcal{O})\mapsto f([Q]_{\Gamma_1(N)})~(f\in\mathcal{F}_N~\textrm{is finite at}~\tau_\mathcal{O}),\\
\sqrt[N]{t}\mapsto
\zeta_N^\frac{b-b_\mathcal{O}}{2}\sqrt[N]{t}
\end{array}\right).
\end{eqnarray*}
First, we shall check that $\psi$ is well defined. Let $Q=ax^2+bxy+cy^2,\,Q'=a'x^2+b'xy+c'y^2\in\mathcal{Q}(D_\mathcal{O},\,N)$ such that
$[Q]_{\Gamma(N)}=[Q']_{\Gamma(N)}$. Since there is a natural surjection $\mathcal{Q}(D_\mathcal{O},\,N)/\sim_{\mathrm{SL}_2(\mathbb{Z})}
\rightarrow\mathcal{C}_N(D_\mathcal{O})$, we attain
$[Q]_{\Gamma_1(N)}=[Q']_{\Gamma_1(N)}$ and so
\begin{equation*}
f([Q]_{\Gamma_1(N)})=f([Q']_{\Gamma_1(N)})\quad(f\in\mathcal{F}_N~\textrm{is finite at}~\tau_\mathcal{O}).
\end{equation*}
Furthermore, since
\begin{equation*}
Q'=Q^\gamma\quad\textrm{for some}~\gamma=\begin{bmatrix}p&q\\r&s\end{bmatrix}\in\Gamma(N),
\end{equation*}
we find that
\begin{eqnarray*}
\frac{b'-b}{2}&=&apq+bqr+crs\quad\textrm{by Lemma \ref{Qgexpand}}\\
&\equiv&0\Mod{N}
\quad\textrm{because}~q\equiv r\equiv0\Mod{N}.
\end{eqnarray*}
Thus we obtain that
\begin{equation*}
\frac{b'-b_\mathcal{O}}{2}\equiv\frac{b-b_\mathcal{O}}{2}\Mod{N},
\end{equation*}
and hence
\begin{equation*}
\zeta_N^{\frac{b'-b_\mathcal{O}}{2}}\sqrt[N]{t}=\zeta_N^{\frac{b-b_\mathcal{O}}{2}}\sqrt[N]{t}.
\end{equation*}
Therefore $\psi$ is well defined.
\par
Second, we shall prove that $\psi$ is injective. Suppose that
\begin{equation*}
\psi([Q]_{\Gamma(N)})=\psi([Q']_{\Gamma(N)})\quad\textrm{for some}~Q=ax^2+bxy+cy^2,\,Q'=a'x^2+b'xy+c'y^2\in\mathcal{Q}(D_\mathcal{O},\,N).
\end{equation*}
Since
\begin{equation*}
f([Q]_{\Gamma_1(N)})=f([Q']_{\Gamma_1(N)})\quad\textrm{for all}~f\in\mathcal{F}_N~
\textrm{finite at}~\tau_\mathcal{O},
\end{equation*}
we get by Theorem \ref{CDGKK} that $[Q]_{\Gamma_1(N)}=[Q']_{\Gamma_1(N)}$ and so
\begin{equation*}
Q'=Q^\gamma\quad\textrm{for some}~\gamma=\begin{bmatrix}p&q\\r&s\end{bmatrix}\in\Gamma_1(N).
\end{equation*}
Moreover, since  $\zeta_N^\frac{b-b_\mathcal{O}}{2}=\zeta_N^\frac{b'-b_\mathcal{O}}{2}$, we achieve
\begin{equation}\label{bbbb}
\frac{b-b_\mathcal{O}}{2}\equiv
\frac{b'-b_\mathcal{O}}{2}\Mod{N}
\end{equation}
and derive that
\begin{eqnarray*}
0&\equiv&\frac{b'-b}{2}\Mod{N}\quad\textrm{by (\ref{bbbb})}\\
&\equiv&apq+bqr+crs\quad\textrm{by Lemma \ref{Qgexpand}}\\
&\equiv&aq\Mod{N}\quad\textrm{because}~p\equiv1~\textrm{and}~r\equiv0\Mod{N}.
\end{eqnarray*}
It then follows from the fact $\gcd(a,\,N)=1$ that $q\equiv0\Mod{N}$. Thus $\gamma$ belongs to $\Gamma(N)$ and so
\begin{equation*}
[Q]_{\Gamma(N)}=[Q']_{\Gamma(N)}.
\end{equation*}
This observation shows the injectivity of $\psi$.
\par
Third, we shall show that $\psi$ is surjective. Let $\sigma\in\mathrm{Gal}(K_{\mathcal{O},\,N}(\sqrt[N]{t})/K(t))$.
Then there exists a pair $(Q,\,m)$ of $Q=ax^2+bxy+cy^2\in\mathcal{Q}(D_\mathcal{O},\,N)$ and $m\in\mathbb{Z}$ such that
\begin{equation*}
f(\tau_\mathcal{O})^\sigma=f([Q]_{\Gamma_1(N)})~(f\in\mathcal{F}_N~\textrm{is finite at $\tau_\mathcal{O}$})
\quad\textrm{and}\quad
\sqrt[N]{t}^{\,\sigma}=\zeta_N^m\sqrt[N]{t}
\end{equation*}
by Theorem \ref{CDGKK}, (\ref{ZNZ}), (\ref{ssK}) and Lemma \ref{semiproduct}.
Observe that these actions of $\sigma$ completely determine $\sigma$.
If we set
\begin{equation*}
Q'=a'x^2+b'xy+c'y^2=Q^{\left[\begin{smallmatrix}1&a'\{m-(b-b_\mathcal{O})/2\}\\0&1\end{smallmatrix}\right]}
\end{equation*}
where $a'$ is an integer satisfying $aa'\equiv1\Mod{N}$,
then we get that
\begin{equation*}
f(\tau_\mathcal{O})^{\psi([Q']_{\Gamma(N)})}=f([Q']_{\Gamma_1(N)})
=f([Q]_{\Gamma_1(N)})\quad(f\in\mathcal{F}_N~\textrm{is finite at $\tau_\mathcal{O}$})
\end{equation*}
and
\begin{equation*}
\sqrt[N]{t}^{\,\psi([Q']_{\Gamma(N)})}=\zeta_N^\frac{b'-b_\mathcal{O}}{2}\sqrt[N]{t}=
\zeta_N^{aa'm}\sqrt[N]{t}=\zeta_N^m\sqrt[N]{t}
\end{equation*}
by Lemma \ref{Qgexpand}. Thus we have $\sigma=\psi([Q']_{\Gamma(N)})$, which
proves that $\psi$ is surjective as desired.
\par
Finally, through the bijection $\psi$ we can endow the set $\mathcal{Q}(D_\mathcal{O},\,N)/\sim_{\Gamma(N)}$ with a binary operation
so that $\mathcal{Q}(D_\mathcal{O},\,N)/\sim_{\Gamma(N)}$ is isomorphic to
$\mathrm{Gal}(K_{\mathcal{O},\,N}(\sqrt[N]{t})/K(t))$.
\end{proof}

\section {The $L$-functions for orders}\label{sect13}

As an analogue of the Weber $L$-function for a ray class character modulo $N\mathcal{O}_K$,
we shall define an $L$-function for a character of $\mathcal{C}_N(\mathcal{O})$.

\begin{definition}\label{defL}
Let $\chi$ be a character of $\mathcal{C}_N(\mathcal{O})$.
\begin{enumerate}
\item[\textup{(i)}]
We define
the $L$-function $L_\mathcal{O}(\,\cdot\,,\,\chi)$ by
\begin{equation*}
L_\mathcal{O}(s,\,\chi)=\sum_{\mathfrak{a}\in
\mathcal{M}(\mathcal{O},\,N)}\frac{\chi([\mathfrak{a}])}{\mathrm{N}_\mathcal{O}(\mathfrak{a})^s}
\quad(s\in\mathbb{C},~\mathrm{Re}(s)>1)
\end{equation*}
where $[\mathfrak{a}]$ is the class of $\mathfrak{a}$ in $\mathcal{C}_N(\mathcal{O})$.
\item[\textup{(ii)}]
For each $C\in\mathcal{C}_N(\mathcal{O})$, we define the $\zeta$-function
$\zeta_\mathcal{O}(\,\cdot\,,\,C)$ by
\begin{equation*}
\zeta_\mathcal{O}(s,\,C)=\sum_{\mathfrak{a}\in
C\cap\mathcal{M}(\mathcal{O},\,N)}\frac{1}{\mathrm{N}_\mathcal{O}(\mathfrak{a})^s}
\quad(s\in\mathbb{C},\,\mathrm{Re}(s)>1).
\end{equation*}
\end{enumerate}
\end{definition}

\begin{remark}
\begin{enumerate}
\item[(i)] We have
\begin{equation*}
L_\mathcal{O}(s,\,\chi)=\sum_{C\in\mathcal{C}_N(\mathcal{O})}\chi(C)\zeta_\mathcal{O}(s,\,C).
\end{equation*}
\item[(ii)]
In Definition \ref{defL} (i),
what if $\mathcal{C}_N(\mathcal{O})$ and $\mathcal{M}(\mathcal{O},\,N)$
are replaced by $\mathcal{C}_N(\mathcal{O},\,\ell_\mathcal{O}N)$ and
$\mathcal{M}(\mathcal{O},\,\ell_\mathcal{O}N)$, respectively\,?
For a character $\psi$ of $\mathcal{C}_N(\mathcal{O},\,\ell_\mathcal{O}N)$,
define
\begin{equation*}
\mathcal{L}_\mathcal{O}(s,\,\psi)=\sum_{\mathfrak{a}\in
\mathcal{M}(\mathcal{O},\,\ell_\mathcal{O}N)}\frac{\psi([\mathfrak{a}])}{\mathrm{N}_\mathcal{O}(\mathfrak{a})^s}
\quad(s\in\mathbb{C},~\mathrm{Re}(s)>1)
\end{equation*}
where $[\mathfrak{a}]$ is the class of $\mathfrak{a}$ in $\mathcal{C}_N(\mathcal{O},\,\ell_\mathcal{O}N)$.
In particular, if $\ell_\mathcal{O}$ divides $N$, then we get
$\mathcal{C}_N(\mathcal{O},\,\ell_\mathcal{O}N)=\mathcal{C}_N(\mathcal{O})$ and so
$\mathcal{L}_\mathcal{O}(s,\,\psi)=L_\mathcal{O}(s,\,\psi)$.
Let $\widetilde{\psi}$ be the character of the ray class group
$\mathcal{C}_{\ell_\mathcal{O}N}(\mathcal{O}_K)$ achieved
by composing three homomorphisms
\begin{equation*}
\mathcal{C}_{\ell_\mathcal{O}N}(\mathcal{O}_K)
\stackrel{\textrm{natural}}{\longtwoheadrightarrow}I(\mathcal{O}_K,\,\ell_\mathcal{O}N)/
P_{\mathbb{Z},\,N}(\mathcal{O}_K,\,\ell_\mathcal{O}N)
\stackrel{\sim}{\longrightarrow}\mathcal{C}_N(\mathcal{O},\,\ell_\mathcal{O}N)
\stackrel{\psi}{\longrightarrow}\mathbb{C}^*.
\end{equation*}
Here the second isomorphism is the one established in Proposition \ref{generalized}.
And we find by Lemmas \ref{normpreserving} (i) and \ref{monoid} that
\begin{equation*}
\mathcal{L}_\mathcal{O}(s,\,\psi)=\sum_{\mathfrak{a}\in\mathcal{M}(\mathcal{O},\,\ell_\mathcal{O}N)}
\frac{\widetilde{\psi}([\mathfrak{a}\mathcal{O}_K])}
{\mathrm{N}_{\mathcal{O}_K}(\mathfrak{a}\mathcal{O}_K)^s}=
\sum_{\mathfrak{b}\in\mathcal{M}(\mathcal{O}_K,\,\ell_\mathcal{O}N)}
\frac{\widetilde{\psi}([\mathfrak{b}])}{\mathrm{N}_{\mathcal{O}_K}(\mathfrak{b})^s}
=L_{\mathcal{O}_K}(s,\,\widetilde{\psi}).
\end{equation*}
When $N=1$, Meyer (\cite{Meyer}) gave a concrete formula for
the value $L_{\mathcal{O}_K}(1,\,\widetilde{\psi})$.

\end{enumerate}
\end{remark}

\begin{lemma}\label{P}
Let $Q=ax^2+bxy+cy^2\in\mathcal{Q}(D_\mathcal{O},\,N)$,
$\mathfrak{c}=[\omega_Q,\,1]$ and
$[\mathfrak{c}]$ be the class of $\mathfrak{c}$ in $\mathcal{C}_N(\mathcal{O})$.
Let
\begin{equation*}
P=\{\lambda\in a\overline{\mathfrak{c}}~|~\lambda\neq0~
\textrm{and}~\lambda\equiv1\Mod{N\mathcal{O}}\}.
\end{equation*}
\begin{enumerate}
\item[\textup{(i)}] We get
$[\mathfrak{c}]\cap\mathcal{M}(\mathcal{O},\,N)=\{
\lambda\mathfrak{c}~|~\lambda\in P\}$.
\item[\textup{(ii)}] If $\lambda,\,\mu\in P$, then
\begin{equation*}
\lambda\mathfrak{c}=\mu\mathfrak{c}
\quad\Longleftrightarrow\quad\mu=\zeta\lambda~
\textrm{for some}~\zeta\in\mathcal{O}^*~\textrm{such that}~\zeta\equiv1\Mod{N\mathcal{O}}.
\end{equation*}
\end{enumerate}
\end{lemma}
\begin{proof}
\begin{enumerate}
\item[(i)] Let $\lambda\in P$. We see that
\begin{equation*}
\lambda\mathfrak{c}\in[\mathfrak{c}]\quad\textrm{and}\quad
\lambda\mathfrak{c}\in a\overline{\mathfrak{c}}\mathfrak{c}
=\frac{1}{a}(a\mathfrak{c})(\overline{a\mathfrak{c}})
=\mathcal{O}
\end{equation*}
by Lemmas \ref{normbasic} (iii) and \ref{QaN}.
It then follows from Lemma \ref{IONMOMON} that
\begin{equation*}
\lambda\mathfrak{c}\in
I(\mathcal{O},\,N)\cap\mathcal{M}(\mathcal{O})=\mathcal{M}(\mathcal{O},\,N).
\end{equation*}
Thus we attain the inclusion
\begin{equation*}
\{\lambda\mathfrak{c}~|~\lambda\in P\}\subseteq[\mathfrak{c}]\cap
\mathcal{M}(\mathcal{O},\,N).
\end{equation*}
\par
Now, let $\mathfrak{a}\in[\mathfrak{c}]\cap\mathcal{M}(\mathcal{O},\,N)$.
Since $\mathfrak{a}\in[\mathfrak{c}]$, we have
\begin{equation*}
\mathfrak{a}=\lambda\mathfrak{c}\quad
\textrm{with}~\lambda=\frac{\lambda_1}{\lambda_2}~
\textrm{for some}~
\lambda_1,\,\lambda_2\in\mathcal{O}\setminus\{0\}~\textrm{such that}~
\lambda_1\equiv\lambda_2\equiv1\Mod{N\mathcal{O}}.
\end{equation*}
Therefore we derive from the fact $\mathfrak{a}\subseteq\mathcal{O}$ that
\begin{equation*}
\lambda\in\mathfrak{c}^{-1}=a\overline{\mathfrak{c}}\subseteq\mathcal{O}
\end{equation*}
again by Lemmas \ref{normbasic} (iii) and \ref{QaN}.
Moreover, since
\begin{equation*}
\lambda_2\lambda=\lambda_1\quad\textrm{and}\quad\lambda_1\equiv\lambda_2\equiv1\Mod{N\mathcal{O}},
\end{equation*}
we deduce $\lambda\equiv1\Mod{N\mathcal{O}}$, and hence $\lambda\in P$ and
$\mathfrak{a}=\lambda\mathfrak{c}\in\{
\lambda\mathfrak{c}~|~\lambda\in P\}$.
This proves the converse inclusion
\begin{equation*}
[\mathfrak{c}]\cap\mathcal{M}(\mathcal{O},\,N)
\subseteq \{
\lambda\mathfrak{c}~|~\lambda\in P\}.
\end{equation*}
\item[(ii)] Assume that $\lambda\mathfrak{c}=\mu\mathfrak{c}$.
Then we attain $\lambda\mathcal{O}=\mu\mathcal{O}$ and so
$\mu=\zeta\lambda$ for some $\zeta\in\mathcal{O}^*$.
Furthermore, since $\lambda\equiv\mu\equiv1\Mod{N\mathcal{O}}$, we must have $\zeta\equiv1\Mod{N\mathcal{O}}$.
And, the converse is obvious.
\end{enumerate}
\end{proof}

Let
\begin{equation*}
\gamma_{\mathcal{O},\,N}=
\left|\{\nu\in\mathcal{O}^*~|~\nu\equiv1\Mod{N\mathcal{O}}\}\right|.
\end{equation*}

\begin{proposition}\label{zetaC}
Let $C\in\mathcal{C}_N(\mathcal{O})$ and so
$C=\phi_{\mathcal{O},\,N}([Q])$ for some
$Q=ax^2+bxy+cy^2\in\mathcal{Q}(D,\,N)$ by \textup{Proposition \ref{phibijective}}.
Then we have
\begin{equation*}
\zeta_\mathcal{O}(s,\,C)=\frac{1}{\gamma_{\mathcal{O},\,N}(N^2a)^s}
\sum_{(m,\,n)
\in\mathbb{Z}^2\setminus
\{(0,\,-\frac{a'}{N})\}}\frac{1}{|m(-\overline{\omega}_Q)+n+\frac{a'}{N}|^{2s}}
\end{equation*}
where $a'$ is an integer which satisfies $aa'\equiv1\Mod{N}$.
\end{proposition}
\begin{proof}
Let $\mathfrak{c}=[\omega_Q,\,1]$.
Then we find that
\begin{eqnarray*}
\zeta_\mathcal{O}(s,\,C)&=&\sum_{\mathfrak{a}\in[\mathfrak{c}]\cap
\mathcal{M}(\mathcal{O},\,N)}\frac{1}{\mathrm{N}_\mathcal{O}(\mathfrak{a})^s}\\
&=&\frac{1}{\gamma_{\mathcal{O},\,N}}\sum_{\lambda\in P}\frac{1}{\mathrm{N}_\mathcal{O}(\lambda\mathfrak{c})^s}
\quad\textrm{where}~
P=\{\lambda\in a\overline{\mathfrak{c}}~|~\lambda\neq0~
\textrm{and}~\lambda\equiv1\Mod{N\mathcal{O}}\},\\
&&\hspace{3.6cm}
\textrm{by Lemma \ref{P}}\\
&=&\frac{1}{\gamma_{\mathcal{O},\,N}}\sum_{\lambda\in P}
\left(\frac{\mathrm{N}_\mathcal{O}(a\mathcal{O})}{\mathrm{N}_\mathcal{O}(\lambda\mathcal{O})\mathrm{N}_\mathcal{O}(a\mathfrak{c})}\right)^s\quad\textrm{by Lemmas \ref{normbasic} (ii) and \ref{QaN}}\\
&=&\frac{1}{\gamma_{\mathcal{O},\,N}}\sum_{\lambda\in P}
\left(\frac{a^2}{\mathrm{N}_{K/\mathbb{Q}}(\lambda)\cdot a}\right)^s\quad\textrm{by Lemmas \ref{normbasic} (i) and \ref{QaN}}\\
&=&\frac{a^s}{\gamma_{\mathcal{O},\,N}}\sum_{\tiny\begin{array}{l}(u,\,v)\in\mathbb{Z}^2
\setminus\{(0,\,0)\}\\
\textrm{such that}~
u\equiv0,\,va\equiv1\Mod{N}\end{array}}
\frac{1}{|u(-a\overline{\omega}_Q)+va|^{2s}}\\
&&\hspace{4cm}\textrm{because}~
\lambda\in P,~
a\overline{\mathfrak{c}}=[-a\overline{\omega}_Q,\,a]~\textrm{and}~
\mathcal{O}=[-a\overline{\omega}_Q,\,1]\\
&=&\frac{a^s}{\gamma_{\mathcal{O},\,N}}\sum_{(m,\,n)\in\mathbb{Z}^2\setminus\{(0,\,-\frac{a'}{N})\}}
\frac{1}{|Nm(-a\overline{\omega}_Q)+(Nn+a')a|^{2s}}\\
&=&\frac{1}{\gamma_{\mathcal{O},\,N}(N^2a)^s}\sum_{(m,\,n)\in\mathbb{Z}^2\setminus\{(0,\,-\frac{a'}{N})\}}
\frac{1}{|m(-\overline{\omega}_Q)+n+\frac{a'}{N}|^{2s}}.
\end{eqnarray*}
\end{proof}

\section{Derivatives of $L$-functions at $s=0$}

In this section, we shall
derive a formula of
the derivative $L'_\mathcal{O}(0,\,\chi)$
in terms of ray class invariants for the order $\mathcal{O}$.
\par
By Lemma \ref{normbasic} (ii),
we can extend the norm map $\mathrm{N}_\mathcal{O}:\mathcal{M}(\mathcal{O})\rightarrow\mathbb{Z}$ to
the well-defined function
\begin{equation*}
I(\mathcal{O})\rightarrow\mathbb{Q},\quad\mathfrak{a}\mathfrak{b}^{-1}\mapsto
\mathrm{N}_\mathcal{O}(\mathfrak{a})\mathrm{N}_\mathcal{O}(\mathfrak{b})^{-1}\quad(\mathfrak{a},\,\mathfrak{b}\in\mathcal{M}(\mathcal{O}))
\end{equation*}
and denote it again by $\mathrm{N}_\mathcal{O}$.

\begin{definition}\label{g(C)}
Let $C\in\mathcal{C}_N(\mathcal{O})$.
\begin{enumerate}
\item[\textup{(i)}] Let $N=1$. Take a proper $\mathcal{O}$-ideal $\mathfrak{c}$ in the class $C$ and
define
\begin{equation*}
g_{\mathcal{O},\,N}(C)=g_{\mathcal{O},\,1}(C)=
(2\pi)^{12}\,\mathrm{N}_\mathcal{O}(\mathfrak{c}^{-1})^6|\Delta(\mathfrak{c}^{-1})|.
\end{equation*}
\item[\textup{(ii)}] If $N\geq2$, then we define
\begin{equation*}
g_{\mathcal{O},\,N}(C)=
g_{\left[\begin{smallmatrix}0&\frac{1}{N}\end{smallmatrix}\right]}^{12N}(C).
\end{equation*}
\end{enumerate}
\end{definition}

\begin{remark}\label{invremark}
\begin{enumerate}
\item[(i)]
Let $\{\xi_1,\,\xi_2\}$ be a $\mathbb{Z}$-basis for $\mathfrak{c}^{-1}$ such that
\begin{equation*}
\xi:=\frac{\xi_1}{\xi_2}\in\mathbb{H}.
\end{equation*}
By using Lemma \ref{normbasic} (i) and the well-known fact that the function
\begin{equation*}
\mathbb{H} \rightarrow \mathbb{C},\quad
\tau \mapsto \Delta([\tau,\,1])
\end{equation*}
is a modular form for $\mathrm{SL}_2(\mathbb{Z})$ of weight $12$
(\cite[Theorem 3 in Chapter 3 and Theorem 5 in Chapter 18]{Lang87}), we find that
\begin{equation}\label{2Nd}
(2\pi)^{12}\,\mathrm{N}_\mathcal{O}(\mathfrak{c}^{-1})^6|\Delta(\mathfrak{c}^{-1})|=(2\pi)^{12}\,\mathrm{N}_\mathcal{O}([\xi,\,1])^6|\Delta([\xi,\,1])|.
\end{equation}
And one can readily check that this value depends only on the class $C$, not on the special choice of $\mathfrak{c}$.
\item[(ii)] Since the modular form $\Delta$ and the Siegel function
$g_{\left[\begin{smallmatrix}0&\frac{1}{N}\end{smallmatrix}\right]}^{12N}$
have no zeros and poles on $\mathbb{H}$, the invariant $g_{\mathcal{O},\,N}(C)$
is always finite and nonzero.
\item[(iii)]
If $N\geq2$, then $g_{\mathcal{O},\,N}(C)=g_{\left[\begin{smallmatrix}0&\frac{1}{N}\end{smallmatrix}\right]}^{12N}(C)$ belongs to $K_{\mathcal{O},\,N}$ and satisfies
\begin{equation*}
g_{\mathcal{O},\,N}(C)^{\sigma_{\mathcal{O},\,N}(C')}=
g_{\left[\begin{smallmatrix}0&\frac{1}{N}\end{smallmatrix}\right]}^{12N}(C)^{\sigma_{\mathcal{O},\,N}(C')}
=g_{\left[\begin{smallmatrix}0&\frac{1}{N}\end{smallmatrix}\right]}^{12N}(CC')
=g_{\mathcal{O},\,N}(CC')
\end{equation*}
by Theorem \ref{transformation}.
\item[(iv)] Recently, Jung and Kim showed that
if $D_\mathcal{O}\neq-3,\,-4$ and $N\geq2$, then $g_{\mathcal{O},\,N}(C_0)^n$ generates $K_{\mathcal{O},\,N}$ over $K$ for any nonzero integer $n$ (\cite[Theorem 1.1]{J-K}).
\end{enumerate}
\end{remark}

For a pair $(\omega,\,z)\in\mathbb{C}\times
\mathbb{H}$, we define the $\xi$-function $\xi(\,\cdot\,,\,\omega,\,z)$ by
\begin{equation*}
\xi(s,\,\omega,\,z)=\sum_{\tiny\begin{array}{l}(m,\,n)
\in\mathbb{Z}^2~\textrm{such that}\\
mz+n+\omega\neq0
\end{array}}\frac{1}{
|mz+n+\omega|^{2s}}\quad(s\in\mathbb{C},~\mathrm{Re}(s)>1).
\end{equation*}
Put
\begin{eqnarray}
\eta(z)&=&e^{\frac{\pi\mathrm{i}z}{12}}\prod_{n=1}^\infty
\left(1-e^
{2\pi\mathrm{i}nz}\right),\nonumber\\
\vartheta_1(\omega,\,z)&=&2e^{\frac{\pi\mathrm{i}z}{6}}
(\sin\pi\omega)\eta(z)\prod_{n=1}^\infty
\left(1-e^{2\pi\mathrm{i}(nz+\omega)}\right)
\left(1-e^{2\pi\mathrm{i}(nz-\omega)}\right).\label{theta}
\end{eqnarray}

\begin{proposition}[Kronecker's limit formula]\label{Kronecker}
The $\xi$-function satisfies the following properties.
\begin{enumerate}
\item[\textup{(i)}] It has an analytic continuation
on the whole complex plane and
\begin{equation*}
\xi(0,\,\omega,\,z)=
\left\{
\begin{array}{rl}
-1 & \textrm{if}~\omega\in[z,\,1],\\
0 & \textrm{if}~\omega\not\in[z,\,1].
\end{array}\right.
\end{equation*}
\item[\textup{(ii)}] If we let  $\xi'=\displaystyle\frac{d\xi}{ds}$, then
\begin{equation*}
\xi'(0,\,\omega,\,z)=
\left\{\begin{array}{ll}
-\ln\left|4\pi^2\eta(z)^4\right| & \textrm{if}~\omega\in[z,\,1],\\
-\ln
\left|\frac{\vartheta_1(\omega,\,z)}{\eta(z)}\,
e^{\frac{\pi\mathrm{i}\omega(\omega-\overline{\omega})}
{z-\overline{z}}}
\right|^2 & \textrm{if}~\omega\not\in[z,\,1].
\end{array}\right.
\end{equation*}
\end{enumerate}
\end{proposition}
\begin{proof}
See \cite[Theorem 4]{Berndt}, \cite{Shintani} or \cite{Stark}.
\end{proof}

\begin{theorem}\label{derivative}
If $\chi$ is a character of
$\mathcal{C}_N(\mathcal{O})$, then we have
\begin{equation*}
L_\mathcal{O}'(0,\,\chi)=
\displaystyle-\frac{1}{\gamma_{\mathcal{O},\,N}6N}
\sum_{C\in\mathcal{C}_N(\mathcal{O})}\chi(C)
\ln\left|g_{\mathcal{O},\,N}(C)\right|.
\end{equation*}
\end{theorem}
\begin{proof}
Let $C\in\mathcal{C}_N$ and so $C=\phi_{\mathcal{O},\,N}([Q])$ for some
$Q=ax^2+bxy+cy^2$. Let $a'$ be an integer satisfying $aa'\equiv1\Mod{N}$.
Since
\begin{equation*}
\zeta_\mathcal{O}(s,\,C)=
\frac{1}{\gamma_{\mathcal{O},\,N}(N^2a)^s}
\,\xi(s,\,\tfrac{a'}{N},\,-\overline{\omega}_Q)
\end{equation*}
by Proposition \ref{zetaC},
we find that
\begin{eqnarray*}
\zeta_\mathcal{O}'(0,\,C)&=&
\frac{1}{\gamma_{\mathcal{O},\,N}}\left(-\ln(N^2a)\xi(0,\,\tfrac{a'}{N},\,-\overline{\omega}_Q)
+\xi'(0,\,\tfrac{a'}{N},\,-\overline{\omega}_Q)\right)\quad\textrm{by Proposition \ref{Kronecker} (i)}\\
&=&\frac{1}{\gamma_{\mathcal{O},\,N}}\times
\left\{\begin{array}{ll}
\ln a-\ln\left|4\pi^2\eta(-\overline{\omega}_Q)^4\right| & \textrm{if}~N=1,\\
-\ln
\left|\frac{\vartheta_1(\tfrac{a'}{N},\,-\overline{\omega}_Q)}{\eta(-\overline{\omega}_Q)}
\right|^2 & \textrm{if}~N\geq2,
\end{array}\right.\quad\textrm{by Proposition \ref{Kronecker} (ii)}\\
&=&\frac{1}{\gamma_{\mathcal{O},\,N}}\times\left\{\begin{array}{ll}
-\ln\left|4\pi^2\mathrm{N}_\mathcal{O}([-\overline{\omega}_Q,\,1])\eta(-\overline{\omega}_Q)^4\right| & \textrm{if}~N=1,\\
-\ln\left|g_{\left[\begin{smallmatrix}0&\frac{a'}{N}
\end{smallmatrix}\right]}(-\overline{\omega}_Q)\right|^2 & \textrm{if}~N\geq2,
\end{array}\right.\\
&&\hspace{6cm}
\textrm{by Lemma \ref{QaN} and Definitions (\ref{infiniteproduct}), (\ref{theta})}\\
&=&\frac{1}{\gamma_{\mathcal{O},\,N}}\times\left\{\begin{array}{ll}
-\frac{1}{6}\ln|g_{\mathcal{O},\,1}(C)| & \textrm{if}~N=1,\\
-\frac{1}{6N}\ln|g_{\mathcal{O},\,N}(C)| & \textrm{if}~N\geq2,
\end{array}\right.\quad\textrm{by Definition \ref{g(C)} and (\ref{2Nd})}.
\end{eqnarray*}
Therefore we conclude that
\begin{equation*}
L_\mathcal{O}'(0,\,\chi)=
\sum_{C\in\mathcal{C}_N(\mathcal{O})}\chi(C)\zeta_\mathcal{O}'(0,\,C)=
-\frac{1}{\gamma_{\mathcal{O},\,N}6N}\sum_{C\in\mathcal{C}_N(\mathcal{O})}
\chi(C)
\ln|g_{\mathcal{O},\,N}(C)|.
\end{equation*}
\end{proof}

\section {An example of $\mathcal{C}_N(D_\mathcal{O})$}

We shall present an example of $\mathcal{C}_N(D_\mathcal{O})$
and its application further in order to find the minimal polynomial of the invariant $g_{\mathcal{O},\,N}(C_0)$ over $K$.
\par
Let $K=\mathbb{Q}(\sqrt{-2})$ and $\mathcal{O}=[5\sqrt{-2},\,1]$.
Then we get
\begin{equation*}
D_\mathcal{O}=-200,\quad\ell_\mathcal{O}=5,\quad
\tau_\mathcal{O}=5\sqrt{-2},\quad
\mathrm{min}(\tau_\mathcal{O},\,\mathbb{Q})=x^2+b_\mathcal{O}x+c_\mathcal{O}=x^2+50.
\end{equation*}
On the other hand, there are six reduced forms of discriminant $D_\mathcal{O}=-200$, namely,
\begin{equation*}
\begin{array}{lll}
Q_1=x^2+50y^2, & Q_2=2x^2+25y^2,\\
Q_3=3x^2-2xy+17y^2, &
Q_4=3x^2+2xy+17y^2,\\
Q_5=6x^2-4xy+9y^2, & Q_6=6x^2+4xy+9y^2.
\end{array}
\end{equation*}
Let $N=3$, and set
\begin{equation*}
\begin{array}{ll}
\widetilde{Q}_1=Q_1=x^2+50y^2,&\widetilde{Q}_2=Q_2=2x^2+25y^2,\vspace{0.1cm}\\
\widetilde{Q}_3=Q_3^{\left[\begin{smallmatrix}0&-1\\1&\phantom{-}0 \end{smallmatrix}\right]}=17x^2+2xy+3y^2,&
\widetilde{Q}_4=Q_4^{\left[\begin{smallmatrix}0&-1\\1&\phantom{-}0 \end{smallmatrix}\right]}=17x^2-2xy+3y^2,\vspace{0.1cm}\\
\widetilde{Q}_5=Q_5^{\left[\begin{smallmatrix}1&-1\\1&\phantom{-}0 \end{smallmatrix}\right]}=11x^2-8xy+6y^2,&
\widetilde{Q}_6=Q_6^{\left[\begin{smallmatrix}\phantom{-}1&1\\-1&0 \end{smallmatrix}\right]}=11x^2+8xy+6y^2\\
\end{array}
\end{equation*}
so as to have $\widetilde{Q}_i\in \mathcal{Q}(D_\mathcal{O},\,N)$ ($i=1,\,2,\,\ldots,\,6$).
By utilizing the fact
\begin{equation*}
W_{\mathcal{O},\,N}/U_{\mathcal{O},\,N}
\simeq\mathrm{Gal}(K_{\mathcal{O},\,N}/H_\mathcal{O})\simeq
P(\mathcal{O},\,N)/P_N(\mathcal{O})
\simeq
P(\mathcal{O},\,\ell_\mathcal{O}N)/P_N(\mathcal{O},\,\ell_\mathcal{O}N)
\end{equation*}
((\ref{WUOG}) and Lemma \ref{PMNPN})
and Lemma \ref{primelemma}, we attain
\begin{eqnarray*}
P(\mathcal{O},\,\ell_\mathcal{O}N)/P_N(\mathcal{O},\,\ell_\mathcal{O}N)=
\left\{
[\mathcal{O}],\,[\tau_\mathcal{O}\mathcal{O}]=[[\tfrac{\sqrt{-2}}{10},\,1]]\right\},
\end{eqnarray*}
which corresponds to the subgroup $\{[\widetilde{Q}_1],\,[50x^2+y^2]\}$
of $\mathcal{C}_N(D_\mathcal{O})=\mathcal{C}_3(-200)$.
In the group $\mathcal{C}_N(D_\mathcal{O})$, we let
\begin{equation*}
g_i=[\widetilde{Q}_i]\quad\textrm{and}
\quad
g_{i+6}=g_i\cdot[50x^2+y^2]\quad
(i=1,\,2,\,\ldots,\,6).
\end{equation*}
By adopting Definition \ref{binaryoperation} one can readily 
find that
$g_j=[\widetilde{Q}_j]$ ($j=7,\,8,\,\ldots,\,12$) with
\begin{equation*}
\begin{array}{ll}
\widetilde{Q}_7=50x^2+y^2, & \widetilde{Q}_8=25x^2+2y^2,\\
\widetilde{Q}_9=22x^2-36xy+17y^2, & \widetilde{Q}_{10}=22x^2+36xy+17y^2,\\
\widetilde{Q}_{11}=25x^2+30xy+11y^2, & \widetilde{Q}_{12}=25x^2-30xy+11y^2.
\end{array}
\end{equation*}
The group table of $\mathcal{C}_N(D_\mathcal{O})$ is given as follows. 

\begin{table}[h]
\centering
\caption*{The group table of $\mathcal{C}_3(-200)$}
\label{table}
\begin{tabular}{c|cccccccccccc}
&$g_1$&$g_2$&$g_3$&$g_4$&$g_5$&$g_6$&$g_7$&$g_8$&$g_9$&$g_{10}$&$g_{11}$&$g_{12}$\\
\hline
$g_1$&$g_1$&$g_2$&$g_3$&$g_4$&$g_5$&$g_6$&$g_7$&$g_8$&$g_9$&$g_{10}$&$g_{11}$&$g_{12}$\\
$g_2$&$g_{2}$&$g_{1}$&$g_{12}$&$g_{11}$&$g_{10}$&$g_{9}$&$g_{8}$&$g_{7}$&$g_{6}$&$g_{5}$&$g_{4}$&$g_{3}$\\
$g_3$&$g_{3}$&$g_{12}$&$g_{11}$&$g_{1}$&$g_{8}$&$g_{10}$&$g_{9}$&$g_{6}$&$g_{5}$&$g_{7}$&$g_{2}$&$g_{4}$\\
$g_4$&$g_{4}$&$g_{11}$&$g_{1}$&$g_{12}$&$g_{9}$&$g_{8}$&$g_{10}$&$g_{5}$&$g_{7}$&$g_{6}$&$g_{3}$&$g_{2}$\\
$g_5$&$g_{5}$&$g_{10}$&$g_{8}$&$g_{9}$&$g_{12}$&$g_{1}$&$g_{11}$&$g_{4}$&$g_{2}$&$g_{3}$&$g_{6}$&$g_{7}$\\
$g_6$&$g_{6}$&$g_{9}$&$g_{10}$&$g_{8}$&$g_{1}$&$g_{11}$&$g_{12}$&$g_{3}$&$g_{4}$&$g_{2}$&$g_{7}$&$g_{5}$\\
$g_7$&$g_{7}$&$g_{8}$&$g_{9}$&$g_{10}$&$g_{11}$&$g_{12}$&$g_{1}$&$g_{2}$&$g_{3}$&$g_{4}$&$g_{5}$&$g_{6}$\\
$g_8$&$g_{8}$&$g_{7}$&$g_{6}$&$g_{5}$&$g_{4}$&$g_{3}$&$g_{2}$&$g_{1}$&$g_{12}$&$g_{11}$&$g_{10}$&$g_{9}$\\
$g_9$&$g_{9}$&$g_{6}$&$g_{5}$&$g_{7}$&$g_{2}$&$g_{4}$&$g_{3}$&$g_{12}$&$g_{11}$&$g_{1}$&$g_{8}$&$g_{10}$\\
$g_{10}$&$g_{10}$&$g_{5}$&$g_{7}$&$g_{6}$&$g_{3}$&$g_{2}$&$g_{4}$&$g_{11}$&$g_{1}$&$g_{12}$&$g_{9}$&$g_{8}$\\
$g_{11}$&$g_{11}$&$g_{4}$&$g_{2}$&$g_{3}$&$g_{6}$&$g_{7}$&$g_{5}$&$g_{10}$&$g_{8}$&$g_{9}$&$g_{12}$&$g_{1}$\\
$g_{12}$&$g_{12}$&$g_{3}$&$g_{4}$&$g_{2}$&$g_{7}$&$g_{5}$&$g_{6}$&$g_{9}$&$g_{10}$&$g_{8}$&$g_{1}$&$g_{11}$\\
\end{tabular}
\end{table}
Since $\mathcal{C}_N(D_\mathcal{O})$ has three elements $g_2$, $g_7$, $g_8$ of order $2$, 
it is isomorphic to $\mathbb{Z}_2\times\mathbb{Z}_6$.
\par
Write $\widetilde{Q}_i=a_ix^2+b_ixy+c_iy^2$ ($i=1,\,2,\,\ldots,\,12$). We then achieve that
\begin{eqnarray*}
\mathrm{min}(g_{\mathcal{O},\,N}(C_0),\,K)
&=&\displaystyle\prod_{C\in\mathcal{C}_N(\mathcal{O})}\left(x-
g_{\left[\begin{smallmatrix}0&\frac{1}{N}\end{smallmatrix}\right]}^{12N}(C)\right)\\
&&\hspace{2cm}\textrm{by Definition \ref{g(C)} (ii) and Remark \ref{invremark} (iii), (iv)}\\
&=&\displaystyle\prod_{i=1}^{12}\left(x-
\left(g_{\left[\begin{smallmatrix}0&\frac{1}{N}\end{smallmatrix}\right]}^{12N}\right)
^{\left[
\begin{smallmatrix}1&-a_i'(b_i+b_\mathcal{O})/2\\0&a_i'\end{smallmatrix}
\right]}
(-\overline{\omega}_{\widetilde{Q}_i})
\right)\\
&&\hspace{2cm}\textrm{where $a_i'$ is an integer such that $a_ia_i'\equiv1\Mod{N}$}\\
&&\hspace{2cm}\textrm{by Definition \ref{DeffQ} and Lemma \ref{fmatrix}}\\
&=&\displaystyle\prod_{i=1}^{12}\left(x-
g_{\left[\begin{smallmatrix}0&\frac{a_i'}{N}\end{smallmatrix}\right]}^{12N}
(-\overline{\omega}_{\widetilde{Q}_i})
\right)\quad\textrm{by Proposition \ref{Frickefamily}}.
\end{eqnarray*}
By making use of the definition (\ref{infiniteproduct}), one can numerically estimate 
$\mathrm{min}(g_{\mathcal{O},\,3}(C_0),\,K)$  as
\begin{equation*}
\begin{array}{rcl}
&&x^{12}
-19732842623587344380 x^{11}
+85622274889372918445313749346 x^{10}\vspace{0.1cm}\\
&&+583422788794106041510392970996250100 x^9\vspace{0.1cm}\\
&&+2412956602599045666947505580865471555967855 x^8\vspace{0.1cm}\\
&&+4622030004758636935674310042187173142345125210120 x^7\vspace{0.1cm}\\
&&+5159683938264742220691719229969015883694331066838711900 x^6\vspace{0.1cm}\\
&&+202375300752001975403428909178152428797277946213173155269640 x^5\vspace{0.1cm}\\
&&+2017771307025673942770713882875344826204880202806909292959103855 x^4\vspace{0.1cm}\\
&&-2883328681523953153105049905288236082276160171409937896788594572300 x^3\vspace{0.1cm}\\
&&+4487601627619641192200184812721309459195966653602482165478526149968226 x^2\vspace{0.1cm}\\
&&-19833699482405442556441925074783039534555541722620 x+1.
\end{array}
\end{equation*}
This tells us that the invariant $g_{\mathcal{O},\,3}(C_0)$ is in fact a unit as algebraic integer. 

\section*{Statements \& Declarations}

\subsection*{Funding}
The first named author was supported by the National Research Foundation of Korea (NRF) grant funded by the Korea government (MSIT) (No. RS-2023-00252986).
The third named (corresponding) author was supported
by Hankuk University of Foreign Studies Research Fund of 2023 and 
by the National Research Foundation of Korea (NRF) grant funded by the Korea government (MSIT) (No. RS-2023-00241953).

\subsection*{Competing Interests}
The authors have not disclosed any competing interests.

\subsection*{Data availability}
Data sharing is not applicable to this article as no datasets
were generated or analysed during the current study.

\bibliographystyle{amsplain}

\address{
Department of Mathematics\\
Dankook University\\
Cheonan-si, Chungnam 31116\\
Republic of Korea} {hoyunjung@dankook.ac.kr}
\address{
Department of Mathematical Sciences \\
KAIST \\
Daejeon 34141\\
Republic of Korea} {jkgoo@kaist.ac.kr}
\address{
Department of Mathematics\\
Hankuk University of Foreign Studies\\
Yongin-si, Gyeonggi-do 17035\\
Republic of Korea} {dhshin@hufs.ac.kr}
\address{
Department of Mathematics Education\\
Pusan National University\\
Busan 46241\\Republic of Korea}
{dsyoon@pusan.ac.kr}

\end{document}